\documentclass[a4paper,10pt,reqno]{amsart}
\usepackage{microtype}

\usepackage{mathrsfs}

\usepackage{MnSymbol}

\usepackage{tikz}
\usepackage{mhequ}
\usepackage{mhsymb}

\usepackage[colorlinks=true, pdfstartview=FitV, linkcolor=black, citecolor=black, urlcolor=black,pagebackref=false]{hyperref}

\let\emptyset \undefined
\newsymbol\emptyset    203F

\definecolor{darkred}{rgb}{0.9,0.1,0.1}
\definecolor{darkblue}{rgb}{0,0,0.7}
\definecolor{darkgreen}{rgb}{0,0.5,0}

\def\refcomment#1{}
\newcommand{\rf}[1]{{#1}}
\theoremstyle{plain}
\newtheorem{theorem}{Theorem}[section]
\newtheorem{mtheorem}[theorem]{Meta-Theorem}

\newtheorem{lemma}[theorem]{Lemma}

\newtheorem{definition}[theorem]{Definition}

\theoremstyle{remark}
\newtheorem{remark}[theorem]{Remark}
\newtheorem{example}[theorem]{Example}

\numberwithin{equation}{section}

\newcommand{\E}{\mathbb{E}}
\renewcommand{\P}{{\mathbb P}}
\newcommand{\T}{\symb{T}}

\newcommand{\KK}{\mathfrak{K}}

\newcommand{\Cc}{\mathcal{C}}
\newcommand{\Cs}{\mathcal{C}_{\mathfrak{s}}}

\newcommand{\Ff}{\mathcal{F}}

\newcommand{\al}{\alpha}
\newcommand{\be}{\beta}
\newcommand{\ga}{\gamma}
\newcommand{\de}{\delta}
\newcommand{\ka}{\kappa}
\newcommand{\la}{\lambda}

\newcommand{\si}{\sigma}

\newcommand{\s}{\mathfrak{s}}
\newcommand{\ls}{\lesssim}

\renewcommand{\subset}{\subseteq}

\newcommand{\ag}{\al}
\newcommand{\bg}{\be}
\newcommand{\dg}{\de}

\def\CD{\mathcal{D}}

\def\CS{\mathcal{S}}

\def\CW{\mathcal{W}}

\def\CB{\mathcal{B}}

\def\one{\mathbf{1}}
\let\1\one

\def\TT{\mathcal{T}}

\usetikzlibrary{shapes.misc}
\usetikzlibrary{shapes.symbols}
\usetikzlibrary{snakes}
\usetikzlibrary{decorations}
\usetikzlibrary{decorations.markings}

\def\drawx{\draw[-,solid] (-3pt,-3pt) -- (3pt,3pt);\draw[-,solid] (-3pt,3pt) -- (3pt,-3pt);}
\tikzset{
	root/.style={circle,fill=testcolor,inner sep=0pt, minimum size=2mm},
	dot/.style={circle,fill=black,draw=black, solid,inner sep=0pt,minimum size=0.5mm},
	square/.style={rectangle,fill=black,draw=black, solid,inner sep=0pt,minimum size=1mm},
	empty/.style={circle,fill=white,draw=white, solid,inner sep=0pt,minimum size=0.5mm},
	var/.style={circle,fill=black!10,draw=black,inner sep=0pt, minimum size=
	2mm},
	symb/.style={circle,fill=symbols,draw=symbols, solid,inner sep=0pt,minimum size=0.5mm},
	yy/.style={circle,fill=gray!20,draw=black,inner sep=0pt,minimum size=0.8mm},
	>=stealth,
	dotred/.style={circle,fill=black!50,inner sep=0pt, minimum size=2mm},
	generic/.style={semithick,shorten >=1pt,shorten <=1pt},
	dist/.style={ultra thick,draw=testcolor,shorten >=1pt,shorten <=1pt},
	testfcn/.style={ultra thick,testcolor,shorten >=1pt,shorten <=1pt,<-},
	testfcnx/.style={ultra thick,testcolor,shorten >=1pt,shorten <=1pt,<-,
		postaction={decorate,decoration={markings,mark=at position 0.6 with {\drawx}}}},
	kprime/.style={semithick,shorten >=1pt,shorten <=1pt,densely dashed,->},
	kprimex/.style={semithick,shorten >=1pt,shorten <=1pt,densely dashed,->,
		postaction={decorate,decoration={markings,mark=at position 0.4 with {\drawx}}}},
	kernel/.style={semithick,shorten >=1pt,shorten <=1pt,->},
	multx/.style={shorten >=1pt,shorten <=1pt,
		postaction={decorate,decoration={markings,mark=at position 0.5 with {\drawx}}}},
	kernelx/.style={semithick,shorten >=1pt,shorten <=1pt,->,
		postaction={decorate,decoration={markings,mark=at position 0.4 with {\drawx}}}},
	kernel1/.style={->,semithick,shorten >=1pt,shorten <=1pt,postaction={decorate,decoration={markings,mark=at position 0.45 with {\draw[-] (0,-0.1) -- (0,0.1);}}}},
	kernel2/.style={->,semithick,shorten >=1pt,shorten <=1pt,postaction={decorate,decoration={markings,mark=at position 0.45 with {\draw[-] (0.05,-0.1) -- (0.05,0.1);\draw[-] (-0.05,-0.1) -- (-0.05,0.1);}}}},
	kernelBig/.style={semithick,shorten >=1pt,shorten <=1pt,decorate, decoration={zigzag,amplitude=1.5pt,segment length = 3pt,pre length=2pt,post length=2pt}},
	rho/.style={dotted,semithick,shorten >=1pt,shorten <=1pt},
	renorm/.style={shape=circle,fill=white,inner sep=1pt},
	labl/.style={shape=rectangle,fill=white,inner sep=1pt},
	xi/.style={circle,fill=symbols!10,draw=symbols,inner sep=0pt,minimum size=1.2mm},
	xix/.style={crosscircle,fill=symbols!10,draw=symbols,inner sep=0pt,minimum size=1.2mm},
	xib/.style={circle,fill=symbols!10,draw=symbols,inner sep=0pt,minimum size=1.6mm},
	xibx/.style={crosscircle,fill=symbols!10,draw=symbols,inner sep=0pt,minimum size=1.6mm},
	not/.style={circle,fill=symbols,draw=symbols,inner sep=0pt,minimum size=0.5mm},
	>=stealth,
	}

\colorlet{symbols}{blue!90!black}
\def\symbol#1{\textcolor{symbols}{#1}}

\makeatletter
\def\DeclareSymbol#1#2#3{\expandafter\gdef\csname MH@symb@#1\endcsname{\tikz[baseline=#2,scale=0.15]{#3}}%
\expandafter\gdef\csname MH@symb@#1s\endcsname{\scalebox{0.6}{\tikz[baseline=#2,scale=0.15]{#3}}}}
\def\<#1>{\csname MH@symb@#1\endcsname}
\makeatother

\DeclareSymbol{X}{-2.4}{\node[dot] {};}
\DeclareSymbol{1}{0}{\draw[white] (-.4,0) -- (.4,0); \draw (0,0)  -- (0,1.2) node[dot] {};}
\DeclareSymbol{2}{0}{\draw (-0.5,1.2) node[dot] {} -- (0,0) -- (0.5,1.2) node[dot] {};}
\DeclareSymbol{3}{0}{\draw (0,0) -- (0,1.2) node[dot] {}; \draw (-.7,1) node[dot] {} -- (0,0) -- (.7,1) node[dot] {};}
\DeclareSymbol{31}{-3}{\draw (0,0) -- (0,-1) -- (1,0) node[dot] {}; \draw (0,0) -- (0,1.2) node[dot] {}; \draw (-.7,1) node[dot] {} -- (0,0) -- (.7,1) node[dot] {};}
\DeclareSymbol{30}{-3}{\draw (0,0) -- (0,-1); \draw (0,0) -- (0,1.2) node[dot] {}; \draw (-.7,1) node[dot] {} -- (0,0) -- (.7,1) node[dot] {};}
\DeclareSymbol{10}{-3}{\draw (0,0) -- (0,-1); \draw (-.8,1) node[dot] {} -- (0,0);}
\DeclareSymbol{32}{-3}{\draw (0,0) -- (0,-1) -- (1,0) node[dot] {}; \draw (0,0) -- (0,-1) -- (-1,0) node[dot] {}; \draw (0,0) -- (0,1.2) node[dot] {}; \draw (-.7,1) node[dot] {} -- (0,0) -- (.7,1) node[dot] {};}
\DeclareSymbol{12}{-3}{\draw (0,0) -- (0,-1) -- (1,0) node[dot] {}; \draw (0,0) -- (0,-1) -- (-1,0) node[dot] {}; \draw (-.8,1) node[dot] {} -- (0,0);}
\DeclareSymbol{22}{-3}{\draw (0,0.3) -- (0,-1) -- (1,0) node[dot] {}; \draw (0,0.3) -- (0,-1) -- (-1,0) node[dot] {};\draw (-.7,1) node[dot] {} -- (0,0.3) -- (.7,1) node[dot] {};}
\DeclareSymbol{20}{-3}{\draw (0,0) -- (0,-1);\draw (-.7,1) node[dot] {} -- (0,0) -- (.7,1) node[dot] {};}
\DeclareSymbol{32W}{-3}{\draw  [densely dotted, semithick] (-1.5,0)node[dot] {}  -- (0,-1.5) node[yy]{} -- (1.5,0) node[dot] {}; \draw (.8,-2.3)node{\tiny $y$} ;\draw[semithick] (0,0) -- (0,-1.5) node[yy] {}; \draw [densely dotted, semithick] (0,0) -- (0,1.8) node[dot] {}; \draw[densely dotted, semithick] (-1,1.5) node[dot] {} -- (0,0) -- (1,1.5) node[dot] {};}
\DeclareSymbol{32WW}{-3}{\draw  [densely dotted, semithick] (-1.5,0)node[dot] {}  -- (0,-1.5) node[yy]{} -- (1.5,0) node[dot] {}; \draw (.8,-2.3)node{\tiny $\bar{y}$} ;\draw[semithick] (0,0) -- (0,-1.5) node[yy] {}; \draw [densely dotted, semithick] (0,0) -- (0,1.8) node[dot] {}; \draw[densely dotted, semithick] (-1,1.5) node[dot] {} -- (0,0) -- (1,1.5) node[dot] {};}

\DeclareSymbol{s1}{0}{\draw[white] (-.4,0) -- (.4,0); \draw[symbols] (0,0)  -- (0,1.2) node[symb] {};}
\DeclareSymbol{s2}{0}{\draw[symbols] (-0.5,1.2) node[symb] {} -- (0,0) -- (0.5,1.2) node[symb] {};}
\DeclareSymbol{s3}{0}{\draw[symbols] (0,0) -- (0,1.2) node[symb] {}; \draw[symbols] (-.7,1) node[symb] {} -- (0,0) -- (.7,1) node[symb] {};}
\DeclareSymbol{s31}{-3}{\draw[symbols] (0,0) -- (0,-1) -- (1,0) node[symb] {}; \draw[symbols] (0,0) -- (0,1.2) node[symb] {}; \draw[symbols] (-.7,1) node[symb] {} -- (0,0) -- (.7,1) node[symb] {};}
\DeclareSymbol{s30}{-3}{\draw[symbols] (0,0) -- (0,-1); \draw[symbols] (0,0) -- (0,1.2) node[symb] {}; \draw[symbols] (-.7,1) node[symb] {} -- (0,0) -- (.7,1) node[symb] {};}
\DeclareSymbol{s32}{-3}{\draw[symbols] (0,0) -- (0,-1) -- (1,0) node[symb] {}; \draw[symbols] (0,0) -- (0,-1) -- (-1,0) node[symb] {}; \draw[symbols] (0,0) -- (0,1.2) node[symb] {}; \draw[symbols] (-.7,1) node[symb] {} -- (0,0) -- (.7,1) node[symb] {};}
\DeclareSymbol{s22}{-3}{\draw[symbols] (0,0.3) -- (0,-1) -- (1,0) node[symb] {}; \draw[symbols] (0,0.3) -- (0,-1) -- (-1,0) node[symb] {};\draw[symbols] (-.7,1) node[symb] {} -- (0,0.3) -- (.7,1) node[symb] {};}
\DeclareSymbol{s20}{-3}{\draw[symbols] (0,0) -- (0,-1);\draw[symbols] (-.7,1) node[symb] {} -- (0,0) -- (.7,1) node[symb] {};}
\DeclareSymbol{s21}{-3}{\draw[symbols] (0,0) -- (0,-1);\draw[symbols] (-.7,1) node[symb] {} -- (0,0) -- (.7,1) node[symb] {}; \draw[symbols] (0,0) -- (0,-1) -- (1,0) node[symb] {};}
\DeclareSymbol{s12}{-3}{\draw[symbols] (0,0) -- (0,-1) -- (1,0) node[symb] {}; \draw[symbols] (0,0) -- (0,-1) -- (-1,0) node[symb] {}; \draw (-.8,1) node[symb] {} -- (0,0);}
\DeclareSymbol{s11}{-3}{\draw[symbols] (0,0) -- (0,-1) -- (-1,0) node[symb] {}; \draw (-.8,1) node[symb] {} -- (0,0);}
\DeclareSymbol{s10}{-3}{\draw[symbols] (0,0) -- (0,-1); \draw[symbols] (-.8,1) node[symb] {} -- (0,0);}

\renewcommand{\star}{\ast}                              % Convolution 
\newcommand{\Xk}{\symbol{\mathbf{X}^k}}
\newcommand{\Xj}{\symbol{\mathbf{X}^j}}
\newcommand{\KKg}{\mathcal{K}_\gamma}
\newcommand{\Ng}{\mathcal{N}_\gamma}

\renewcommand{\eg}{\eps}                                  
\newcommand{\hg}{h_{\ga}}                                  %Rescaled field 
\newcommand{\kg}{\ka_\ga}                                  % Interaction kernel
\newcommand{\Kg}{K_\ga}                                    % Interaction kernel in macroscopic coordinates
\newcommand{\LN}{\Lambda_N}                         % discrete space microscopic coordinates
\newcommand{\Le}{\Lambda_\eps}                     % discrete space macroscopic coordinates
\newcommand{\SN}{\Sigma_N}                            % Spin configurations
\newcommand{\Hg}{\mathscr{H}_{\ga}}               % Hamiltonian
\newcommand{\lbg}{\lambda_{\ga}}                     % Gibbs measure
\newcommand{\Zbg}{\mathscr{Z}_{\ga}}               % Partition function
\newcommand{\LgN}{\mathscr{L}_{\ga}}               % Generator
\newcommand{\cg}{c_{\ga}}                                     % Jump rate microscopic coordinates
\newcommand{\Cg}{C_\ga}                                     % Jump rate macroscopic coordinates
\newcommand{\Mg}{M_\ga}                                     % noise in macroscopic coordinates
\renewcommand{\ae}{\ast}                              % Convolution on macroscopic discrete torus
\newcommand{\Eg}{E_\ga}                                       %Error term appearing in the equation due to expansion of tanh
\newcommand{\Xg}{X_{\ga}}                                     %The rescaled field - our main object

\newcommand{\ct}{c_{\ga}}                                     %Constant close to 1, due to approximation of integral not exactly =1 

\begin{document}
\begin{abstract}
These lecture notes grew out of a series of lectures given by the second named author in short courses in Toulouse, Matsumoto, and Darmstadt. 
The main aim is to explain some aspects of the theory of ``Regularity structures'' developed recently by Hairer in \cite{RegStr}. This theory gives a way to study well-posedness for a class of stochastic PDEs that could not be treated previously. Prominent examples include the KPZ equation as well as the dynamic $\Phi^4_3$ model.  

Such equations can be expanded into formal perturbative expansions. Roughly speaking the theory of regularity structures provides a way to truncate this expansion after finitely many terms and to solve a fixed point problem for the ``remainder''. The key ingredient is a new notion of ``regularity'' which is based on the terms of this expansion.\\ 

\noindent\textsc{R\'esum\'e.} Ces notes sont bas\'ees sur trois cours que le deuxi\`eme auteur a donn\'es \`a Toulouse, Matsumoto et Darmstadt. 
L'objectif principal est d'expliquer certains aspects de la th\'eorie des ``structures de r\'egularit\'e'' d\'evelopp\'ee r\'ecemment
par Hairer \cite{RegStr}. Cette th\'eorie permet de montrer que certaines EDP
stochastiques, qui ne pouvaient pas \^etre trait\'ees auparavant,
sont bien pos\'ees. Parmi les exemples se trouvent l'\'equation KPZ et le mod\`ele $\Phi^4_3$ dynamique.

Telles \'equations  peuvent \^etre d\'evelopp\'ees en s\'eries perturbatives formelles.
La th\'eorie des structures de r\'egularit\'e
permet de tronquer ce d\'eveloppement apr\'es un nombre fini de termes, et de r\'esoudre un probl\`eme de point fixe  pour le reste. 
L'id\'ee principale est une nouvelle notion de r\'egularit\'e des distributions, qui d\'epend  des termes de ce d\'eveloppement. 
\end{abstract}

\title[SPDE]
{Stochastic PDEs, Regularity structures, and interacting particle systems}

\author{Ajay Chandra}
\address{Ajay Chandra, University of Warwick
}
\email{ajay.chandra@warwick.ac.uk}

\author{Hendrik Weber}
\address{Hendrik Weber, University of Warwick
}
\email{hendrik.weber@warwick.ac.uk}


\subjclass[2000]{}

% \begin{abstract}

% \end{abstract}

%\dedicatory{}

\date\today

\maketitle

%\accomment{The referee mentioned 22 "major" comments/typos. Items 16, 17, 18, 20, 21, and 22 were already fixed. I don't understand Item 11 and I disagree with item 4 and item 19. In places with verbose changes or disagreement I put in a comment REF\#, this is done for items 1,2,4,19.} \hwcomment{For 11 - Martin already pointed out to us that in the previous version we had used to different fonts  $\mathcal{T}$ and $\mathscr{T}$. I already unified this, so probably it is already fixed} 
%%%%%%%%%%%%%%%%%%%%%%%%
\section{Lecture 1}
%%%%%%%%%%%%%%%%%%%%%%%%
\label{s:L1}
In this introductory lecture we outline the scope of the theory of regularity structures. We start by discussing two important stochastic PDE (SPDE) coming from physics. The first is the Kardar-Parisi-Zhang (KPZ) equation which is \emph{formally} given by

\begin{align}\label{e:KPZ}
\partial_t h(t,x) = \partial_x^2 h(t,x) + \frac{1}{2} (\partial_x h(t,x))^2 + \xi(t,x) \;.\tag{KPZ}
\end{align}
We will restrict ourselves to the case where the spatial variable $x$ takes values in a one dimensional space. The term $\xi(t,x)$ denotes space-time white noise which is not an actual function but  a quite irregular random (Schwartz) distribution.
This equation was introduced in \cite{KPZ} in 1986 and is a model for the fluctuations of an evolving one dimensional interface which separates two competing phases of a physical system. An explanation for the presence of the individual terms on the right hand side of \eqref{e:KPZ} can be found in \cite{QuasSpo15}.
The KPZ equation has recieved a lot of attention from mathematicians in recent years: One major development was an exact formula for the one point distribution of solutions to \eqref{e:KPZ} which was found independently by \cite{spohn} and \cite{ACQ}. This formula is susceptible to asymptotic analysis which reveals that the scaled one-point distributions converge to the Tracy-Widom distribution, a result that has been spectacularly confirmed by physical experiments  \cite{TS10}. 

\medskip
Throughout these lectures, we will focus more on our second example, the dynamic $\Phi^4_d$ model. Formally this model is given by 
\begin{align}\label{e:AC}
\partial_t \phi(t,x) = \Delta \phi(t,x) - \phi^3(t,x) - m^{2} \phi(t,x) + \xi(t,x) \;.\tag{$ \Phi^4_d$}
\end{align}
Here the spatial variable $x$ takes values in a $d$-dimensional space and $\xi$ is again space-time white noise.\
The invariant measure of \eqref{e:AC} was studied intensively in the seventies in the context of Constructive Quantum Field Theory (see e.g. \cite{glimm1974wightman, feldman1974lambdaphi, gjbook, brydges}). Formally this invariant measure is given by
\begin{equation}\label{e:pfd}
\mu(d\phi) \propto \exp\left[- 2\int_{\R^{d}} \frac{1}{4}
\phi^4(x) 
+ \frac{1}{2} 
m \phi^2 \, dx \right] \, \nu(d\phi)
\end{equation}
where $\nu$ is the law of Gaussian Free Field (GFF). The GFF can be thought of as a Gaussian random field on $\phi:\R^{d} \rightarrow \R$ with covariance given by $\mathbf{E}_{\nu}[\phi(x)\phi(y)] = \frac12 G(x-y)$ where $G$ is  the Green's function of the $d$-dimensional Laplacian. However when $d > 1$ the measure $\nu$ is not supported on a space of functions so $\phi$ must  actually be a distribution. A rigorous understanding of \eqref{e:pfd} then requires interpreting nonlinearities of distributions.

\medskip
In addition to being a toy model for QFT the measure \eqref{e:pfd} can be seen as a continuum analog of the famous ferromagnetic Ising model. For example, in \cite{GlimmPhase} the authors showed that the concrete measure corresponding to \eqref{e:pfd} in $d=2$ has a phase transition; their proof is a sophisticated version of the classical Peierls argument \cite{peierls} for the Ising model. We will close the first lecture by describing how the equation \eqref{e:AC} can be obtained as the scaling limit of a dynamical Ising model with long range interaction (at least for $d=1,2$).

\medskip

An important remark is that the theory of regularity structures will be restricted to studying \eqref{e:AC} in the regime $d < 4$ and \eqref{e:KPZ} for space dimension $d < 2$. These are both manifestations of a fundamental restriction of the theory which is the assumption of \emph{subcriticality} which will be discussed later.
Another important remark about the scope of the theory is that  regularity structures deliver a robust mathematical theory for making sense of \eqref{e:KPZ} and \eqref{e:AC} on compact space-time domains  and describe their solutions on \emph{very small scales}. The large scale behaviour of these solutions is mostly out of the current theory's scope (although some results have been obtained, see e.g. \cite{Cyril, JCH2}). This is quite important since it is primarily the large scale behaviour of solutions which makes the equations \eqref{e:KPZ} and \eqref{e:AC} experimentally validated models of physical phenomena - in particular the macroscopic behaviour of \emph{critical} systems.
However, understanding small scale behaviour and proving well-posedness is a fundamental step towards having a complete theory for these SPDE
\footnote{There are also some physical phenomena appearing in the scale regimes that regularity structures can access, such as \emph{near}-critical systems at large volume.}.
As mentioned earlier, a large obstacle we must overcome is that the $\nabla h$ of \eqref{e:KPZ} and $\phi$ of \eqref{e:AC} will in general be distributions. This makes interpreting the nonlinearities appearing in these equations highly non-trivial.

\medskip
{\bf Acknowledgements:} AC was supported by the Leverhulme trust. HW was supported by an EPSRC First grant.  We thank Martin Hairer for teaching us a lot about this subject, for giving us helpful feedback on a draft of these notes, and for helping us describe a beautiful theory with beautiful \LaTeX -macros. We also thank Cyril Labb\'e and the referee for a careful reading  and providing many comments.

\subsection{Random Distributions and Scaling Behaviour}\label{sec:scaling}

\subsubsection{Space-time white noise}\label{ss:whiteNoise}
We start by defining space-time white noise $\xi$ which appeared in both \eqref{e:KPZ} and \eqref{e:AC}. Formally $\xi(t,x)$ is a random Gaussian function on $\R \times \R^{d}$, its covariance is given by
\begin{equation}\label{e:White_noise_covarianceformal}
\E \left[ \xi(t,x) \xi(t^\prime,x^\prime) \right]= \delta(t-t^\prime) \, \delta^d(x - x^\prime) \;,
\end{equation}
where $\delta^d$ denotes the $d$-dimensional Dirac $\delta$ distribution. However for any fixed $(t,x)$ one cannot rigorously interpret $\xi(t,x)$ as a random variable, there is no coordinate process. Instead $\xi$ must be interpreted as a random distribution, a random element of $\mathcal{S}^{\prime}(\R \times \R^d)$ whose law is centered Gaussian. For any $f \in \mathcal{S}^{\prime}(\R \times \R^d)$ and smooth test function $\eta$ on $\R \times \R^d$  we denote by $(f,\eta)$ the corresponding duality pairing. The quantity $(\xi,\bullet)$ is then the analog of the coordinate process for $\xi$ and the rigorous version of \eqref{e:White_noise_covarianceformal} is given by
\begin{equation}\label{e:White_noise_covariance}
\E \left[(\xi,\eta_{1}) (\xi,\eta_{2}) \right] = \int_{\R \times \R^d} 
\eta_{1}(t,x)\eta_{2}(t,x)\, dt \; dx \;
\textnormal{ for any smooth }\eta_{1},\eta_{2}.
\end{equation}
\begin{remark}\label{BMwhitenoise}
The formula \eqref{e:White_noise_covariance} implies that $(\xi,\bullet)$ can be extended beyond smooth functions to an isometry from $L^2(\R \times \R^d)$ to $L^2(\Omega, \Ff,\P)$ where $(\Omega, \Ff,\P)$ is the underlying probability space. Adapting the definition to the case of $\R$ instead of $\R \times \R^{d}$ gives us the process called white noise, in this case one has 
\begin{align*}
\E \big[ (\xi, \mathbf{1}_{[0,s]}) (\xi, \mathbf{1}_{[0,t]}) \big] = \int_{\R} \mathbf{1}_{[0,s]}(r) \; \mathbf{1}_{[0,t]}(r) \, dr = s \wedge t \;,
\end{align*}
so $(\xi, \mathbf{1}_{[0,t]})\; ``=\; \int_{0}^{t}\xi(r)dr"$ is a Brownian motion and we see that $\xi$ can be thought of as the derivative of Brownian motion. In these lectures we will focus on equations driven by space-time noise processes so we will use the term white noise to refer to space-time white noise.
\end{remark}
\medskip
We will frequently be interested in the scaling behaviour of space-time distributions. Given a white noise $\xi$ and positive parameters $\tau, \lambda > 0$ we can \emph{define} a new random distribution $\xi_{\tau,\lambda}$ via
\begin{align*}
(\xi_{\tau, \lambda}, \eta ) := (\xi, \CS^{\tau,\lambda}\eta ) \,
\end{align*} 
where for any smooth function $\eta$ we have set $
(\CS^{\tau,\lambda}\eta)(t,x) 
:= 
\tau^{-1} \lambda^{-d} 
\eta(\tau^{-1} t,\lambda^{-1}x)$. This is a simple rescaling operation, if $\xi$ was an actual function then this would amount to setting $\xi_{\tau,\lambda}(t,x) = \xi(\tau t, \lambda x)$.
By \eqref{e:White_noise_covariance} one has
\begin{align}
\E  \left[ ( \xi_{\tau, \lambda}, \eta )^2 \right] &= \int_{\R \times \R^d}  \tau^{-2} \lambda^{-2d}  \eta(\tau^{-1}t , \lambda^{-1} x)^2 \, dt \, dx \notag\\
&= \tau^{-1}  \lambda^{-d}  \int_{\R \times \R^d} \eta(t,x)^{2}\, dt \; dx\;. \label{e:white_noise_scaling}
\end{align}
Since $\xi$ and $\xi_{\tau,\lambda}$ are centred Gaussian processes we can conclude that $\xi$ is scale invariant in distribution, in particular $ \xi_{\tau, \lambda} \overset{\text{law}}{=} \tau^{-\frac{1}{2}}\lambda^{-\frac{d}{2}} \xi$.

\subsubsection{Scaling Behaviour for SPDEs and Subcriticality}
Both \eqref{e:KPZ} and \eqref{e:AC} are non-linear perturbations of a linear SPDE called the stochastic heat equation (SHE)
\begin{equation}\label{e:SHE}
\partial_t Z(t,x) = \Delta Z(t,x) + \xi(t,x) \tag{SHE} \;
\end{equation}
where as before $(t,x) \in \R \times \R^{d}$. As before, $\xi$ cannot be evaluated pointwise and \eqref{e:SHE} has to be interpreted in the distributional sense.  Since \eqref{e:SHE} is linear it follows that the solution $Z$ will be Gaussian (for deterministic or Gaussian initial conditions).
\begin{remark}
The equation \eqref{e:SHE} is sometimes called the \emph{additive} stochastic heat equation in order to distinguish it from the \emph{multiplicative} stochastic heat equation which is given by
\begin{equation*}
\partial_t Z(t,x) = \Delta Z(t,x) +  Z(t,x) \,\xi(t,x)\;.
\end{equation*}
The above equation has a close relationship to \eqref{e:KPZ} via a change of variables called the Cole-Hopf transformation. However we will not use this transformation nor investigate the multiplicative SHE in these notes. Whenever we refer to the stochastic heat equation we are always refering \eqref{e:SHE}.
\end{remark}

\medskip
We now perform some formal computations to investigate the scaling behaviour of solutions \eqref{e:SHE}. For $\lambda > 0$ and suitable scaling exponents $\al, \be,\ga \in \R$ we define $\hat{Z}(t,x) := \la^{\al}Z(\la^{\be} t, \la^\ga x)$ and $\hat{\xi} := \la^{ \frac{\be}{2}} \la^{ \frac{d\ga }{2}} \xi_{\la^\be, \la^\ga }$, it then follows that
\begin{equation}
\partial_t \hat{Z} =  \la^{\be - 2 \ga} \Delta \hat{Z} + \la^{\al + \frac{\be}{2} - \frac{d\ga}{2}} \hat{\xi} \;.
\end{equation}
We have already shown that $\hat{\xi} \overset{\text{law}}{=} \xi$. Therefore, if we set
\begin{equation}\label{e:scalinG}
\al = \frac{d}{2}-1 \;,  \qquad    \be = 2\;,   \qquad \text{and} \qquad \ga =1 \;
\end{equation}
then we see that $\hat{Z} \overset{\text{law}}{=} Z$ (ignoring boundary conditions) so the solution to \eqref{e:SHE} is also scale invariant.

\medskip
In general non-linear equations like \eqref{e:KPZ} and \eqref{e:AC} will not be scale invariant.  If one rescales these equations according to the exponents given in \eqref{e:scalinG} then the non-linearity will be multiplied by a prefactor which is some power of $\lambda$; the assumption of subcriticality then requires that this prefactor vanish as $\lambda \rightarrow 0$. Roughly speaking, this condition enforces that the solutions to \eqref{e:KPZ} and \eqref{e:AC} both behave like the solution to the \eqref{e:SHE} \emph{on small scales}. 
Let us illustrate this for \eqref{e:KPZ}. We perform the same scaling as in \eqref{e:scalinG} and set $\hat{h}(t,x) = \la^{-\frac{1}{2}}h(\la^2 t, \la x)$. This gives
\begin{align*}
\partial_t \hat{h} = \partial_x^2 \hat{h} + \frac{\la^{\frac12} }{2} (\partial_x \hat{h})^2 + \hat{\xi}\;.
\end{align*}
On small scales, i.e. for $\la \to 0$, the prefactor $\la^{\frac12}$ of the non-linear term goes to zero. 
We perform the same calculation for \eqref{e:AC}, for this discussion the mass term $m^2\phi$ is irrelevant so we drop it. Applying the scaling \eqref{e:scalinG}, i.e. setting $\hat{\phi}(t,x) = \la^{\frac{d}{2}-1} \phi( \la^2t,  \la x)$ we get
\begin{align*}
\partial_t \hat{\phi}(t,x) = \Delta \hat{\phi} (t,x) -  \la^{4-d}\hat{\phi}^3  + \hat{\xi} \;.
\end{align*}
If the spatial dimension $d$ is strictly less than $4$ the prefactor $\la^{4-d}$ vanishes in the limit $\la \to 0$. We call $d < 4$ the \emph{subcritical
} regime. If $d=4$ the prefactor $\la^{4-d} = 1$; this is the \emph{critical} regime. The regime $d\geq5$ is called the \emph{supercritical} regime.  

\medskip
We now state a crude ``definition" of subcriticality which will be sufficient for these notes. The interested reader is referred to \cite[Assumption 8.3]{RegStr} for a more precise definition of subcriticality which also extends to different types of noise $\xi$ and a \rf{larger class of regularising linear operators}.
%\footnote{The framework here is not restricted to just parabolic PDE. If one writes the underlying homogenous linear equation for a PDE as $\mathcal{L}u = 0$ then the key property we ask for is that $\mathcal{L}^{-1}$ be representable as a regularizing integral operator.} PDE.
\refcomment{REF1}
\begin{definition}\label{e:subcritical}
Consider the equation 
\begin{equation}\label{e:general_equation}
\partial_t u = \Delta u + F(u, \nabla u) + \xi \;, 
\end{equation}
in $d$ spatial dimensions. Equation \eqref{e:general_equation} is called \emph{subcritical} if under the scaling \eqref{e:scalinG}  the non-linear term $F(u, \nabla u)$ gets transformed into a term $F_\lambda(u, \nabla u)$ which formally goes to zero as $\la$ goes to $0$.
\end{definition}
The main result of \cite{RegStr} can roughly be paraphrased as follows.
 \begin{mtheorem}[\cite{RegStr}]\label{metatheorem}
 Assume that SPDE \eqref{e:general_equation} is subcritical.
We assume that $x$ takes values in a compact subset of $\R^d$ with some boundary conditions. Furthermore, we prescribe an initial condition $u_0$ which has the same spatial regularity as we expect for the solution $u$. 

There is a natural notion of solution to \eqref{e:general_equation} and such solutions exist and are unique on a time interval $[0,T)$ for some random $T>0$.
 \end{mtheorem}

\begin{remark}
The assumption of subcriticality is not just a technical restriction. For example it has been proven that a non-trivial $\Phi^4_{d}$ cannot exist for $d \ge 5$ (this result extends to $d \ge 4$ with some caveats) \cite{Aizenman},\cite{FrohTriv}. 
\end{remark}
\begin{remark}
We will see below that the statement of Metatheorem~\ref{metatheorem} really consists of two independent statements: (i) For subcritical equations it is possible to build the algebraic and analytic structure that allows to formulate the equation and (ii) all the stochastic processes entering the expansion converge (after renormalisation). It is an astonishing fact that in the case of equations driven by white noise, the scaling conditions for these two statements to hold coincide. It is however possible to define a notion of subcriticality for more general equations driven by a more general noise term. This generalised notion still implies that it is possible to build the algebraic and analytic structure, but there are examples, where stochastic terms fail to be renormalisable. \footnote{This may be familiar to readers who know the theory of rough paths: In principle this theory allows to solve differential equations with a driving noise $dW$ for $W$ of arbitrary positive regularity by increasing the number of iterated integrals one considers. However, the  stochastic calculations needed to actually construct the iterated integrals fail for
 fractional Brownian motion of Hurst index $H<\frac14$  \cite{laure}.}
\end{remark}

\begin{remark}
For what follows we restrict ourselves to the subcritical regime. While the equations \eqref{e:KPZ} and \eqref{e:AC} are not scale invariant themselves they do interpolate between two different scale invariant space-time processes, one governing small scales and another governing large scales. As mentioned before the small scale behaviour should be governed by the solution to \eqref{e:SHE}. At large scales it is expected that (i) one must use different exponents then \eqref{e:scalinG} to get a meaningful limit and (ii) the limiting object will be a non-Gaussian process scale invariant under these different exponents. 

For \eqref{e:KPZ} one should apply the famous $1,2,3$-scaling $\hat{h}(t,x) = \la^{-\frac12} h(\la^{\frac{3}{2}} t , \la x) $. Then, setting $\hat{\xi} = \la  \xi_{\la^\frac{3}{2}, \la } $ one has the equation
\begin{equation}\label{e:KPZlargescale}
\partial_t \hat{h}(t,x) =  \frac{1 }{2} (\partial_x \hat{h})^2 + \la^{-\frac12} \partial_x^2 \hat{h} + \la^{-\frac14} \hat{\xi}(t,x) \,.
\end{equation}
Modulo the subtraction of a drift term, as $\la \rightarrow \infty$ the solution of \eqref{e:KPZlargescale} is conjectured to converge to an object called the \emph{KPZ fixed point} (see \cite{JeremyIvan}). This limiting object is not yet very well understood\footnote{In \cite{JeremyIvan} it was shown that this object does \emph{not} coincide with the entropy solution of the Hamilton-Jacobi equation $\partial_t h = \frac12 (\partial_x h)^2$.}. 

The behaviour of \eqref{e:AC} at large scales is also of interest, but much less is known in this case.
\end{remark}
\begin{remark}
The main aim of these lectures is to show how the theory of regularity structures can be used to construct local-in-time solutions for $\Phi^4_3$. Let us point out however, that after this result was first published by Hairer in \cite{RegStr}, two alternative methods to obtain similar results have been put forward: In \cite{Massimiliano1}  Gubinelli, Imkeller and Perkowski developed  the method of ``paracontrolled distributions"  to construct solutions to singular stochastic PDEs and this method was used in   \cite{Massimiliano2} to construct local in time solutions to $\Phi^4_3$. Independently, in \cite{Kupi} Kupiainen proposed yet another method based on Wilsonian renormalization group analysis. 
The result for $\Phi^4_3$ that can be obtained by the method of ``paracontrolled distributions" is essentially equivalent to the result obtained in the framework of regularity structures and arguably this method is simpler because less machinery needs to be developed. However, the construction of a more comprehensive theory pays off when looking at more complicated models. For example, approximation results for the multiplicative stochastic heat equation such as obtained in \cite{HairerPardoux} seem out of reach of the method of ``paracontrolled distributions" for the moment.
\end{remark} 

\begin{remark}
At the time of writing these lectures there were at least three other works (\cite{hairer2014singular}, \cite{hairer2015introduction}, and \cite[Chapters 13-15]{friz2014}) that survey the theory of regularity structures. In particular, \cite{hairer2015introduction} gives a much more detailed exposition for many of topics we only briefly discuss in Lecture 4. The authors' goal for the present work was (i) to clarify certain central concepts via numerous concrete computations with simple examples and (ii) to give a panoramic view of how the various parts of the theory of regularity structure work together. 
\end{remark}

\subsection{The need for renormalisation}

We must clarify what is meant by \emph{solution theory} and \emph{uniqueness} in the \emph{Metatheorem}~\ref{metatheorem}. \emph{Classical} solution theories for SPDEs (see e.g. \cite{dPZ,Hai09,PR}) do not apply here since the solutions are too irregular. For \eqref{e:KPZ} the solution $h(t,x)$ has the regularity of a Brownian motion in space - the mapping $x \mapsto h(t,x)$ for fixed $t$ is almost surely $\alpha$-H\"older continuous for every $\alpha <\frac12$ but not for any $\alpha \geq \frac12$. Remembering Remark \ref{BMwhitenoise} we expect that the distributional derivative $\partial_x h$ has the regularity of spatial white noise.
For \eqref{e:AC} the solution theory was already fairly understood only in $d=1$ - there $\phi$ is $\alpha$-H\"older for every $\alpha < \frac12$ which is largely sufficient to define $\phi^3$ (see \cite{FunakiString}). In the cases $d=2,3$  the subcriticality assumption stated in Definition~\ref{e:subcritical} still applies but $\phi$ will not be regular enough to be a function.

\medskip
A natural way to try to interpret nonlinear expressions involving highly irregular objects is regularization. In the context of our singular SPDE this means that if we show that solutions of \emph{regularized} equations converge to some object as we remove the regularization then we can define this limiting object as the solution of the SPDE.
Unfortunately this naive approach does not work, the solutions to the regularized equations will either fail to converge or converge to an uninteresting limit. We use the dynamic $\Phi^4_2$ model as a concrete example of this. One natural regularization consists of replacing $\xi$ by a smoothened  noise process. Let $\rho$ be a smooth function on $\R \times \R^{d}$ which integrates to $1$. For $\delta >0$ we set 
\begin{equation}\label{e:eta_delta}
\rho_\delta(t,x) :=\delta^{- (2 + d) }   \rho(\delta^{-2} \, t, \delta^{-1 } \,x )  \;.
\end{equation}
We use the parabolic scaling $\delta^{-2}t$ and $\delta^{-1}x$ since it will be a convenient choice for later examples.
For any $\delta > 0$ we define the regularized noise $\xi_\delta  := \xi \star \rho_\delta$ where $\star$ indicates space-time convolution. For any fixed positive $\delta$ proving (local) existence and uniqueness for the solution of 
\begin{align}\label{e:naive-approximation}
\partial_t \phi_\delta = \Delta \phi_\delta - \phi_\delta^3 + \xi_\delta \;
\end{align}
poses no problem in any dimension since the driving noise $\xi_{\delta}$ is smooth. However in \cite{HRW} this example was studied\footnote{Actually in \cite{HRW} a different regularisation of the noise is considered, but that does not change the result.} on the two dimensional torus and it was shown that as $\delta \downarrow 0$ the solutions $\phi_\delta$ converge to the trivial limit $0$ for any initial condition! 
In order to obtain a non-trivial limit the equation \eqref{e:naive-approximation} has to be modified in a $\delta$ dependent way. We will see that in dimensions $d=2,3$ if one considers
\begin{align}\label{e:less_naive1}
\partial_t \phi_\delta = \Delta \phi_\delta - (\phi_\delta^3 - 3 \mathfrak{c}_\delta \phi_\delta) + \xi_\delta \; ,
\end{align}
for a suitable dimension dependent choice of \emph{renormalisation} constants $\mathfrak{c}_\delta$, then the solutions $\phi_\delta$ do indeed converge to a non-trivial limit $\phi$. This constant $\mathfrak{c}_{\delta}$ will diverge as $\delta \downarrow 0$. For \eqref{e:less_naive1} in $d=2$ one can take $C_1 \log(\delta^{-1})$ for a specific constant $C_1$, while for $d=3$ one can take $\mathfrak{c}_\delta = C_1 \delta^{-1} + C_2 \log(\delta^{-1})$ for specific constants $C_{1},C_{2}$ where $C_{1}$ depends on the choice of $\rho$.
A similar \emph{renormalisation} procedure is necessary for the KPZ equation. In \cite{MartinKPZ} it was shown that solutions of 
\begin{align}\label{e:less_naive2}
\partial_t h_\delta(t,x) = \partial_x^2 h_\delta(t,x) + \frac{1}{2} (\partial_x h_\delta(t,x))^2  -\mathfrak{c}_\delta  + \xi_\delta(t,x) \;
\end{align}
on the one-dimensional torus converge to a non-trivial limit $h$ when one sets $\mathfrak{c}_\delta = C_1 \delta^{-1}$ for a specific constant $C_1$.
We call \eqref{e:less_naive1} and \eqref{e:less_naive2} \emph{renormalized equations} and the limits of their corresponding solutions $\phi:= \lim_{\delta \downarrow 0} \phi_{\delta}$ and $h := \lim_{\delta \downarrow 0} h_{\delta}$ are what we define to be solutions of \eqref{e:AC} and \eqref{e:KPZ} in Metatheorem \ref{metatheorem}; such solutions are often called \emph{renormalized solutions}.

\medskip
We now turn to discussing uniqueness for these SPDE. For a fixed subcritical equation one can choose different renormalization schemes which yield different families of renormalized equations and different corresponding renormalized solutions. A simple example of this in the cases of \eqref{e:less_naive1} or \eqref{e:less_naive2} would be shifting $\mathfrak{c}_\delta$ by a finite constant independent of $\delta$, this would change the final renormalized solution one arrives at. One could also change the renormalization scheme by using a different mollifier $\rho$ or use a non-parabolic scaling for a given mollifier.
Even with all these degrees of freedom in choosing a renormalization scheme it turns out that for a given subcritical equation the corresponding family of possible renormalized solutions will be parameterized by a finite dimensional space. If a renormalization scheme yields a non-trivial limit renormalized solution then this solution will lie in this family. For \eqref{e:KPZ} and \eqref{e:AC} the family of solutions is parameterized by a single real parameter.

\begin{remark}
The reader should compare the situation just described to the familiar problem one encounters when solving the stochastic differential equation
\begin{equation}\label{e:SDE1}
\dot{x}(t) = b(x(t)) + \sigma(x(t)) \xi(t)\,;
\end{equation}
(which is more conventially written as $dx_t = b(x_t) dt + \sigma(x_t) dW_t$). There it is well-known that different natural regularisations converge to different solutions. An explicit Euler scheme, for example, will lead to the solution in the \emph{It\^o} sense (see e.g. \cite{KloedenPlaten})  whereas smoothening the noise usually leads to the \emph{Stratonovich} solution (see the e.g. the classical papers \cite{WongZakai1}, \cite{WongZakai2}). There is a whole one-parameter family of meaningful solution-concepts to \eqref{e:SDE1} and the question for uniqueness is only meaningful once it is specified which particular solution one is looking for. 
\end{remark}

\medskip
Later in these lecture notes we will discuss how the theory of regularity structures gives a ``recipe'' for coming up with renormalization schemes which give non-trivial limits, we will also see that the limiting solution itself will be a fairly concrete object in the theory.

\subsection{Approximation of renormalised SPDE by  interacting particle systems} 
One might think that by introducing \eqref{e:less_naive1} and \eqref{e:less_naive2} we have turned our back on the original equations and physical phenomena they represent. This is not the case however. 
There is strong evidence, at least for KPZ and for $\Phi^4_d$, that the  \emph{renormalised} solutions are the physical solutions. For the KPZ equation subtracting a diverging constant corresponds simply to a change of reference frame. Furthermore, it was shown in \cite{BG97} that the solutions to KPZ arise as continuum limits for the weakly asymmetric simple exclusion process, a natural surface growth model. 

\medskip
We will now discuss how the dynamic $\Phi^4_d$ model can be obtained as a continuum limit of an Ising model with long range interaction near criticality. In the one dimensional case (where no renormalisation is necessary) this is a well known result \cite{BPRS,FR} and the right scaling relations for spatial dimensions $d=1,2,3$ were conjectured in \cite{GLP}. One of the interesting features of these scaling relations is that  the ``infinite" renormalisation constant has a natural interpretation as shift of the critical temperature. The two dimensional convergence result was established only recently in \cite{MW}. We will now briefly discuss this result and show how the relevant scaling relations  relate to the subcriticality assumption for \eqref{e:AC}.

\medskip

 For $N \geq 1$ let $\LN =\Z^d / (2N+1) \Z^d$ be the $d$-dimensional discrete torus. 
% which we identify with the set  $  \{-N, -(N-1), \ldots, N \}^d$. 
 Denote by $\SN = \{ -1, +1 \}^{\LN}$ the set of spin configurations on $\LN$. For a spin configuration $\si = (\si(k), \, k \in \Lambda_N )$ we define the \emph{Hamiltonian} as  
\begin{equation*}%\label{e:Hamiltonian}
\Hg(\si) := -  \frac12 \sum_{k, j \in \LN}  \kg(k-j) \si(j)   \si(k)  \;.
\end{equation*}
$\gamma \in (0,1) $ is a model parameter which determines the interaction range between spins. It enters the model through the definition of the interaction kernel $\kg$ which is given by 
\begin{equation*}%\label{e:samplek}
\kg(k) = \ct \, \gamma^d \, \KK(\ga k)  \;,
\end{equation*}
where $\KK\colon \R^d \to \R$ is a smooth, nonnegative function with compact support and $\ct$ is chosen to ensure that $\sum_{k \in \LN} \kg =1$. One should think of this model as an interpolation between the 
\emph{classical} Ising model where every spin interacts only with  spins in a fixed neighbourhood (corresponding to the case $\gamma=1$) and the \emph{mean-field} model, where every spin interacts with every other spin and the 
geometry of the two-dimensional lattice is completely lost (corresponding to the case $\gamma =0$). 
Then for any \emph{inverse temperature} $\beta>0$ we can define the Gibbs measure $\lbg$ on $\SN$ as 
\begin{equation*}
\lbg (\si) := \frac{1}{\Zbg}\exp\Big( - \be\Hg(\si) \Big)\; ,
\end{equation*}
where as usual
\begin{equation*}
\Zbg := \sum_{\si \in \SN} \exp\Big( - \be\Hg(\si) \Big) \; ,
\end{equation*}
denotes the normalisation constant that makes $\lbg$ a probability measure. 

\medskip

We want to obtain the SPDE \eqref{e:AC} as a scaling limit for this model and therefore, we have to introduce dynamics. One natural choice 
is given by the \emph{Glauber-dynamics} which are defined by the generator
\begin{equation*}%\label{e:Generator}
\LgN f(\si) = \sum_{j \in \LN} \cg(\si,j) \big(f(\si^j) -f(\si) \big) \;,
\end{equation*}
 acting on functions $f \colon\SN \to \R$. Here $\si^j \in \SN$ is the spin configuration that coincides with $\si$ except for a flipped spin at position $j$.  
The jump rates  $\cg(\si, j)$  are given by 
$$
\cg(\si, j) := \frac{\lbg(\si^j)}{\lbg(\si) + \lbg(\si^j)}\;.
$$
It is easy to check that these jump rates are reversible with respect to the measure $\lbg$. 

\medskip

In order to derive the right rescaling for the spin-field $\sigma$ we write
\begin{equation*}
%\label{e:defXg}
\Xg(t,x) = \frac{1}{\dg} \hg\bigg( \frac{t}{\ag}, \frac{x}{\eps} \bigg)  \qquad x\in \Le, \, t \geq 0.
\end{equation*}
Here $\hg(k,t) = \sum_{\ell \in \LN} \kg(k - \ell) \sigma(t,\ell)$ is a field of local spatial averages of the field $\sigma$ and $\ag,\dg,\eg$ are scaling factors to be determined
\footnote{Working with a field of local averages rather than with the original field $\sigma$ is more convenient technically, but  a posteriori convergence for the original field $\sigma$ 
in a weaker topology can be shown as well.}.
Let us sketch how to derive the right scaling relations for $\ag,\dg,\eg,\ga$. We only sketch the calculation - the details  can be found in \cite{MW}.
If we apply the generator $\LgN$ to the field $\Xg$ an explicit calculation shows that
\begin{align}
\Xg(t,x) =  & \Xg(0,x) + \int_0^t \!   \bigg(  \frac{ \eg^2}{\ga^2 } \frac{1}{\ag}  \Delta_\gamma \Xg(s,x) + \frac{(\be -1)}{\al} \Kg \ae \Xg(s,x) \notag \\  
&- \frac{\be^3}{3} \frac{\dg^2}{\al} \Kg \ae \Xg^3  (s,x)  + \Kg \ae \Eg(s,x) \bigg)\, ds + \Mg(t,x) \;  ,\label{e:evolution2}
\end{align}
for $x \in \Le$. Here $\Delta_\gamma$ is a difference operator (based on the kernel $\kg$) which is scaled to approximate the  Laplacian acting on the rescaled spatial variable $x$. $\Kg$ is an approximation of a Dirac delta function, $\Mg$ is a martingale and $\Eg$ is a (small) error term. The second relevant relation concerns the noise intensity. This is determined by the quadratic variation of $\Mg$ which is given by 
\begin{align*}
\langle \Mg(\cdot, x ),& \Mg (\cdot, y) \rangle_t  \notag\\
&=4 \frac{\eg^d}{\dg^2\ag } \int_0^t  \!\! \sum_{z \in \Le}\eg^2  \Kg( x -z) \, \Kg(y-z) \,  \Cg\big(s, z\big) \, ds \;, \label{e:QuadrVar2}
\end{align*} 
where $\Cg(s,z) :=  \cg(\si(s/\ag), z/\eg )$. 

\medskip

In order to obtain \eqref{e:AC}  we need to choose scaling factors satisfying  
\begin{equation*} %\label{e:scaling1}
1 \approx \frac{\eps^2}{ \ga^2}\frac{1}{\ag}  \approx \frac{\dg^2}{\ag} \approx \frac{\eg^d}{\dg^2\al } \;,
\end{equation*}
which leads to 
\begin{equation*}%\label{e:scalingn}
\eg \approx \ga^{\frac{4}{4-d}}, \quad \qquad \ag\approx \ga^{\frac{2d}{4-d}}, \quad  \qquad \dg \approx \ga^{\frac{d}{4-d}} \;.
\end{equation*}
It is striking to note, that these equations can be satisfied for spatial dimensions $d=1,2,3$ but they cannot hold as soon as $d=4$. This corresponds exactly to the criticality assumption for \eqref{e:AC}.

\medskip

At first sight \eqref{e:evolution2} suggests that $\beta$ should be so close to one that $(\be -1)/ \al = O(1)$. Note that $\beta=1$ is the critical temperature for the mean field model in our setup. But for $d\geq2$ this naive guess is incorrect. As for the macroscopic equation the microscopic model has to be be renormalised. Indeed, the main result of \cite{MW} states that for $d=2$ if we  set
\begin{align*}%\label{e:scalingbeta}
( \bg -1) = \ag (\mathfrak{c}_\gamma - m^2)\;,
\end{align*}
where the ``mass" $m\in \R$ is fixed and the extra term $\mathfrak{c}_\gamma$ chosen in a suitable way (diverging logarithmically) as $\ga$ goes to $0$, then (under suitable assumptions on the initial data)  $\Xg$ does indeed converge in law 
to the solution of \eqref{e:AC}. A similar result is expected to hold   in three dimensions.

\section{Lecture 2}
\label{s:l2}
We start this lecture by describing how we will keep track of the regularity of space-time functions \emph{and} distributions. After that we give a review of classical solution techniques for semilinear (stochastic) PDEs. We will explain how a lack of regularity causes problems for these theories, using \eqref{e:AC} in $d=2,3$ as our examples. 
We will then describe a \textit{perturbative} approach to these equations. Divergences will be seen clearly in formal expansions of the solutions, this will motivate the choice of diverging renormalization constants appearing in the renormalized equations. We will also go through some  calculations to make the objects at hand concrete; this will prepare us for Section~\ref{s:l3} where we present more abstract parts of the theory.

\subsection{Regularity} 

The functional spaces we use in these notes are a generalization of the usual family of H\"older spaces, these spaces will be denoted by $\Cs^{\alpha}$ where $\alpha$ is the analog of the H\"older exponent. We will measure space-time regularity in a parabolic sense which is why we write $\mathfrak{s}$ in the subscript of $\Cs^{\alpha}$ (the $\mathfrak{s}$ stands for ``scaled"). For $z, z' \in \R \times \R^{d}$ we denote by $|| z' - z||_{\mathfrak{s}}$ the parabolic distance between $\bar{z}$ and $z$. Writing $z' = (t',x')$ and $z = (t,x)$ we set
\[
||z' - z||_{\mathfrak{s}} := 
|t' - t|^{\frac{1}{2}}
+
\sum_{j=1}^{d} |x_{j}' - x_{j}|.
\]
Below it will also be useful to have the notion of  \emph{scaled dimension} $d_{\mathfrak{s}} = d+2$ for space-time $\R \times \R^d$, i.e. the time variable counts for two dimensions.

\begin{definition}\label{positiveholder}
For $\alpha \in (0,1)$ the space $\Cs^{\alpha}(\R \times \R^{d})$ consists of all continous functions $u \colon \R \times \R^d \rightarrow \R$ such for every compact set $\mathfrak{K} \subset \R \times \R^d$ one has 

\begin{align}\label{holdernorm}
\sup_{ 
\substack{ z,z' \in \mathfrak{K}\\ z \not = z' }
}
\frac{\left|u(z) - u(z') \right|}{ ||z - z'||^{\alpha}_{\mathfrak{s}} }
<
\infty\;.
\end{align}
\end{definition}

\begin{remark} In these notes the theory of regularity structures will be applied to problems in compact space-time domains. Accordingly we will be concerned with estimates that are uniform over compacts instead of trying to get estimates uniform over all of space-time.
\end{remark} 

In order to accomodate distributions we will want an analog of H\"older spaces where $\alpha$ is allowed to be negative. A natural choice are the (parabolically scaled) Besov spaces $\{ \CB_{\infty, \infty}^\alpha \}_{\alpha \in \R}$. In particular these spaces agree with our earlier definition for $\alpha \in (0,1)$. In analogy to the positive H\"older spaces we still denote these Besov spaces by $\Cs^{\al}$ when $\alpha < 0$.

There are several ways to characterise these Besov spaces (including Paley-Littlewood decomposition (\cite{BCD}) or wavelet decompositions). For these notes we use a simple definition that is convenient for our purposes. 
First we need some more notation. For any positive integer $r$ we define $B_{r}$ to be the set of all smooth functions $\eta:\R^{d+1} \rightarrow \R$ with $\eta$ supported on the unit ball of $\R^{d+1}$ (in the parabolic distance $||\cdot||_{\mathfrak{s}}$) and $||\eta||_{C^{r}} \le 1$. \rf{Here $||\cdot||_{C^{r}}$ denotes the standard norm on $C^{r}$, that is 
\[
||f||_{C^{r}} 
:= 
\sup_{\alpha, |\alpha| \le r}
\sup_{x \in \R^{d+1}} 
|D^{\alpha}f(x)|
\] where we used multi-index notation.} \refcomment{REF2} We then have the following definition.
\begin{definition}\label{negativeholder}
Suppose that $\alpha < 0$. We define $\Cs^{\alpha}$ to be the set of all distributions $u \in \mathcal{S}'(\R^{d+1})$ such that for any compact set $\mathfrak{K} \subseteq \R \times \R^d$ one has
\begin{equation*}
||u||_{\Cs^{\alpha}(\mathfrak{K})}
:=
\sup_{z \in \mathfrak{K}}
\sup_{
\substack{
\eta \in B_{r} \\
\lambda \in (0,1]}
}
\left|
\frac{
\langle u , \CS_z^\lambda \eta \rangle 
}
{\lambda^{\alpha}}
\right|
<
\infty
\end{equation*}
where have set $r = \lceil -\alpha \rceil$ and
\begin{equation}\label{e:scaledFunction}
 \CS_z^\lambda \eta(s,y) :=  \la^{-d-2} \;\eta\big(\la^{-2}(s-t), \la^{-1}(y-x)  \big) \;.
\end{equation}
\end{definition}
One can adapt Definition \ref{negativeholder} to the case $\alpha > 0$ (extending Definition \ref{positiveholder}). We first need to define the parabolic degree of a polynomial. Given a multindex $ k = (k_{0},k_{1},\dots,k_{d}) \in \mathbf{N} \times \mathbf{N}^{d}$ we define the monomial $z^{k}$ in the standard way, we also define the parabolic degree of this monomial to be $|k|_{\mathfrak{s}}:= 2k_{0} + \sum_{j=1}^{d} k_{j}$. We then define the parabolic degree of a polynomial $P(z)$ to be the maximum of the parabolic degree of all of its constituent monomials.

\begin{definition}\label{positiveholder2}
Suppose that $\alpha \ge 0$. We define $\Cs^{\alpha}$ to be the set of all functions $u \in \mathcal{S}'(\R^{d+1})$ such that there exist polynomials $\{P_{z}\}_{z \in \R^{d+1}}$, each of parabolic degree less than $\alpha$, such that for any compact set $\mathfrak{K} \subseteq \R \times \R^d$ one has
\begin{equation}
||u||_{\Cs^{\alpha}(\mathfrak{K})}
:=
\sup_{z \in \mathfrak{K}}
\sup_{
\substack{
\eta \in B_{0} \\
\lambda \in (0,1]}
}
\left|
\frac{
\langle u - P_{z}, \CS_z^\lambda \eta \rangle 
}
{\lambda^{\alpha}}
\right|
<
\infty.
\end{equation}
\end{definition}

\begin{remark}
It is easily checked that in the above definition $P_{z}$ must just be the $\lfloor \alpha \rfloor$ -th order Taylor expansion for the function $u$ centered at $z$.
\end{remark}

\begin{remark}
Important theorems about $\Cs^{\alpha}$ spaces (like Theorem \ref{Schaudestimate}) fail to hold when $\alpha$ takes integer values \footnote{For example, when $\alpha$ is a positive integer the Besov space $\CB_{\infty, \infty}^\alpha$ differs from the classical H\"older space with parameter $\alpha$.}. We will implicitly assume that any $\Cs^{\alpha}$ space appearing in the assumption or conclusion of a theorem uses a non-integer value of $\alpha$. 
\end{remark}

We now investigate the regularity of space-time white noise. A calculation similar to \eqref{e:white_noise_scaling} shows that for $\lambda \in (0,1]$ one has
\begin{equation}\label{e:Gaussian_scaling}
\E \langle \xi,  \CS_z^\lambda \eta \rangle^2 \lesssim \la^{-d-2} \; . 
\end{equation}
This suggests that $\xi$ has regularity $\al = -\frac{d}{2}  -1$. The following ``Kolmogorov like" theorem which is small variation of \cite[Thm 5.2]{JCH2}  shows that this is almost true. 

\begin{theorem}\label{Kolmogorov} 
Suppose that we are given a $\mathcal{S}(\R^{d+1})$-indexed stochastic process $\xi(\bullet)$ which is linear (that is a linear map from $\mathcal{S}(\R^{d+1})$ to the space of random variables).

Fix any $\alpha < 0$ and a \rf{$p\geq 1$}. \refcomment{REF3} Suppose there exists a constant $C$ such that for all $z \in \R^{d+1}$, and for all $\eta \in \mathcal{S}(\R^{d+1})$ which are supported on the unit ball of $\R^{d+1}$ and satisfy $\displaystyle\sup_{z' \in \R^{d+1}} |\eta(z')| \le 1$ one has
\begin{align}\label{regularitymomentbound}
\E \big| \xi(\CS_z^\la \eta) \big|^p \leq C \la^{\alpha p} \, \textnormal{ for any } \lambda \in (0,1]
\end{align}
then there exists a random distribution $\tilde{\xi}$ in  $\mathcal{S}(\R \times \R^d) $ such that for all $\eta$ we have $\xi(\eta) = \tilde{\xi}(\eta)$ almost surely. Furthermore, for any $\alpha^{\prime} < \alpha - \frac{d+2}{p}$ and any compact $\mathfrak{K} \subseteq \R \times \R^d$ we have
\begin{align*}
\E \| \tilde{\xi} \|^p_{\Cs^{\alpha^\prime}(\mathfrak{K})} < \infty\;.
\end{align*}
\end{theorem}
\begin{proof}[Sketch of proof] We start by recalling the argument for the classical Kolmogorov criterion for a stochastic process $X(t)$ indexed by $t \in \R$ (ignoring all questions of suitable modifications). The first step is the purely deterministic observation that for any continuous function $X$ we have
\begin{align*}
\sup_{s \neq t \in [-1,1]} \frac{|X(s) - X(t) |}{|s-t|^{\alpha^\prime}} \ls \sup_{k \geq 0} \sup_{s\in 2^{-k}\Z \cap [-1,1)} 2^{k \alpha^\prime} |X(s+ 2^{-k}) - X(s) |\;.
\end{align*} 
This implies (still in a purely deterministic fashion) that
\begin{align*}
\Big(& \sup_{s \neq t \in [-1,1]} \frac{|X(s) - X(t) |}{|s-t|^{\alpha^\prime}}\Big)^p \\
& \ls \sup_{k \geq 0} \sup_{s\in 2^{-k }\Z \cap [-1,1)} 2^{k \alpha^\prime p} |X(s+ 2^{-k}) - X(s) |^p\\
& \ls  \sum_{k \geq 0} \sum_{s\in 2^{-k }\Z \cap [-1,1)} 2^{k \alpha^\prime p} |X(s+ 2^{-k}) - X(s) |^p.
\end{align*}
The only stochastic ingredient consists of taking the expectation of this expression which yields 
\begin{align*}
\E  \Big(& \sup_{s \neq t \in [-1,1]} \frac{|X(s) - X(t) |}{|s-t|^{\alpha^{\prime}}} \Big)^p \\
& \ls \sup_{s \neq t}\Big( \frac{1}{|s-t|^{\alpha p}} \E |X(s) - X(t)|^p \Big) \;\sum_{k \geq 0} 2^k 2^{k {\alpha'} p } 2^{-k \alpha p}\;,
\end{align*}
and summing the geometric series.

\medskip
The argument for  Theorem \ref{Kolmogorov} follows a very similar idea. The crucial deterministic observation is that the Besov norm $\| \xi \|_{\Cs^{\alpha^\prime}}$ can be controlled by 
\begin{align*}
\| \xi \|_{\Cs^{\alpha^\prime}(\mathfrak{K})} \ls \sup_{k \geq 0} \sup_{x \in 2^{-2k}\Z \times 2^{-k}\Z^d \cap \bar{\mathfrak{K}}} 2^{- k \alpha'} \big( \xi, \mathcal{S}_{x}^{2^{-k}}\eta \big) \;,
\end{align*}
where $\bar{\mathfrak{K}}$ is another compact set that is slightly larger than $\mathfrak{K}$ and $\eta$ is a \emph{single}, well chosen test function. There are different ways to construct such a function $\eta$ e.g. using wavelets (as in \cite{RegStr}) or using the building blocks of the Paley-Littlewood decomposition (as in \cite{JCH2}). The argument then follows the same strategy replacing the $\sup$ by a sum and taking the expectation in the last step only. 
\end{proof}

Going back to the discussion of white noise, we recall the basic fact that for a Gaussian random variable $X$ control of the second moment gives control on all moments - for all positive integers $p$ there exists a constant $C_{p}$ such that 
\[
\mathbb{E}[ |X|^p ] 
\le 
C_{p}\left( \mathbb{E}[X^2] \right)^{p/2}.
\] 
It follows that for Gaussian processes once one has \eqref{regularitymomentbound} for $p=2$ then a similar bound holds for all $p$. Thus we can conclude that $\xi$ has regularity $\Cs^{-\frac{d}{2}-1 - \ka}$ for every $\ka>0$.

\subsection{Linear theory}
From now on we will assume periodic boundary conditions in space - the spatial variable $x$ will take values in the $d$-dimensional torus $\T^d$ (which is identified with $[-\pi, \pi]^d$). When more convenient 
we will sometimes view a function or distribution defined on $\R \times \T^d$ as defined on $\R \times \R^d$ and periodic in space. 
We first recall Duhamel's principle or the \textit{variation of constants formula}. Consider the inhomogeneous heat equation given by 
\begin{align*}
\partial_t u &= \Delta u +f \\
u(0, \cdot) &= u_0 \;
\end{align*}
where $f$ is a space-time function and $u_{0}$ is a spatial initial condition. Under very general regularity assumptions on $f$ and $u_0$ the solution is given by the formula
\begin{equation}\label{e:Duhamel}
u(t,x) =  \int_0^t \int_{\T^d} K(t-s,x-y) \; f(s,y) \,dy \, ds  + \int_{\T^d} K(t, x-y) u_0(y) \, .
\end{equation}
Here $K$  is the heat kernel on the torus, which for $t>0$ and $x \in \T^d$ is given by
\begin{equation*}
K(t,x) = \sum_{k \in 2 \pi \Z^d}  \frac{1}{(4 \pi t )^{\frac{d}{2}}} \exp \Big(- \frac{(x- k)^2}{t} \Big) \;.
\end{equation*}
We extend $K$ to negative times $t$ by defining it to be $0$ on  $((-\infty,0] \times \T^d) \setminus \{(0,0)\}$.  We will then view $K$ as smooth function on $\R \times \T^{d} \setminus \{0,0\}$. When $f$ is a space-time distribution and/or $u_{0}$ is a space distribution the right hand side of \eqref{e:Duhamel} is a formal expression but in many cases it can be made rigorous via the duality pairing between distributions and test functions.  More precisely, we say that $\xi$ is a distribution on  $(0,\infty) \times \T^d$ if it is a distribution on $\R \times \T^d$ which vanishes when tested against test-functions that are supported in $(-\infty, 0] \times \T^d$. Note that in general it is not possible, to multiply a distribution with the indicator function $\mathbf{1}_{(0,\infty)}(t)$, so that even in a distributional sense the integral over $(0,t)$ cannot always be defined 
(think e.g. of the distribution $\varphi \mapsto \mathrm{P.V.} \int \frac{1}{t} \varphi\, dt$  on $\R$).
However, for white noise $\xi$ it is easy to define  $\xi \mathbf{1}_{(0,\infty)}(t)$ as an element of $\Cc^{-\frac{d+2}{2} -\ka}$.

\medskip

To keep our exposition simple we will always assume that the initial condition $u_0$ is zero. We now give an important classical result (essentially a version of the parabolic Schauder estimates, see e.g. \cite{Krylov}, for a statement which implies our version see  \cite[Sec. 5]{RegStr}). In what follows $\Lambda_{t}$ denotes the domain of integration in \eqref{e:Duhamel}, that is $\Lambda_t := (0,t) \times \T^d$ and we use $\Lambda = (0,\infty) \times \T^d$. 

\begin{theorem}[Schauder Estimate]\label{Schaudestimate} For  $f$ in $\Cs^\alpha(\Lambda)$  define 
\begin{equation}\label{Duh}
u(s,x) :=   \int_{\Lambda_s} \,K(s-r,x-y)\; f(r,y) \, d y \; dr \;
\end{equation}
interpreted in a distributional sense if  $\alpha<0$. Then if $\alpha \notin \Z$, we have 
\begin{equation*}
\|u\|_{\Cs^{\alpha+2}(\Lambda_t)}  \ls \|f \|_{\Cs^{\alpha}(\Lambda_t)}\,.
\end{equation*}  
\end{theorem}

The Schauder estimate shows that the use of parabolic scaling of space-time is natural when measuring regularity. We do not give a proof of this result here; compare however to the 
discussion of the integration map for regularity structures in Section~\ref{sec: integration} below.

\medskip

We now apply Duhamel's principle to the stochastic heat equation \eqref{e:SHE} (again with vanishing initial condition). Formally the solution is given by
\begin{equation}\label{e:Z}
Z(t,x) = \int_{\Lambda_t} K(t-s,x-y) \; \xi(s,y) \,dy \, ds \;.
\end{equation}
The standard approach is to view $Z$ as a stochastic integral (the resulting object is called a stochastic convolution). However we can also define $Z$ determinstically for each fixed realization of white noise. Each such realization of white-noise will be an element of $\Cs^{-\frac{d}{2}-1-\ka}(\R \times \T^d)$, the Schauder estimate then implies that $Z \in \Cs^{1- \frac{d}{2}- \ka}(\R \times \T^d)$
for every $\ka>0$. It follows that $Z$ is a continuous function in $d=1$ while for $d \geq 2$ we expect $Z$ to be a distribution.

\medskip
Instead of using the Schauder estimate we can also get a handle on the regularity of $Z$ by establishing the estimate \eqref{regularitymomentbound} for $p=2$ (since $Z$ is Gaussian). This is an instructive calculation and it gives us a good occasion to introduce \emph{graphical notation} in familiar terrain. From now on we denote the process $Z$ introduced in \eqref{e:Z} by $\<1>$. This is to be read as a simple graph, where the circle at the top represents an occurrence of the white noise and the line below represents an integration against the heat kernel. As above we will  use the convention to combine the space and time variable into a single variable $z=(t,x)$. With these conventions testing $\<1>$ against the rescaled test-function  $\CS_z^\la \eta$,  defined as above in \eqref{e:scaledFunction}, yields  
\begin{equation}\label{e:Z2}
( \<1> , \CS_z^\la \eta)  =  \int_\Lambda   \int_{\Lambda}  \CS_z^\la \eta (z_1) \, K(z_1-z_2) \, dz_1  \;\xi(dz_2) \;. \end{equation} 
Then, using the characterising property \eqref{e:White_noise_covariance} of white noise we get
\begin{align}\label{e:integral11}
\E 
&\left[
( \<1> ,  \CS_{z}^\la \eta )^2 
\right] \notag\\
&=
\int_{\Lambda}  \int_{\Lambda}\int_{\Lambda} \CS_{z}^\la \eta(z_1) \CS_{z}^\la \eta(\bar{z}_1)  K(z_1 - z_2)  K(\bar{z}_1-z_2) \,  dz_2 \,dz_1 d\bar{z}_1 \; .
\end{align}

The only property of the kernel $K$ that enters our calculations is how quickly it grows near the origin. This motivates the following definition.

\begin{definition}\label{def: order}
Given a real number $\zeta$, we say a function $G: \R^{d+1} \setminus \{0\} \mapsto \R$ is a kernel of order $\zeta$ if
\[
\| G \|_{\zeta} 
:= 
\sup_{0 < |z|_{\s} < 1}
|G(z)| \times  \| z\|_{\mathfrak{s}}^{- \zeta} 
 < \infty.
\]
\end{definition}
Then one has the following easy lemma.
\begin{lemma}
The kernel $K(z)$ of \eqref{e:Z} and \eqref{e:Z2} is of order $-(d_{\mathfrak{s}}-2)$ where $d_{\mathfrak{s}}=d+2$ is the scaled dimension introduced above.
\end{lemma}

We now introduce a graphical representation of the \emph{integral}  \eqref{e:Z2} -- of course at this stage one could still evaluate the integral by hand easily, but this formalism becomes very convenient for more complicated integrals. In this graphical formalism we represent \eqref{e:integral11} by
\begin{equation}\label{e:graph1}
\E 
\left[
( \<1> ,  \CS_{z}^\la \eta )^2 
\right]= \quad
\begin{tikzpicture}[scale=0.5,baseline=-0.0cm]
\node at (-2.5,0) [var] (left){};
\node at (2.5,0) [var] (right) {};
\node at (0,0) [square](middle){};
\draw[generic] (left) to node[labl, anchor=north]{\tiny $-d_{\mathfrak{s}}+2$} (middle) ; 
\draw[generic] (middle) to node[labl, anchor=north]{\tiny $-d_{\mathfrak{s}}+2$} (right) ;

\end{tikzpicture}
\;.
\end{equation}
Again, each line represents an occurrence of the kernel $K$ and the order is denoted below. The black square in the middle represents an integration over the space-time $\Lambda$ and the grey vertices at the sides represent an integration against the scaled test-function $\CS_{z}^\la \eta$. Note that there is a simple ``graphical derivation" of \eqref{e:graph1} which consists of ``gluing" the dots in two copies of $\<1>$ together.

\medskip
The following lemma (essentially \cite[Lemma 10.14]{RegStr}) is simple but extremely helpful, because it permits to perform the analysis of integrals, which are potentially much more complicated than \eqref{e:integral11} on the level of graphs, by recursively 
reducing the complexity of the graph, keeping track only of the relevant information. 
\begin{lemma}
Let $K_1, K_2$ be kernels of order $\zeta_1, \zeta_2 \in \R$ with compact support. Then their product $K_1 K_2$ is a kernel of order $\zeta_1 + \zeta_2$ and we have $\| K_1 K_2\|_{\zeta_1 + \zeta_2} \leq \| K_1 \|_{\zeta_1}  \|K_2\|_{\zeta_2}$. If furthermore, 
\begin{equation}\label{e:condKernel}
\zeta_1 \wedge \zeta_2 > - d_{\mathfrak{s}} \qquad \text{ and} \qquad \zeta = \zeta_1 + \zeta_2 + d_{\mathfrak{s}} < 0 \;,
\end{equation} \refcomment{REF4} 
\refcomment{The second condition of \eqref{e:condKernel} is correct and agrees with \cite[Lemma 10.14]{RegStr}}
then $K_1 \ast K_2$ (where $\ast$ denotes the convolution on $\R \times \T^d$) is a kernel of order $\zeta$ and we have $\| K_1 \ast K_2\|_{\zeta } \ls \| K_1 \|_{\zeta_1}  \|K_2\|_{\zeta_2}$.
\end{lemma}
\begin{remark}\label{remark-kernelconvolutions}
The first condition in \eqref{e:condKernel} is necessary in order to ensure that the convolution $K_1 \ast K_2$ is well defined. The integration is restricted to a \emph{compact} space-time domain, so that we only have to deal with convergence at the singularities, but of course the constant depends on the choice of domain.
The second condition ensures that the resulting kernel actually does have a blowup at the origin. In the case $\zeta_1 + \zeta_2 +d_{\mathfrak{s}}=0$ it is in general \emph{not} true that $K_1 \ast K_2$ is bounded. However, one can obtain a bound with only a logarithmic divergence at the origin -- we will see that in our discussion of the two dimensional stochastic heat equation below. There is in general no reason to expect that  if $\zeta_1 + \zeta_2 +d_{\mathfrak{s}}>0$  we will have $K_1 \ast K_2 (0) =0$. In this case $K_{1} \ast K_{2}$ is not the correct object to work with, one must also subtract a partial Taylor expansion.
\end{remark}

We now apply this result to the integral over $z_2$ appearing in \eqref{e:integral11}. Note that if $\eta$ has compact support this integration is over a compact space-time domain (depending on the point $z$). For $d\geq 3$, (i.e. $d_{\mathfrak{s}} \geq 5$) condition \eqref{e:condKernel} is satisfied and  we can replace the convolution of the kernels by a single kernel of order $-d_{\mathfrak{s} }+4$.  In our convenient graphical notation this can be written as
\begin{equation}\label{e:graph2}
\begin{tikzpicture}[scale=0.5,baseline=-0.0cm]
\node at (-2.5,0) [var] (left){};
\node at (2.5,0) [var] (right) {};
\node at (0,0) [square](middle){};
\draw[generic] (left) to node[labl, anchor=north]{\tiny $-d_{\mathfrak{s}}+2$} (middle) ; 
\draw[generic] (middle) to node[labl, anchor=north]{\tiny $-d_{\mathfrak{s}}+2$} (right) ; 
\end{tikzpicture}
\quad \ls \quad
\begin{tikzpicture}[scale=0.5,baseline=-0.0cm]
\node at (-1.5,0) [var] (left){};
\node at (1.5,0) [var] (right) {};
\draw[generic] (left) to node[labl, anchor=north]{\tiny $-d_{\mathfrak{s}}+4$} (right) ; 
\end{tikzpicture} \;.
\end{equation}
At this stage it only remains to observe that for $\eta$ with compact support and for $\zeta > - d_{\mathfrak{s}}$
\begin{align*}
\int_{\Lambda} \int_\Lambda   \CS_{z}^\la \eta (z_1) \; \CS_{z}^\la \eta (\bar{z}_1) \|z_1 - \bar{z}_1\|_{\mathfrak{s}}^{\zeta} d z_1 \, d\bar{z}_1\ls \lambda^{\zeta} \;,
\end{align*}
and we have derived \eqref{regularitymomentbound} and therefore  the right regularity of $\<1>$ at least for $d\geq3$. For $d=2$ we are in the critical case $-d_{\mathfrak{s}} +4 =0$. According to Remark~\ref{remark-kernelconvolutions} the inequality \eqref{e:graph2} remains valid if we interpret a kernel of order $0$ as a logarithmically diverging kernel.

\medskip

For $d=1$ condition \eqref{e:condKernel} fails and we cannot use the same argument to derive the desired $\frac12 -$ regularity. This is due to the fact that in order to obtain \emph{positive} regularity, $( \<1> ,  \CS_{z}^\la \eta )$ is plainly the wrong quantity to consider. As observed in Definition~\ref{positiveholder2} rather than bounding the blowup of  local averages of $\<1>$ near $z$ we need to control how fast these local averages go to zero if a suitable \emph{polynomial approximation} (the Taylor polynomial) is subtracted. In the  case of $\<1>$ we aim to show $\frac12 -$ regularity, so we need to control how quickly $( \<1> - \<1>(z) ,  \CS_{z}^\la \eta )$ goes to zero for small $\lambda$. This observation may seem harmless, but we will encounter it again and again below. Arguably much of the complexity of the theory of regularity structures is due to the extra terms we encounter when we want to obtain bounds  on a quantity of \emph{positive} regularity (or \emph{order}). 
In this particular case it is not too difficult to modify the graphical argument to get a bound on $( \<1>- \<1>(z) ,  \CS_{z}^\la \eta )$. The integral \eqref{e:integral11} turns into
\begin{align*}
\E 
&\left[
( \<1> - \<1>(z) ,  \CS_{z}^\la \eta )^2 
\right] \
=
\int_{\Lambda}  \int_{\Lambda}\int_{\Lambda} \CS_{z}^\la \eta(z_1) \CS_{z}^\la \eta(\bar{z}_1) \\
& \qquad \qquad \times \big( K(z - z_2) -  K(z_1 - z_2)\big) \, \big(  K(z- z_2) -   K(\bar{z}_1-z_2)\big) \,  dz_2 \,dz_1 d\bar{z}_1 \; .
\end{align*}
Now we need to use the fact that not only $K$ has the right blowup near the diagonal but also its derivatives. More precisely, for every multi-index $k$ we have that  
\begin{align*}
|D^k K(z) | \ls \|z\|_{\mathfrak{s}}^{-d_{\mathfrak{s}} +2 - |k|_{\mathfrak{s}} } \;.
\end{align*}
In fact, these additional bounds are imposed in the version of Definition \ref{def: order} found in \cite{RegStr} and also appear in some statements of harmonic analysis relating to singular kernels. In \cite[Lemma 10.18]{RegStr} it is shown how the kernel $ K(z - z_2) -  K(z_1 - z_2)$ can be replaced by 
a ``Taylor approximation" $ DK(z_1 - z_2)(z-z_1) $. The factor $(z-z_1)$ can then be pulled out of the convolution integral over $z_2$ and the ``graphical algorithm" can be applied to the convolution of two copies of $DK$ which \emph{do} satisfy \eqref{e:condKernel}.

\subsection{Nonlinear equations}

For non-linear equations Duhamel's principle turns into a fixed point problem. We illustrate this for equation \eqref{e:AC} in one spatial dimension where one gets
\begin{align}
\phi(t,x) = &\int_0^t \int_{\T^1} K(t-s,x-y) \; \xi(s,y) \,dy \, ds \notag\\
& -  \int_0^t \int_{\T^1} K(t-s,x-y) \; \phi^3(s,y) \,dy \, ds \;. \label{e:Picard1}
\end{align} 
For simplicity we have dropped the mass term $m^2 \phi$ and set the initial condition to be $0$. The Schauder estimate tells us that the first term on the right hand side of \eqref{e:Picard1} is in $\Cs^{\frac12  -\ka}(\R_+ \times \T^1)$ for any $\ka$. The following theorem characterizes when we can understand products like $\phi^{3}$ classically  \refcomment{REF5} 
\def\MultValid{\cite[Prop 4.11]{RegStr}}
\def\BCD{\cite[Sec 2.6]{BCD}}
\begin{theorem}[\MultValid, see also \BCD]\label{multiplication}
Suppose that $\alpha  + \beta > 0$, then there exists a bilinear form $B(\cdot,\cdot):\Cs^{\alpha} \times \Cs^{\beta} \rightarrow \Cs^{\alpha \wedge \beta}$ such that 

\begin{itemize}
\item For smooth functions $f,g$ one has that $B(f,g)$ coincides with the point-wise product of $f$ and $g$.

\item For arbitrary $f \in \Cs^{\alpha}$, $g \in \Cs^{\beta}$ one has
\[
||B(f,g)||_{\Cs^{\alpha \wedge \beta}} 
\lesssim
||f||_{\Cs^{\alpha}}
\times
||g||_{\Cs^{\beta}}.
\] 
\end{itemize}

Additionally, if $\alpha + \beta \le 0$ then no bilinear form $B(\cdot,\cdot):\Cs^{\alpha} \times \Cs^{\beta} \rightarrow \Cs^{\alpha \wedge \beta}$ satisfying both of the above statements exists.
\end{theorem}

It is then natural to treat \eqref{e:Picard1} as a fixed point problem in $\Cs^{\frac12  -\ka}(\R_+ \times \T^1)$ - by Theorem \ref{multiplication} the definition of $\phi^3$ poses no difficulty in this space. For any fixed realisation of $\xi$ there exists $T(\xi) > 0$ such that the operator 
\begin{align}
\Psi\ \colon\ \varphi \mapsto \int_0^{\cdot} \int_{\T^1} & K(\cdot-s,\cdot-y) \; \xi(s,y) \,dy \, ds \notag\\
& -  \int_0^{\cdot} \int_{\T^1} K(\cdot-s,\cdot-y) \; \phi^3(s,y) \,dy \, ds \label{e:Picard2}
\end{align}
is a contraction on bounded balls in  $\Cs^{\frac12  -\ka}([0,T] \times \T^1)$. An important observation is that $v = \phi - \<1> = -  \int_0^t \int_{\T^1} K(t-s,x-y) \; \phi^3(s,y) \,dy \, ds$ is much more regular than $\phi$ itself, in fact the  Schauder estimate implies that it is  a $\Cs^{\frac{5}{2} -\ka}$ function.
It is important to note that this argument does \textit{not} make use of the sign of the nonlinear term $- \phi^3$. Of course, this sign is essential when deriving bounds that imply non-explosion,  the existence of invariant measures for solutions, or even getting existence and uniqueness when $\T^{1}$ is replaced by $\R$.

\medskip

For $d \geq 2$ it is not so easy to solve the fixed point problem \eqref{e:Picard1} (with the one-dimensional torus $\T^1$ replaced by $\T^d$). As we have seen above the stochastic convolution $\<1>$ only takes values in the distributional spaces $\Cs^{-\frac{2-d }{2}-\ka}$ but there is no canonical way to define the mapping $\phi \mapsto \phi^3$ for $\phi \in \Cs^{\alpha}$ with $\alpha < 0$. 
We will now try to find a way around this issue in the  case of $d\geq 2$, we start by running a Picard iteration step by step. More precisely we set $\phi_0 =0$ and aim to study the behaviour of the sequence $\{\phi_n\}_{n=0}^{\infty}$ defined recursively as 
\begin{align*}
\phi_{n+1} = \Psi(\phi_n) \;,
\end{align*}
where $\Psi$ is defined in \eqref{e:Picard2} (with $\T^1$ replaced by $\T^d$). 

\medskip
With our choice of $\phi_0 =0$ the first step in the Picard iteration yields $\phi_1 = \<1>$ which is of regularity $\Cs^{\frac{2-d}{2} -\kappa}$. When going to $\phi_{2}$ we immediately run into trouble when we try to apply $\Psi$ to $\<1>$ since this requires us to define $\<1>^{3}$ for which Theorem \ref{multiplication} is of no use.

So far our analysis of \eqref{e:Picard2} could be performed entirely deterministically (occuring for a fixed realization of $\xi$) but at this point it becomes essential to use a probablistic approach. While there is no canonical way of cubing an arbitrary distribution of negative regularity,  we will now see that there are ways to define polynomials in $\<1>$ by exploiting its Gaussian structure.

\subsubsection{Construction of Wick powers} 

We will define $\<1>^3$ by approximation. The calculations will be performed in the framework of  \emph{iterated stochastic integrals}. Definition and elementary properties of these are recalled in Appendix~\ref{s:AppA}.  Let $\rho_\delta$ be  a smoothing kernel on scale $\delta$ (as was used in \eqref{e:eta_delta}) and set
 \begin{align}\label{e:int1}
 \<1>_\delta(z) := \<1> \star \rho_\delta(z) = \int_{\Lambda} K \star \rho_\delta(z - \bar{z}) \; \xi (d \bar{z})  \;.
 \end{align}
 For every $\delta >0$ the random function $\<1>_\delta(z)$ is smooth and we can define $\<1>_\delta(z)^3$ without ambiguity. To analyse the behaviour of $\<1>_\delta^3$ as $\delta \to 0$ we interpret $\<1>_{\delta}$ as a stochastic integral  against $\xi$ and apply \eqref{e:True11} which gives 
 \refcomment{REF6 - Made notation consistent, only $\rho_{\delta}$ appears, never a $\eta_{\delta}$}
 \begin{align}\label{Zdeltacubed}
 \<1>_\delta(y)^3  =& \int \CW^{(3)}_\delta(y; z_1, z_2, z_3) \xi(dz_1) \, \xi(dz_2) \, \xi(dz_3) \notag\\
 &+ 3 \int \CW^{(1)}_\delta(y; z_1) \, \xi(dz_1) \,, 
 \end{align} 
where 
\begin{align*}
\CW^{(3)}_\delta(y; z_1, z_2, z_3) = 
\prod_{j=1}^{3}
\left[
K\star \rho_\delta(y- z_j )
\right],
\end{align*}
and 
\begin{align}\label{e:W1}
\CW^{(1)}_\delta(y; z_1) = K\star \rho_\delta(y-z_1 ) \,\int  \big( K \star \rho_\delta (y-z) \big)^2  \; dz .
\end{align}
As before, we will introduce a graphical notation to analyse these expressions. In this notation \eqref{Zdeltacubed} becomes
\begin{align*}
 \<1>_\delta(y)^3  =& \<3>_\delta + 3 
 \quad
\begin{tikzpicture}[scale=0.25,baseline=-0.0cm]
\node at (0,1) [square] (above){};
\draw[generic, bend left = 60] (above) to  (0,-.8)  to (above); 
\end{tikzpicture}_\delta
\quad 
  \<1>_\delta \;.
 \end{align*} 
As before, each black dot represents an occurrence of the space-time white noise, and each line  represents an integration against a singular kernel. The black square appearing in the second term 
is a dummy variable which is integrated out.
The subscript $\delta$ appearing in all the graphs represents the fact that all singular kernels are regularised at scale $\delta$, i.e. $K$ is replaced by $K \star \rho_\delta$ which satisfies
\begin{align*}
|K\star \rho_\delta (z) |\ls \frac{1}{(\|z\|_{\mathfrak{s}}+\delta) ^{d_{\mathfrak{s}}-2}} \;.
\end{align*}

Applying the same graphical rule as above, we get
\begin{equation}\label{graphAnalysis}
\E 
\left[
( \<3>_\delta ,  \CS_{z}^\la \eta )^2 
\right]= \quad
\begin{tikzpicture}[scale=0.5,baseline=-0.0cm]
\node at (-2.5,0) [var] (left){};
\node at (2.5,0) [var] (right) {};
\node at (0,0) [square](middle){};
\node at (0,1.5) [square](up){};
\node at (0,-1.5) [square](down){};

\draw[generic] (left) to node[labl, anchor=north]{\tiny $-d_{\mathfrak{s}}+2$} (middle) ; 
\draw[generic] (middle) to node[labl, anchor=north]{\tiny $-d_{\mathfrak{s}}+2$} (right) ;

\draw[generic] (left) to node[labl, anchor=south]{\tiny $-d_{\mathfrak{s}}+2$} (up) ; 
\draw[generic] (up) to node[labl, anchor=south]{\tiny $-d_{\mathfrak{s}}+2$} (right) ; 

\draw[generic] (left) to node[labl, anchor=north]{\tiny $-d_{\mathfrak{s}}+2$} (down) ; 
\draw[generic] (down) to node[labl, anchor=north]{\tiny $-d_{\mathfrak{s}}+2$} (right) ;

\end{tikzpicture}
\;\quad \ls \quad 
\begin{tikzpicture}[scale=0.4,baseline=-0.0cm]
\node at (-2,0) [var] (left){};
\node at (2,0) [var] (right) {};

\draw[generic] (left) to node[labl, anchor=north]{\tiny $-d_{\mathfrak{s}}+4$} (right) ;

\draw[generic, bend left = 60] (left) to node[labl, anchor=south]{\tiny $-d_{\mathfrak{s}}+4$} (right) ; 

\draw[generic,  bend right =60] (left) to node[labl, anchor=north]{\tiny $-d_{\mathfrak{s}}+4$} (right) ;

\end{tikzpicture}
\quad \ls \quad 
\begin{tikzpicture}[scale=0.5,baseline=-0.0cm]
\node at (-1.5,0) [var] (left){};
\node at (1.5,0) [var] (right) {};

\draw[generic] (left) to node[labl, anchor=north]{\tiny $-3d_{\mathfrak{s}}+12$} (right); 

\end{tikzpicture} \;,
\end{equation}
uniformly in $\delta$. For $d=3$, i.e. for $d_{\mathfrak{s}}=5$, we have $-3d_{\mathfrak{s}}+12 = 3 > -d_{\mathfrak{s}}$, which yields  the uniform-in-delta bound 
\begin{align*}
\E 
\left[
( \<3>_\delta ,  \CS_{z}^\la \eta )^2 
\right] \ls \lambda^{-3} \;,
\end{align*} 
while in the case $d=2$ we get as above
\begin{align*}
\E 
\left[
( \<3>_\delta ,  \CS_{z}^\la \eta )^2 
\right] \ls |\log(\lambda)|^{3}.
\end{align*}

However the lower order \textit{It\^o correction} $3  \int \CW^{(1)}_\delta(y; z_1) \, \xi(dz_1) $ will be a problem in the $\delta \downarrow 0$ limit. The explicit form \eqref{e:W1} of the kernel $\CW^{(1)}_\delta$ shows that it can be rewritten as $3 C_\delta Z $ where $C_{\delta}$ \rf{is a constant} given by \refcomment{REF7 - See colored text, also footnote}
\begin{align}\label{wicksubtract}
C_\delta := \int  \big( K \star \rho_\delta (z) \big)^2  \; dz\;.
\end{align}
For $\delta \downarrow 0$ these $C_\delta$ diverge logarithmically for $d=2$ and like $\frac{1}{\delta}$ for $d=3$\footnote{Actually, there is a slight cheat in  \eqref{wicksubtract} because we do not specify the domain of integration. In each case $C_\delta$ does not depend on the spatial variable $y$, but if we define $C_\delta$ as an integral over $\Lambda_t$ then it actually depends on $t$ which one may consider ugly. But the integral over $\Lambda_t$ can be decomposed into a part which does not depend on $t$ and which diverges as $\delta \to 0$ (e.g. the integral over $[0,1] \times \T^d$) and a part which depends on $t$ but remains bounded in $\delta$ and which can be ignored in the renormalization procedure. There are many ways to choose $C_\delta$ in a $t$-independent way. None of these choices is canonical but all only differ by a quantity that remains bounded as $\delta \to 0$.  }. To overcome this problem we simply remove the diverging term $3 C_\delta Z$. From our second moment bound, the Nelson estimate (see \eqref{e:Nelson} in Appendix~\ref{s:AppA}), and Theorem \ref{Kolmogorov} one can then show that the limit
\begin{align*}
\<3>\ :=
\lim_{\delta \downarrow 0}
\left(
\<1>_\delta^3 - 3 C_\delta \<1>_\delta
\right)
% =
%\lim_{\delta \downarrow 0} 
%\int \CW^{(3)}_\delta(y; z_1, z_2, z_3) \xi(dz_1) \, \xi(dz_2) \, \xi(dz_3) \;.
\end{align*}
exists as random elements of $\Cs^{-\frac32 -\ka}$ for $d=3$ and as random elements of $\Cs^{-\ka}$ for $d=2$, where the convergence holds for every stochastic $L^p$ space. The subtraction implemented above is called \emph{Wick renormalization}, and the object $\<3>$ is called the third Wick power of $\<1>$ and is sometimes denoted by  $\colon Z^3 \colon $  (we could write $: \<1>^3:$ to be more consistent with our graphical notation). 

\medskip
The general recipe for defining Wick powers is as follows: to define $\colon Z^{n} \colon$ one applies the $n$-th order analog of the identity \eqref{e:True11} to $Z_{\delta}^{n}$, drops \emph{all} lower-order It\^o corrections \footnote{By lower order we mean all the terms involving strictly less than $n$ factors of $\xi$.} to get an object we denote $:Z_{\delta}^{n}:$, and then takes the $\delta \downarrow 0$ limit. The graphical analysis of $n$-th Wick powers is very similar to \eqref{graphAnalysis}, the only difference being that there are $n$ edges connecting the left and right vertices. 
In this way one obtains a singular kernel of order $-nd_{\mathfrak{s}}+4n$ as a final result. Hence for $d=2$ the blowup of this  kernel on the diagonal can be bounded by $|\log(z_1 - \bar{z_1})|^n$ which is integrable for every $n$. For $d=3$ however we get a polynomial blowup $|z_1 - \bar{z}_1|^{-n}$ which fails to be integrable for $n\geq 5$. For $d=2$ we can define arbitrary Wick powers of $Z= \<1>$ while for $d=3$ we can only define Wick powers up to $n=4$.

\begin{remark}
Our  reasoning shows that  in the three dimensional case we can define Wick powers up to order $n=4$ as space-time distributions. It is however \emph{not} possible to evaluate these distribution for fixed $t$ in the cases $n\geq 3$. Only space time averages are well defined.
\end{remark}

\subsubsection{Back to the Picard iteration}
We now return to our Picard iteration, \rf{still working formally}. 
The process $\colon \<1>^3 \colon$ is denoted by $ \<3>$, where again each dot represents an occurrence of white noise and each line represents one integration against a kernel. The fact that they are merged at the bottom corresponds to multiplication. For now we will just replace the $Z^{3}$ that would have appeared in $\phi_{2}$ with $\<3>$ so that we have
\begin{align*}
\phi_2 = \<1> -  \<30> \;
\end{align*}
where $\<30> = K \ast \<3>$. In the next step of the Picard iteration we would get
\begin{align*}
\phi_3 = \<1> - K \star\Big( \<3> - 3 \<1>^{2}  \, K \ast (\<3>)  + 3 \<1> \big( K \ast \<3>\big)^2 - \big( K \ast \<3>\big)^3  \Big) \,.
\end{align*}
\refcomment{REF8}
\rf{If we restrict to the case $d=2$} then almost all of these terms are well defined. Indeed, according to the Schauder estimates $K \star \<3>$ is a function of class $\Cs^{2- \ka}$ for any $\ka>0$. And this is enough to define most of the products. The only term that causes a problem is 
the term $\<1>^{2} = Z^2$, however the corresponding Wick power $\<2> := \lim_{\delta \downarrow 0} (\<1>_\delta^2 -C_\delta)$ is well defined.

\medskip
It turns out that these are all the terms that need to be renormalised when $d=2$, and that after modifying these first few steps, the Picard iteration can actually be closed. Of course we have been working somewhat formally here, instead of replacing certain powers $Z^{n}$ with Wick powers $\colon Z^{n} \colon$ one should instead modify the equation so it automatically generates the needed Wick renormalizations. In the next section we will explain in more detail, how the above method of treating $\Phi^{4}_{2}$ can be implemented and we explain why this approach fails for $\Phi^4_3$.

\section{Lecture 3} 
\label{s:l3}
The renormalization and  the Picard iteration for $\Phi^{4}_{2}$ were performed in a very elegant way in \cite{DPD03} by a method we call the Da Prato - Debussche argument. We start this lecture by discussing this argument and then sketch why it fails for $\Phi^{4}_{3}$. This motivates us to turn to a more robust approach, the theory of regularity structures \cite{RegStr}.
In particular we will introduce some of basic objects of the theory: \textit{regularity structures}, \textit{models}, and \textit{modelled distributions}. 

 \subsection{The Da Prato - Debussche Argument}\label{dpdsec}
 At the end of Lecture 2 we expanded the solution $\phi$ of the $\Phi_{2}^{4}$ equation in terms of objects built out of the linear solution $\<1>$. A key observation is that the most singular term in our partial expansion of $\phi$ was $\<1>$, if we write $\phi = \<1> + v$ then we expect the remainder $v$ to be of better regularity.
 While we are unable to directly treat $\Phi^{4}_{2}$ equation as fixed point problem in a $\Cs^{\alpha}$ space, it turns out that one can renormalize the original equation so that it generates Wick powers of $\<1>$ and then solve a fixed point equation in a nicer space for the remainder $v$. As we already announced in \eqref{e:less_naive1}, the renormalized equation is
\begin{equation}\label{e:wickrenormalizedphi4}
\partial_{t} \phi_{\delta} =
\Delta \phi_{\delta}
-
(\phi_{\delta}^{3}
-
3C_{\delta}
\phi_{\delta}
)
+\xi_{\delta},
\end{equation}
where $C_{\delta}$ is given by \eqref{wicksubtract}. Now we write $\phi_{\delta} = \<1>_{\delta} + v_{\delta}$ where $\<1>_{\delta}$ is given by \eqref{e:int1} so that it solves $\partial_{t} \<1>_{\delta} = \Delta \<1>_{\delta} + \xi_{\delta}$. Subtracting this linear equation from \eqref{e:wickrenormalizedphi4} gives us \refcomment{REF9}
\begin{equation}\label{dapratoremainder}
\begin{split}
\partial_{t} v_{\delta}
=&
\Delta v_{\delta}
-
\left(
(v_{\delta} + \<1>_{\delta})^{3} 
-
3C_{\delta}(v_{\delta} + \<1>_{\delta})
\right)\\
=&
\Delta v_{\delta}
-
v_{\delta}^{3}
-
3
\<1>_{\delta}
\rf{v_{\delta}^{2}}
-
3
(
\<1>_{\delta}^{2}
-
C_{\delta}
)
v_{\delta}
-
(\<1>_{\delta}^{3}
-
3C_{\delta}
\<1>_{\delta}
).
\end{split}
\end{equation}
This equation looks more promising since the rough driving noise $\xi_{\delta}$ has dropped out and from the previous lecture we know that the polynomials in $\<1>_{\delta}$ appearing above converge in probability to the corresponding Wick powers as $\delta \rightarrow 0$. We pass to the limit and try to solve the fixed point equation
\begin{equation}\label{dapratofixedpointeqn}
v
=
K \star
\left[
-
v^{3}
-
3
\<1>
v^{2}
-
3
\<2>
v
-
\<3>
\right].
\end{equation}
Recall that $\<1>, \<2>,$ and $\<3>$ are in $\Cs^{-\kappa}$ for any (small) $\kappa>0$, when $d=2$. Using Theorem \ref{Schaudestimate} and Theorem \ref{multiplication} we can formulate \eqref{dapratofixedpointeqn} in $\Cs^{2 - \kappa}$, the key point being that all products on the right hand side of \eqref{dapratofixedpointeqn} make sense in $\Cs^{2 - \kappa}$. 
By exploiting the sign of $v^{3}$ in \eqref{dapratoremainder} one can also show global in time existence for $v$ (and therefore for $\phi$ as well), see \cite{JCH2}.
\refcomment{We added this remark.}
\begin{remark}
Remarkably a similar argument was originally discovered by Bourgain in the context of the two dimensional non-linear Schr\"odinger equation with 
defocussing cubic non-linearity. More precisely, in \cite{bourgain} Bourgain studied the deterministic PDE
\begin{align*}
i \partial_t \phi  = \Delta \phi - \phi^3.
\end{align*}
When written in the mild form 
\begin{align}\label{mild-Schrodinger}
\phi(t) = e^{-i \Delta t} \phi(0) - \int_0^t e^{-(t-s) i \Delta} \phi^3(s) ds
\end{align}
it resembles \eqref{e:Picard1} with the important difference that unlike the heat semigroup the Schr\"odinger semigroup $e^{-i t \Delta}$ does not improve differentiability.
Bourgain studied \eqref{mild-Schrodinger} when the initial datum $\phi(0)$ is a complex Gaussian free field on the torus in which case $z(t) = e^{-i \Delta t} \phi(0)$ is a Gaussian 
evolution with regularity properties identical to those of the process $\<1>$. He then performed the same Wick renormalisation for the square and the cube of $z(t)$
and showed that the equation for the remainder $v = \phi - z$  can be solved as a fixed point problem in a space of function of positive differentiability. This is a remarkable result 
because, as said above, the Schr\"odinger semigroup does not usually improve regularity. See e.g. \cite{AndreaGigliola} for  recent work in this direction.
\end{remark}

\medskip

The above argument does not apply for $\Phi^{4}_{3}$. In this case $\<3> \in \Cs^{-3/2 - \kappa}$ so we expect the remainder $v$ to be no better than $\Cs^{1/2 - \kappa}$. Since $\<2> \in \Cs^{-1 - \kappa}$ we fall far short of the regularity needed to define the product $\<2>v$.
One might try to defeat this obstacle by pushing the expansion further, writing $\phi = \<1> -  \<30> + v$ and solving for $v$. The new fixed point equation is
\begin{equation*}
v
=
K \star
\left[
-
v^{3}
-
3
\<1>
v^{2}
-
3
\<2>
v
-
3
( \<30>)^{2}
v
-3
( \<30>)v^{2}\\
-
6
\<1> 
( \<30>)
v
- 3 \<2> ( \<30>) - 3 \<1> ( \<30>)^{2}.
\right]
\end{equation*}
%where we have set $R = 3 \<2> ( \<30>) + 3 \<1> ( \<30>)^{2}$. 
Since we have pushed the expansion further we do not see the term $\<3>$ anymore. \rf{However we are now confronted} with the product $\<2>  \<30>$ which cannot be defined using Theorem~\ref{multiplication} since $ \<30> \in \Cs^{1/2 - \kappa}$ and $\<2> \in \Cs^{-1 - \kappa}$. In fact, this ill-defined product is the reason for the second logarithimically diverging renormalization constant for $\Phi^{4}_{3}$. 
\rf{But after defining} this product by inserting another renormalization constant by hand, we are still unable to close the Picard iteration. \refcomment{REF10}
The real problematic term is the product $\<2> v$ which creates a vicious circle of difficulty. If we could define the product $\<2> v$ then it would have regularity $\Cs^{-1 - \kappa}$, this means at best one could have $v \in \Cs^{1 - \kappa}$. However, this is not enough regularity to define the product $\<2> v$ and so we are unable to close the fixed point argument.

\subsection{Regularity Structures}
The Da Prato - Debussche argument for $\Phi^{4}_{2}$ consisted of using stochastic analysis to control a finite number of explicit objects built out of the linear solution followed by the  application of a completely deterministic fixed point argument in order to solve for a relatively smooth remainder term.
For $\Phi^{4}_{3}$ we saw that regardless of how far one expands $\phi$, writing 
\begin{equation}\label{perturbexp}
\phi = \<1> - \<30>  + \dots + v,
\end{equation}
the product $\<2>v$ always prevents us from formulating a fixed point argument for $v$. We cannot make the remainder $v$ arbitrarily smooth just by pushing the expansion further.

\medskip
In the theory of regularity structure we will again postulate an expansion for $\phi$ which looks more like
\begin{equation}\label{roughreg}
\Phi(z) 
= 
\Phi_{\<s1>}(z)\<s1> 
+ \Phi_{\<s30>}(z)\<s30> + \dots 
+
\Phi_{\symbol{\mathbf{1}}}(z) \symbol{\mathbf{1}}.
\end{equation}
One immediately visible difference is that the expansion \eqref{roughreg} allows varying coefficients in front of various stochastic objects. Instead of solving a fixed point equation for a single function $v$, we will instead solve a fixed point equation for a family of functions $(\Phi_{\<s1>}, \Phi_{\<s30>},\dots,\Phi_{\symbol{\mathbf{1}}})$. We will also be interested in something called
the ``order''\footnote{What we call order is referred to as homogeneity by  Hairer in \cite{hairer2015introduction} and \cite{RegStr}.}
of objects $\<s1>,\dots,\<s30>, \dots$ in \eqref{roughreg} instead of just their regularity. The order of an object describes how much we expect it to vanish or blow up when evaluated at small scales, one of the main goals of this section is to clarify the concept of ``order''. Finally, while the objects $\<s1>,\dots,\<s30>, \dots$ appearing in \eqref{roughreg} are related to the corresponding stochastic objects in \eqref{perturbexp}, they will turn out to be a totally different sort of object so we have distinguished them by coloring them blue.  

\medskip
In \cite{RegStr}  the fixed point problem associated with SPDE
is solved in an abstract setting. The rest of this lecture will be devoted to introducing the hierarchy of objects that populate this abstract setting and we begin by defining the most basic object.
\def\RegStrDef{\cite[Def. 2.1]{RegStr}}
\begin{definition}[\RegStrDef]\label{Regstrucdef}
A regularity structure $\mathcal{T}$ consists of a triple $(A,T,G)$.
\begin{itemize}
\item $A \subset \R$ is an indexing set which is bounded from below and has no accumulation points.
\item $T = \bigoplus_{\alpha \in A} T_{\alpha}$ is a graded vector space where each $T_{\alpha}$ is a finite dimensional\footnote{Actually, in \cite{RegStr} these spaces are note required to be finite-dimensional, but in most examples we are aware of they are even of very low dimension.} real vector space which comes with a distinguished basis and norm.
\item $G$ is a family of linear transformations on $T$ with the property that for every $\Gamma \in G$, every $\alpha \in A$, and every $\tau \in T_{\alpha}$ one has
\begin{equ}[e:basicRel]
\Gamma \tau - \tau \in T^{-}_{\alpha},\;
\end{equ}
where we have set
\[
T^{-}_{\alpha} 
:= 
\bigoplus_{\beta < \alpha} T_{\beta}.
\]
Additionally we require that $G$ form a group under composition.
\end{itemize}
\end{definition}

In the triple $(A,T,G)$ the set $A$ is an indexing set that lists the orders of the objects that we allow to appear in our expansions. We will always assume $0 \in A$. For any $\alpha \in A $ an element $\tau \in T_{\alpha}$ should be thought of as an abstract symbol that represents an object of order $\alpha$ - for such a ``homogenous'' element $\tau$ we write $|\tau| = \alpha$. We denote by $||\cdot||_{\alpha}$ the norm on $T_{\alpha}$
\footnote{Since all norms on such $T_{\alpha}$ are equivalent we may not fix a specific one when defining a regularity structure.}.
For general $\tau \in T$ we set $||\tau||_{\alpha} := ||\mathcal{Q}_{\alpha} \tau ||_{\alpha}$ where $\mathcal{Q}_{\alpha}:T \rightarrow T_{\alpha}$ is just projection onto the $\alpha$-component.

\medskip
Returning to \eqref{roughreg}, the objects $\<s1>$ and $\<s30>$ no longer represent fixed space-time distributions but instead are abstract symbols which are homogenous elements of $T$. The object $\Phi$ in \eqref{roughreg} is actually a map $\Phi:\R^{d+1} \rightarrow T$.
The family of linear transformations $G$, called the \emph{structure group}, will play an important role in the theory but we will introduce it slowly as we introduce examples of increasing complexity \footnote{ In practice we will not explicitly define the \emph{entire} structure group $G$ when we encounter more complex regularity structures $\mathcal{T}$, only a small subgroup germane to our discussion.}.

\subsection{An abstract generalization of Taylor expansions}\label{abstractpoly}
While \eqref{perturbexp} is a perturbative expansion generated by Picard iteration, one should think of \eqref{roughreg} as a  jet
\footnote{More specifically, a collection of Taylor expansions indexed by space-time ``base-points''.},
at each space-time point this expansion represents the solution as a linear combination of objects that vanish (or blow up) at controlled rates when evaluated near that space-time point. We will now show how the actual Taylor expansions familiar to a calculus student can be formulated in the theory of regularity structures.

\medskip
We claim that the statement a function $f:\R^{d+1} \rightarrow \R$ belongs to $\Cs^{\alpha}(\R^{d+1})$ for some $\alpha > 0$ is equivalent to requiring that (i) for any multi-index $j$ with $|j|_{\mathfrak{s}} \le \alpha$, $D^{j}f$ exists and is continuous, and (ii) for every $z \in \R^{d+1}$ one has the bound

\begin{equation}\label{e:Lipschitz1}
\sup_{ 
\substack{
 \bar{z} 
 \\ ||\bar{z} - z||_{\mathfrak{s}} \le 1
 }
 }
\left| f(\bar{z}) -  \sum_{|k|_{\mathfrak{s}} \leq \alpha} \frac{1}{k!}D^{k}f(z) (\bar{z}-z)^{k}\right|  \le C ||\bar{z} - z||_{\mathfrak{s}}^{\alpha}\;.  
\end{equation}
It is not hard to check that together the conditions (i) and (ii) are equivalent to the requirements of Definition \ref{positiveholder2}. Moreover, estimate \eqref{e:Lipschitz1} implies that for any multi-index $j$ with $|j|_{\mathfrak{s}} \le \alpha$ one has the bound
\begin{equation}\label{e:Lipschitz2}
\sup_{ 
\substack{
 \bar{z} 
 \\ 0 < ||\bar{z} - z||_{\mathfrak{s}} \le 1
 }
 }
\left| D^{j}f(\bar{z}) -  \sum_{|k|_{\mathfrak{s}} \leq \alpha - |j|_{\mathfrak{s}}} \frac{1}{k!}D^{j + k}f(z) (\bar{z}-z)^{k}\right|  \le C ||\bar{z} - z||_{\mathfrak{s}}^{\alpha - |j|_{\mathfrak{s}}}\;.
\end{equation}

Our point is that the statement $f \in \Cs^{\alpha}(\R^{d+1})$ corresponds to the existence of a family of polynomials indexed by $\R^{d+1}$ which do a sufficiently good job of describing $f$ locally. To implement this in our setting we will define a regularity structure, denoted \rf{$\bar{\mathcal{T}}$}, which we call the regularity structure of abstract polynomials. 
\refcomment{REF11}
\medskip
More precisely, \rf{$\bar{\mathcal{T}} = (A,T,G)$} where $A = \mathbf{N}$ and $T$ is the algebra of polynomials in the commuting indeterminates $\symbol{\mathbf{X}_{0}},\symbol{\mathbf{X}_{1}}, \dots, \symbol{\mathbf{X}_{d}}$. We write $\symbol{\mathbf{X}}$ for the associated $(d+1)$-dimensional vector indeterminant. For any $l \in \mathbf{N}$ we set $T_{l}$ to be the Banach space whose basis is the set of monomials $\Xk$ of parabolic degree $l$ (i.e. $|k|_{\mathfrak{s}} = l$). For a general $\tau \in T$ and monomial $\Xk$ we denote by $\langle \tau, \Xk \rangle$ the coefficient of $\Xk$ in the expansion of $\tau$. We will explicitly describe the structure group $G$ for $\bar{\TT}$ a little later.

\medskip

Given any $f \in \Cs^{\alpha}$ we can associate with it the function $F:\R^{d+1} \rightarrow T$ given by
\begin{equation}\label{liftofholder}
F(z) = 
\sum_{ |k|_{\mathfrak{s}} \le \alpha }
\frac{1}{k!}
D^{k}f(z)
\Xk.
\end{equation}
The object $F$ should be thought of as a lift, or enhancement, of $f$. The original function is easily recovered since $f(z) = \langle F(z), \symbol{\mathbf{1}} \rangle$, where we have used the notation $\symbol{\mathbf{1}} :=\symbol{ \mathbf{X}^{0}}$. However, at each space time point $F$ also provides additional local information about $f$.

\medskip
Next, we make a connection between the abstract polynomials of $\bar{\TT}$ and concrete polynomials on $\R^{d+1}$. We define a family of linear maps $\{ \Pi_{z} \}_{z \in \R^{d+1}}$ where for any $z \in \R^{d+1}$ one has $\Pi_{z}:T \rightarrow \mathcal{S}'(\R^{d+1})$. The map $\Pi_{z}$ takes an element $\tau \in T$ and returns a concrete space-time distribution which is ``based at $z$''. In this section these space-time distributions will just be polynomials so we can specify them pointwise, for any $z \in \R^{d+1}$ and multi-index $k$ we set
\[
\left(
\Pi_{z}
\Xk
\right)(\bar{z})
:=\ 
(\bar{z} - z)^{k},
\]
where $\bar{z}$ is just a dummy variable. We then extend $\Pi_{z}$ to all of $T$ by linearity. The concrete Taylor polynomial for $f$ with base point $z$ is then given by
$\left(
\Pi_{z}F(z)
\right)(\cdot)$.

\medskip
A key ingredient of the theory of regularity structures  is a notion of smoothness for space-time distributions that are classically thought of as very singular. This requires lifting a space-time distribution to a family of local expansions at each space-time point, the notion of smoothness will then be enforced by comparing these local expansions at nearby space-time points. We make this analogy more concrete by showing how conditions \eqref{e:Lipschitz1} and \eqref{e:Lipschitz2} on $f$ can be elegantly encoded in terms of more abstract conditions on $F$.

\medskip
Directly comparing $F(z)$ and $F(z')$ for two close space-time points $z$ and $z'$ is quite unnatural since each of these local expansions are based at different space-time points. What we need is an analog of the parallel transport operation of differential geometry, we must transport a local description at one space-time point to another space-time point. 
For every pair $x,y \in \R^{d+1}$ we will define a linear map $\Gamma_{xy}:T \mapsto T$ which plays the role of parallel transport. $\Gamma_{xy}$ takes something which is written as a local object at the space-time point $y$ and ``transports'' it to $x$. This property is encoded in the algebraic relation
\begin{equation}\label{pigammaalg}
\Pi_{y}\tau = \Pi_{x} \Gamma_{xy} \tau 
\textnormal{ for all }\tau \in T,\ x,y \in \R^{d+1}.
\end{equation}
The structure group $G$ will provide all the operators $\Gamma_{xy}$. For any $h \in \R^{d+1}$ we set
\[
\Gamma_{h}\Xk := (\symbol{\mathbf{X}} - h)^{k},
\]
and we extend this definition to all of $T$ by linearity. $G$ is defined to be the collection of operators $\{\Gamma_{h}\}_{h \in \R^{d+1}}$, one can easily check this satisfies the necessary conditions (and that $G$ is isomorphic to $\R^{d+1}$ as a group). If $\Gamma_{xy} := \Gamma_{y-x}$ then \eqref{pigammaalg} is satisfied.

With all this in place we can give the following characterization of $\Cs^{\alpha}$ spaces for $\alpha \ge 0$. 

\begin{theorem}\label{modelledholderfunctions}
Let $\bar{\mathcal{T}} = (A,T,G)$ be the regularity structure of abstract polynomials (in $d+1$ components). Suppose that $\alpha > 0$. Then a function $f:\R^{d+1} \rightarrow \R$ is a member of $\Cs^{\alpha}$ (as in Definition \ref{positiveholder2}) if and only if there exists a function $F:\R^{d+1} \rightarrow T^{-}_{\alpha}$ such that $\langle F(z), \mathbf{1} \rangle = f(z)$ and for every compact set $\mathfrak{K} \subset \R^{d+1}$ and every $\beta \in A$ with $\beta < \alpha$ one has \refcomment{REF12}
\begin{equation}\label{modeledholderbound}
\sup_{x \in \mathfrak{K}} 
{\|F(x)\|_\beta} 
+
\sup_{
\substack{
x,y \in \mathfrak{K}\\
x \not = y}
}
\frac{
||F(x) - \Gamma_{x y}F(y)||_{\beta}
}
{||x-y||_{\mathfrak{s}}^{\alpha - \beta}}
< \infty,
\end{equation}
\end{theorem}

For checking that $f \in \Cs^{\alpha}$ implies \eqref{modeledholderbound} one defines $F$ as in \eqref{liftofholder} and check that the case of $\beta = 0$ encodes \eqref{e:Lipschitz1} and more generally the case of $\beta = l$ encodes \eqref{e:Lipschitz2} where $|j|_{\mathfrak{s}} = l$. For example, if $\beta = 1$ (and $\alpha > 1$) then one can check that \refcomment{REF13 - Parenthesis and exponent in first line corrected, along with changing a $k$ into a $k-j$ in the second line.}
\begin{equation*}
\begin{split}
\mathcal{Q}_{1} 
\Gamma_{x y}F(y)
=&\ 
\mathcal{Q}_{1}
\left(
\sum_{|k|_{\mathfrak{s}} \le \alpha} 
\frac{1}{k!}
D^{k}f(y)
\left(\symbol{\mathbf{X}} - (y-x)\right)^{k}
\right)\\
=&\ 
\mathcal{Q}_{1}
\left(
\sum_{|k|_{\mathfrak{s}} \le \alpha} 
\sum_{ j \le k}
\frac{1}{k!}
D^{k}f(y)
\frac{k!}{j!(k-j)!}
\Xj
(x-y)^{\rf{k - j}}
\right)\\
=&\ 
\sum_{|j|_{\mathfrak{s}} = 1}
\Xj
\sum_{|k|_{\mathfrak{s}} \le \alpha - 1}
\frac{1}{k!}
D^{j+k}f(y)
(x-y)^{k}.
\end{split}
\end{equation*}
We can assume that the $||\cdot||_{1}$ norm on $T_{1}$ is an $\ell_{1}$ type norm (with respect to the basis of monomials $\Xj$ with $|j|_{\mathfrak{s}} = 1$) and so we have
\begin{equation*}
||F(x) - \Gamma_{x y}F(y)||_{1}
=
\sum_{|j|_{\mathfrak{s}} = 1}
\left|
D^{j}f(x) - 
\sum_{|k|_{\mathfrak{s}} \le \alpha - 1}
\frac{1}{k!}
D^{j+k}f(y)
(x-y)^{k}
\right|
\end{equation*}
which combined with  \eqref{e:Lipschitz2} gives us \eqref{modeledholderbound}for all multi-indices $j$ with $|j|_{\mathfrak{s}} = 1$.
Showing that the existence of such an $F$ implies $f \in \Cs^{\alpha}$ is quite similar, \eqref{modeledholderbound} implies that $\{\Pi_{z}F(z)\}_{z \in \R^{d+1}}$ is a family of sufficiently good polynomial approximations for $f$.

\subsection{Models and Modelled Distributions}
We now give a more general and axiomatic description for some of the new objects we encountered in the last section. The first concept is that of a \emph{model} which is what allowed us to go from abstract symbols in a regularity structure to concrete space-time distributions.

\begin{definition}\label{defofmodel}
Let  $\TT = (A,T,G)$ be a regularity structure. A \textit{model} for $\TT$ on $\R^{d+1}$ consists of a pair $(\Pi,\Gamma)$ where
\begin{itemize}
\item $\Gamma$ is a map $\Gamma \colon \R^{d+1}\times \R^{d+1} \to G$ which we write $(x,y) \mapsto \Gamma_{xy}$. We require that $\Gamma_{xx} = I$ and $\Gamma_{xy}\, \Gamma_{yz} = \Gamma_{xz}$ for all $x,y,z \in \R^{d+1}$.
\item $\Pi = \{ \Pi_{x}\}_{x \in \R^{d+1}}$ is a family of linear maps $\Pi_x \colon T \to \CS'(\R^{d+1})$. 

\item One has the algebraic relation
\begin{equation}\label{algconditionmodel}
\Pi_y = \Pi_x \Gamma_{xy} \textnormal{ for all }
x,y \in \R^{d+1}.
\end{equation}
\end{itemize}
Finally, for any $\alpha \in A$ and compact set $\mathfrak{K} \subset \R^{d+1}$ we also require that the bounds \refcomment{REF14}
\begin{equation}\label{analyticboundmodel}
\left|
(\Pi_x \tau)
\bigl(\CS_{x}^\lambda \eta \bigr)
\right|
\lesssim
||\tau||_{\alpha} \lambda^{\alpha}
\textnormal{ and } 
\rf{\sup_{\beta < \alpha}\ 
\frac{\left|\left|\Gamma_{xy} \tau \right|\right|_\beta}
{||x-y||_{\mathfrak{s}}^{\al-\beta}}
\lesssim
||\tau||_{\alpha}}  
\; 
\end{equation} 
\noindent hold uniformly over all $\tau \in T_{\alpha}$, $\lambda \in (0,1]$, space-time points $x,y \in \mathfrak{K}$, and test functions $\eta \in B_{r}$ for $r :=  \lceil - \min A \rceil $.
\end{definition}

Given a fixed regularity structure $\TT$, let $\mathcal{M}$ be the set of all models on $\TT$. For any compact set $\mathfrak{K} \subset \R^{d+1}$ one can define a ``seminorm'' $||\cdot||_{\mathfrak{K}}$ on $\mathcal{M}$ by defining $||(\Pi,\Gamma)||_{\mathfrak{K}}$ to be the smallest real number $K$ such that the inequalities of \eqref{analyticboundmodel} hold over $x,y \in \mathfrak{K}$ with proportionality constant $K$ \footnote{We used the word seminorm in quotation marks since $\mathcal{M}$ is \emph{not} a linear space due to the algebraic constraint \eqref{algconditionmodel}.}. One can then define a corresponding metric on $\mathcal{M}$. While we do not explicitly give the metric here, the corresponding notion of convergence on $\mathcal{M}$ is very important and will be referenced when we introduce more of the machinery of regularity structures. 

\begin{remark}
It is straightforward to check that the $(\Pi,\Gamma)$ introduced last section satisfies the conditions to be a model for the regularity structure of abstract polynomials $\bar{\TT}$. 
\end{remark}

\begin{remark}
Given  $\tau \in T_{\alpha}$ and a model $(\Pi,\Gamma)$ it is not necessarily the case that $\Pi_{z}\tau \in \Cs^{\alpha}$. The key point here is that the first bound of \eqref{analyticboundmodel} is only enforced for test functions centered at $z$.
\end{remark}

Another thing we did in the previous section was develop a notion of regularity for families of local expansions $F:\R^{d+1} \rightarrow T$. More generally, such families $F$ with good regularity properties will be called \emph{modelled distributions}. 

\begin{definition}\label{modeldistdef}
Fix a regularity structure $\TT$ and a model $(\Pi,\Gamma)$ for $\TT$.
Then for any $\gamma \in \R$ the space of \textit{modelled distributions} $\CD^\gamma$ consists of all functions $F:\R^{d+1} \rightarrow T^{-}_{\gamma}$ such that for any compact set $\mathfrak{K} \subset \R^{d+1}$
\begin{equ}
\| F\|_{\gamma; \mathfrak{K}} := 
\sup_{x \in \mathfrak{K}} 
\sup_{\beta < \gamma} 
{\|F(x)\|_\beta}  
+ 
\sup_{
\substack{
x,y \in \mathfrak{K}\\ 
0 < \|x-y\|_\mathfrak{s} \le 1} 
}
\sup_{\beta < \gamma} 
\frac{\|F(x) - \Gamma_{xy} F(y)\|_\beta}
{ \|x-y\|_{\mathfrak{s}}^{\gamma - \beta} } < \infty\;.
\end{equ}
\end{definition}
The definition above generalizes the idea behind Theorem \ref{modelledholderfunctions}. In the next section we will see a scenario where a certain class of functions with classical regularity $\Cs^{\alpha}$ for $\alpha \in (\frac{1}{3}, \frac{1}{2})$ can be thought of as more regular via the construction of lifts to modelled distributions in a $\mathcal{D}^{\gamma}$ space with $\gamma = 2\alpha$. This corroborates our earlier remark that objects with bad classical regularity can be thought of as more regular via a lift to a well behaved family of local expansions.
In the next lecture we will see how this point of view actually pays off. Even if two space-time distributions $f,g$ are too irregular to define their product $fg$ via Theorem \ref{multiplication}, we will in fact be able to make sense of their product if we can lift them to a appropriate $\mathcal{D}^{\gamma}$ spaces. 

\begin{remark}
As was the case with the $\Cs^{\alpha}$ spaces, certain theorems for $\mathcal{D}^{\gamma}$ spaces fail when $\gamma \in \mathbf{Z}$ or more general, when $\gamma \in A$ (in particular, the abstract Schauder estimate in second part of \rf{Theorem \ref{reconstructionintegration}} \refcomment{REF15} fails). Therefore we implicitly assume that any $\mathcal{D}^{\gamma}$ space entering the assumptions or conclusion of theorem involve a value  $\gamma \notin A$. 
\end{remark}

The machinery of regularity structures operates with a \emph{fixed} regularity structure $\TT$ and \emph{varying} models $(\Pi,\Gamma)$. Therefore it is very important to remember that the definition of a $\mathcal{D}^{\gamma}$ space strongly depends on the choice of model (even though their constituent objects $F:\R^{d+1} \rightarrow T$ don't make reference to any model). We will sometimes use the notation $\mathcal{D}^{\gamma}[(\Pi,\Gamma)]$ to make the dependence of this space on the choice of model explicit. 
Furthermore, we will sometimes be interested in comparing modelled distributions that live in different $\mathcal{D}^{\gamma}$ spaces coming from different models. We define $\mathcal{M} \ltimes \mathcal{D}^{\gamma}$ to be the set of triples $(\Pi,\Gamma,F)$ such that $(\Pi,\Gamma) \in \mathcal{M}$ and $F \in \mathcal{D}^{\gamma}[(\Pi,\Gamma)]$. There is a natural way to turn $\mathcal{M} \ltimes \mathcal{D}^{\gamma}$ into a metric space, we do not describe this here but refer the reader to \cite[Remark 2.11]{hairer2015introduction}. 

\subsection{Controlled rough paths}

The theory of rough paths was originally developed by Lyons in \cite{Ly98}, in this section we will see how the theory of regularity structures is related to a variant of Lyons' rough paths due to Gubinelli \cite{Gu04} called \textit{controlled rough paths}. 
For the purposes of this section we will work with $\R^{d}$ valued functions defined on $[0,1]$ instead of the real-valued functions defined on space-time we looked at earlier. Modifying definitions \ref{defofmodel} and \ref{modeldistdef} to this setting is straightforward.

\medskip
Gubinelli was interested in defining the Riemann-Stieltjes type integral
\begin{align}\label{e:stochInt}
\int_0^1 f \cdot dg \;,
\end{align}
for functions $f,g$ in the H\"older space $\Cc^\ga([0,1],\R^{d})$ for some $\ga \in (\frac13, \frac12)$. The classical theory breaks down in this regime for familiar reasons, morally $dg$ is in $\Cc^{\gamma - 1}$ and the product $f\ dg$ is not canonically defined since $2 \gamma - 1 < 0$.   

The strategy of controlled rough paths can be paraphrased as follows. If $g$ is a well understood stochastic process one might be able to define the objects 
\begin{equation}\label{e:iteratedint}
\int_{0}^{\bullet} g_{i} \,dg_{j}
\end{equation}
for $1 \le i,j \le d$ via some probabilistic procedure (this is analogous to our construction of Wick powers of $Z$ earlier). Then based on a completely deterministic analysis, the integral \eqref{e:stochInt} can be constructed for a whole class of functions $f$ which admit a type of local expansion in terms of $g$.

\def\Gubi{\cite{Gu04}}
\begin{definition}[Gubinelli \Gubi]\label{controldef} A function $f:[0,1] \rightarrow \R^{d}$ is \emph{controlled} by a function $g \in \Cc^\ga([0,1],\R^{d})$ if there exists a function $D_{g}f: [0,1] \to \R^{d \times d}$ such that one has the bounds
\begin{equation}\label{e:Lip3}
\bigg| 
f(t)
-
\left[ 
f(s) 
+
D_{g}f(s) 
(g(t) - g(s)) 
\right]
\bigg| 
\lesssim |t-s|^{2\ga}
\end{equation}
and
\begin{equation}\label{e:Lip4}
\left| D_{g}f(t) - D_{g}f(s) \right| \lesssim |t-s|^\ga 
\end{equation}
uniformly over $s,t \in [0,1]$. Above $D_{g}f(s)$ is being thought as a $d \times d$ matrix acting on $\R^{d}$.
\end{definition}

The requirements \eqref{e:Lip3} and \eqref{e:Lip4} should be seen as analogs of \eqref{e:Lipschitz1} and \eqref{e:Lipschitz2} above and the object $D_{g}f$ is analogous to a derivative. Gubinelli's observation was that although $f$ will only be a $\Cc^\ga$ function, the fact that $f$ is controlled by $g$ allows one to treat $f$ as if it had $\Cc^{2\ga}$ regularity.  

\medskip
Fix a choice of $\gamma \in (\frac{1}{3},\frac{1}{2})$. We will now define a regularity structure $\TT$ and an associated model $(\Pi,\Gamma)$ built using a function $g \in \Cc^\ga([0,1],\R^{d})$. In this setting the statement that $f:[0,1] \rightarrow \R^{d}$ is controlled by $g$ will be equivalent to the existence of a lift of $f$ to modelled distribution in $\CD^{2\gamma}$. One difference we will see here versus Section \ref{abstractpoly} will be in the action of the structure group $G$ and the $\Gamma_{xy}$ of the model. The interested reader can also look at \cite[Section 3.2]{hairer2015introduction} where it is shown how enlarging the regularity structure given here and doing the same for the model (which is where one needs a definition for \eqref{e:iteratedint}) allows one to define the integral \eqref{e:stochInt}. 

\medskip
The regularity structure $\TT = (A,T,G)$ we use has indexing set $A = \{0, \gamma\}$ where $\gamma \in (\frac{1}{3},\frac{1}{2})$. We set $T_{0} = \R^{d}$ with distinguished basis $\{\symbol{\mathbf{E_i}}\}_{i=1}^{d}$ and $T_{\gamma} = \R^{d \times d}$ with distinguished based $\{\symbol{\mathbf{M_{i,j}}}\}_{i,j=1}^{d}$. We now turn to defining the structure group $G$, for any $h \in \R^{d+1}$ we define $\Gamma_{h}: T \rightarrow T$ by setting 
\[
\Gamma_{h}
\symbol{\mathbf{E_j}}\ 
= \symbol{\mathbf{E_j}},\ 
\Gamma_{h}\symbol{\mathbf{M_{i,j}}}
=\ 
\symbol{\mathbf{M_{i,j}}}
+
h_{j}\symbol{\mathbf{E_{i}}},\ 
\]
and extending by linearity. We then set $G: = \{\Gamma_{h}\}_{h \in \R^{d}}$. It is an easy exercise to check that $G$ satisfies the necessary properties to be a structure group and is in fact ismorphic to $\R^{d}$. 
If a function $f$ is controlled by $g$ we can lift $f$ to vector $F \colon [0,1] \to T^{d}$  of modelled distributions by setting
\begin{align}
F_i(s) =  f_i(s) \symbol{\mathbf{E_i}}  + \sum_j D_{g}f_{i,j}(s) \symbol{\mathbf{M_{i,j}}}  \;
\end{align} 
%

%where $\langle f(s), \symbol{\mathbf{E}} \rangle = \sum_{i=1}^{d} f_{j}(s)\symbol{\mathbf{E_i}}$, the second term above is defined analogously.

\medskip
We now describe a way to  build a (vector-valued) model $(\Pi,\Gamma)$ for this regularity structure for any fixed $g \in \Cc^{\gamma}$. For $t \in [0,1]$ we set
\begin{equation*}
\begin{split}
\left( \Pi_{t}\symbol{\mathbf{E_{i}}}\right)(r) &= e_{i},\\ 
\left( \Pi_{t}\symbol{\mathbf{M_{i,j}}}\right)(r) &= (g_{j}(r) - g_{j}(t))e_{i}
\end{split}
\end{equation*}
where $r \in [0,1]$ is a dummy variable and $\{e_{i}\}_{i=1}^{d}$ are the standard basis vectors for $\R^{d}$ (these are concrete vectors, as opposed to the abstract symbols $\{\symbol{\mathbf{E_i}}\}_{i=1}^{d}$). 
Finally we define the second part of the model as follows, for $s,t \in [0,1]$ we set $\Gamma_{st} = \Gamma_{g(t) - g(s)} \in G$. One can then check that $(\Pi,\Gamma)$ satisfy the algebraic and analytic conditions to be a model. Finally one has the following theorem.

\begin{theorem}\label{modeleddistdgamma} Let $\TT$ be the regularity structure defined above and let $(\Pi,\Gamma)$ be a model built out a fixed $g \in \Cc^\ga([0,1],\R^{d})$. Then a function $f:[0,1] \rightarrow \R^{d}$ is controlled by $g$ if and only if there exists a modelled distribution $F \in \CD^{2\gamma}$ with $\mathcal{Q}_{0}F(t) = \sum_{i=1}^{d}f_{i}(t)\symbol{\mathbf{E_i}}$. 
\end{theorem}

\subsection{Regularity Structures for SPDEs}\label{sec: regstructforspde}
We take a moment to discuss the vector space $T$ that appears in regularity structures $\TT$ used for solving equations like \eqref{e:general_equation}. 
The space $T$ will be formed by the linear span of abstract symbols. 
We denote by $\mathfrak{T}$ the set of all abstract symbols appearing in $T$. 
$\mathfrak{T}$ contains the symbol $\symbol{\Xi}$ which represents the driving noise $\xi$, since $\xi$ is taken to be space-time white noise we set $|\Xi| = -d/2 - 1 - \kappa$ where $\kappa > 0$ can be taken arbitrary small.
$\mathfrak{T}$ will also have the symbol $\symbol{\mathcal{I}}[\symbol{\Xi}]$ which represents the solution to the underlying linear equation. 
More generally, given a symbol $\tau \in \mathfrak{T}$ it will \emph{sometimes} be the case that $\mathfrak{T}$ also contains the abstract symbol $\symbol{\mathcal{I}}[\tau]$ which represents the ``integral'' of $\tau$, that is ``$K \star \tau$''. 
Inspired by the Schauder estimate (Theorem \ref{Schaudestimate}) we would then require $|\symbol{\mathcal{I}}[\tau]| = |\tau| + 2$. However,  we do not allow any symbol of the form $\symbol{\mathcal{I}[\mathbf{X^{k}}]}$ \footnote{Said differently, we assume that $\symbol{\mathcal{I}}[\cdot]$ applied to any abstract polynomial vanishes, Section \ref{sec: integration} will clarify this.}.
Given symbols $\tau_{1}, \tau_{2} \in \mathfrak{T}$ it will \emph{sometimes} be the case that $\mathfrak{T}$ will contain the abstract symbol $\tau_{1} \tau_{2}$, which represents a commutative product of $\tau_{1}$ and $\tau_{2}$. In this case we will require $|\tau_{1} \tau_{2}| = |\tau_{1}| + |\tau_{2}|$. This condition on products is an important way that the concept of order differs from that of regularity \footnote{In particular, this will allow us to bypass circular issues like the product $v \<2>$ in Section \ref{dpdsec}.}.

\medskip

The symbols mentioned above are generated recursively: one starts with a set of primitive symbols which consists of $\symbol{\Xi}$ and various powers of $\symbol{\mathbf{X}}$ and then builds more complicated symbols by using $\symbol{\mathcal{I}}[\cdot]$ and our abstract product. The graphical notation we used for $\Phi^{4}$ should be seen as a shorthand for the symbols we have described. For example, we have \refcomment{REF16}
\[
\<s31>
=
\symbol{
\mathcal{I} 
\left[ \mathcal{I}[\Xi]^{3}\right]
\mathcal{I}[\Xi] }.
\]
However, the regularity structures one encounters in practice do not contain \emph{all} the symbols generated by the recursive procedure sketched above, doing so would usually create problems for the first and second conditions of Definition \ref{Regstrucdef}. 

\medskip
To construct the right list of symbols $\mathfrak{T}$ we start by iteratively applying a particular set of rules $\mathfrak{R}_{F}$ determined by the structure of the non-linearity $F$ appearing in \eqref{e:general_equation}. The list of rules for $\Phi^{4}$ equations is given by

\begin{equation}\label{rules}
\mathfrak{R}_{\Phi^{4}} := 
\{ 
\symbol{\mathbf{X}^{k} \mathcal{I}[\cdot]},\ 
\symbol{\mathbf{X}^{k} \mathcal{I}[\cdot] \mathcal{I}[\cdot]}, \
\symbol{ \mathcal{I}[\cdot] \mathcal{I}[\cdot]\mathcal{I}[\cdot]} \ 
\} \;\rf{.}
\end{equation}
\refcomment{REF17}
Above and in what follows $k$ represents an arbitrary multi-index, sometimes subject to a stated constraint.
We set $\mathfrak{T}_{0} := \{ \symbol{\Xi},\ \Xk \}$ to be the set of primitive symbols. Then for $j \ge 1$, the set $\mathfrak{T}_{j}$ is formed by taking the union of $\mathfrak{T}_{0}$ with the set of all the symbols that one gets by applying any of the operations listed in the given rule $\mathfrak{R}_{F}$ to any of the elements of $\mathfrak{T}_{j-1}$. For example, in the case of $\Phi^{4}$ it is the case that 
\[
\tau_{1}, \tau_{2}, \tau_{3} \in \mathfrak{T}_{j} 
\Rightarrow 
\symbol{\mathcal{I}}[\tau_{1}]\symbol{\mathcal{I}}[\tau_{2}]\symbol{\mathcal{I}}[\tau_{3}] \in \mathfrak{T}_{j+1}.
\]

An important consequence of subcriticality of the equation \eqref{e:general_equation} is the following: if one defines the sets of symbols $\mathfrak{T}_{j}$ using $\mathfrak{R}_{F}$ then there will exist some $\beta > 0$ such that for all $j \ge 1$ one has
\[
\min_{\tau \in \mathfrak{T}_{j} \setminus \mathfrak{T}_{j-1}} |\tau|
>
\min_{\tau \in \mathfrak{T}_{j-1}} |\tau|
+ \beta.
\] 
This means that as we iterate the application of the rule the new symbols we generate are increasing in order - this guarantees that if we set $\mathfrak{T}$ equal to $\cup_{j=0}^{\infty} \mathfrak{T}_{j}$ then the corresponding list of orders $A$ will be bounded below and will not contain any accumulation points. 

\medskip
However, $\mathfrak{T}$ would still include an infinite list of symbols. In practice one wants $\mathfrak{T}$ to be a finite set - to do this we fix a constant $\gamma$ which is the upper limit on what order symbols we include in our regularity structure \footnote{$\gamma$ will need to be sufficiently large to allow one to pose the abstract fixed point problem, see Section~\ref{subsec: multiplication}.}. We can then modify our previous construction.  For $j \ge 1$ we define the sets $\bar{\mathfrak{T}}_{j}$ by taking the union of $\mathfrak{T}_{0}$ with the set of all the symbols that one gets by applying any of the operations listed in the given rule $\mathfrak{R}_{F}$ to any of the elements of $\bar{\mathfrak{T}}_{j-1}$, but now with the convention that $\symbol{\mathcal{I}}[\tau] > \gamma$ then  $\symbol{\mathcal{I}}[\tau]$ is considered to vanish \footnote{In particular, no symbol can contain $\symbol{\mathcal{I}}[\tau]$ as a subsymbol.}. We then set

\[
\mathfrak{T}:= \left\{ \tau \in \mathfrak{T}_{0} \cup \left(\bigcup_{j=1}^{\infty} \bar{\mathfrak{T}}_{j} 
\right)
:\ |\tau| \le \gamma \right\}.
\]

\subsection{The regularity structure and model for $\Phi^{4}_{2}$}\label{regphi42}

When defining the regularity structure $\TT$ for the $\Phi^{4}_{2}$ equation the list of symbols is given by $\mathfrak{T}$ defined as above with  the rule $\mathfrak{R}_{\Phi^{4}}$ and $\gamma$ taken positive but sufficiently small ($\gamma > 2 \kappa$ suffices). 
\begin{table}[h]
\caption{Symbols for $\Phi^{4}_{2}$}
\centering
\begin{tabular}{|c|c|}
\hline
Symbol & Order \\ \hline
$\symbol{\Xi}$ & $-2 -\kappa$ \\ \hline
$\<s1>$ & $- \kappa $ \\ \hline
$\<s2>$ & $ - 2 \kappa $ \\ \hline
$\<s3>$ & $ - 3 \kappa $ \\ \hline
$\symbol{\mathbf{1}}$ & $0$ \\ \hline
\end{tabular}
\label{Phi42symbols}
\end{table}

Any realization of the driving noise $\xi$ can then be lifted to a model $(\hat{\Pi}^{\delta},\hat{\Gamma}^{\delta})$ as follows. For any $z \in \R^{2+1}$ we set:

\begin{equation}\label{Phi42model}
\begin{split}
\left( \hat{\Pi}^{\delta}_{z} \symbol{\Xi} \right)(\bar{z}) 
&=\  \xi_{\delta}(\bar{z})\\ 
\left( \hat{\Pi}^{\delta}_z \<s1>\right)(\bar{z}) 
&=\ \<1>_{\delta}(\bar{z})
\end{split}
\qquad
\begin{split} 
\left( \hat{\Pi}^{\delta}_z\<s2> \right)(\bar{z})
&=\ \<2>_{\delta}(\bar{z})\\
\left(\hat{\Pi}^{\delta}_z\<s3>\right)(\bar{z}) 
&=\  \<3>_{\delta}(\bar{z})\\ 
\left( \hat{\Pi}^{\delta}_z \symbol{\1}\right)(\bar{z}) 
&=\ 1 
 \;.
\end{split}
\end{equation}
 Here $\<1>_\delta$, $\<2>_\delta$, $\<3>_\delta$ are the approximate Wick powers introduced in Lecture~\ref{s:l2}.
A key simplification with $\Phi^{4}_{2}$ is that the maps $\Pi_{z}$ do not depend on $z$, this means we can set $\Gamma_{xy} = \mathrm{Id}$ where $\mathrm{Id}$ is the identity on $T$.  It can be checked that the models $(\Pi^{\delta},\Gamma^{\delta})$ satisfy the conditions of Definition \ref{defofmodel}. Additionally, one can remove the regularization and show that the models $(\Pi^{\delta},\Gamma^{\delta})$, viewed as random elements of $\mathcal{M}$, converge in probability as $\delta \downarrow 0$ to a limiting random model $(\Pi,\Gamma)$ given by \footnote{We continue to abuse notation here, using point-wise equalities for singular distributions.}

\begin{equation*}
\begin{split}
\left( \hat{\Pi}_{z} \symbol{\Xi} \right)(\bar{z}) 
&=\  \xi(\bar{z})\\ 
\left( \hat{\Pi}_z \<s1>\right)(\bar{z}) 
&=\ \<1>(\bar{z})
\end{split}
\qquad
\begin{split} 
\left( \hat{\Pi}_z\<s2> \right)(\bar{z})
&=\ \<2>(\bar{z})\\
\left(\hat{\Pi}_z\<s3>\right)(\bar{z}) 
&=\  \<3>(\bar{z})\\ 
\left( \hat{\Pi}_z \symbol{\1}\right)(\bar{z}) 
&=\ 1 
 \;
\end{split}
\end{equation*}
and with $\Gamma_{xy} = \mathrm{Id}$ for all $x,y \in \R^{d+1}$.

\medskip
Although the full regularity structure $\TT$ is required to formulate the fixed point problem, the solution $\Phi$ will be of the form $\Phi(z) = \Phi_{\symbol{\mathbf{1}}} \symbol{\mathbf{1}} + \<s1>$ which is similar to the decomposition seen in the Da Prato - Debussche argument. The fact that the structure group can be chosen to be trivial is why the Da Prato - Debussche argument works for $\Phi^{4}_{2}$.  

\subsection{The regularity structure and model for $\Phi^{4}_{3}$}\label{regphi43}
For $\Phi^{4}_{3}$, 
%if one uses the appropriate model 
it suffices to define $\mathfrak{T}$ by taking $\gamma$ slightly greater than $1$. We include a table of these symbols below \footnote{The algorithm for construction $\mathfrak{T}$ that we have given may produce extraneous symbols and indeed  the last three symbols given in Table~\ref{Phi43symbols} are unnecessary to set up a fixed point problem for $\Phi^4_3$. See \cite[Sec 8.1]{RegStr} for an algorithm that will give a minimal list of symbols. Also see \ref{sec:FPA} of these lecture notes for a discussion of relevant terms.}.  Whenever a factor $\Xk$ appears in a symbol the multi-index $k$ can vary but is constrained by the condition that the symbol's order be less than $\gamma$.

\begin{table}[h] %\label{t2}
\caption{Symbols for $\Phi^{4}_{3}$}
\centering
\scalebox{.9}{
\begin{tabular}{|c|c|}
\hline
Symbol  &  Order \\ \hline
$\symbol{\Xi}$ & $-5/2 - \kappa$ \\ \hline
$\Xk$ & $|k|_{\mathfrak{s}}$ \\ \hline
$\<s1> \Xk$ & $|k|_{\mathfrak{s}} - 1/2 - \kappa $\\ \hline
$\<s2> \Xk$ & $|k|_{\mathfrak{s}} - 1 - 2\kappa $\\ \hline
$\<s3> $ & $- 3/2 - 3\kappa $\\ \hline
$\<s30>$ & $ 1/2 - 3 \kappa$ \\ \hline
$\<s12>$ & $1/2 - 3 \kappa$ \\ \hline
$\<s20>$ & $1 - 2 \kappa $ \\ \hline
$\<s31>$ & $- 4 \kappa$ \\ \hline
$\<s32>$ & $ -1/2 - 5 \kappa$ \\ \hline
$\<s22>$ & $ -4 \kappa$ \\ \hline
$\<s30>\ \<s30>$ & $ 1 - 6 \kappa$ \\ \hline
$\symbol{\mathbf{X}^{k}} \<s31>$ & $ -4 \kappa + |k|_{\mathfrak{s}}$ \\ \hline
$\<s21>$ & $ 1/2 - 3 \kappa$ \\ \hline
\end{tabular}}
\label{Phi43symbols}
\end{table}
%
%\hwcomment{Actually, I don't. If you check carefully, all of these trees still appear even with your new rules. In practice we do not need any  of them, because they all only appear in the description of the rhs of the equation, which we only need up to order $>0$. (this would be easier to capture with the old rule where one has separate sets $\mathcal{U}$ etc...) It is only the lhs that one needs to describe up to order $\gamma$}\accomment{Added symbols to table, added footnote}
%
%{\color{red} Shouldn't the symbols $\<s30>^2$ (order $1-6 \kappa$), $\<s31> \mathbf{X}^k$ (order $-4\kappa +|k|_{\mathfrak{s}}$), $\<s20> \<s1>$ (order $\frac{1}{2} - 3\kappa$) be included as well according to our rule? - I am not sure if our rule produces some terms that we actually do not need... I think a ``true" rule could also state that once $\mathcal{I}[\tau]$ has order $>\gamma$ we trash all symbols including $\mathcal{I}[\tau]$, even if it has got order $<\gamma$ by multiplying with a tree of negative order}
%
Again, the approach is to define a family of random models $(\hat{\Pi}^{\delta},\hat{\Gamma}^{\delta})$, defined via lifting $\xi$, which converge in probability to a limiting random model $(\hat{\Pi},\hat{\Gamma})$ as $\delta \downarrow 0$. 
For $\tau = \symbol{\Xi}, \<s1>, \<s2>$, and $\<s3>$, we define $\hat{\Pi}^{\delta}_{z}\tau$ in the same way as we did for \eqref{Phi42model} (where the objects are replaced by their $d=3$ counterparts). We will not explicitly describe all of the model $(\hat{\Pi},\hat{\Gamma})$, the goal for our present discussion is to show how the $\Gamma^{\delta}_{xy}$'s and structure group $G$ are forced to be non-trivial. 

\medskip
This is easily seen with the symbol $\<s30>$. A naive definition one might make is
\begin{equation}\label{naivedefinition}
\begin{split}
\left( 
\hat{\Pi}^{\delta}_{z}\<s30>
\right)(\bar{z})
=&\ 
\<30>_{\delta}(\bar{z})
=
\int_{\R^{d+1}} du\ K(\bar{z} - u) \<3>_{\delta}(u)
\end{split}
\end{equation}
for any $z \in \R^{3+1}$. However this definition will \emph{not} satisfy the first bound of \eqref{analyticboundmodel}.
%
%At this point we are directly facing the difference between regularity and order. \hwcomment{Actually, for this term regularity and order are still the same. It is only for products with this term that the difference appears}  
While the objects on the right hand side of \eqref{naivedefinition} is of regularity $\frac{1}{2} - 3\kappa$, it  does not satisfy the bound 
\begin{equation}\label{orderbounds30}
\left|
\left(\hat{\Pi}_{z}\<s30>\right)
\bigl(\CS_{z}^\lambda \eta \bigr) 
\right|
\lesssim 
\lambda^{\frac{1}{2} - 3\kappa}
\end{equation}
uniformly in $\lambda \in (0,1]$ for an arbitrary test function $\eta$.
%
%\medskip
A way to reconcile this difference was already seen in Definition \ref{positiveholder2} - if we want to see a space-time function of regularity $\gamma > 0$ vanish at order $\gamma$ then we should subtract a suitable Taylor polynomial. We will get the bound \eqref{orderbounds30} if \eqref{naivedefinition} is changed to
\begin{equation*}
\left( 
\hat{\Pi}^{\delta}_{z}\<s30>
\right)(\bar{z})
=\ 
\int du\ K(\bar{z} - u) \<3>_{\delta}(u)\ 
-
\int du\ K(z - u) \<3>_{\delta}(u).\
\end{equation*}

However $\left( \Pi^{\delta}_{z}\<s30> \right)(\bar{z})$ now has a dependence on $z$ which means that the structure group $G$ cannot be chosen to act trivially on $\<s30>$. 
The compatibility condition \eqref{algconditionmodel} determines completely how $\Gamma$  acts on $\<30>$.
%We can  guess what $\hat{\Gamma}^{\delta}_{xy}\<s30>$ should be via some computation. 
Indeed,  $\hat{\Pi}^{\delta}_{x}\hat{\Gamma}^{\delta}_{xy} = \hat{\Pi}^{\delta}_{y}$ gives us that
\begin{equation*}
\begin{split}
\left[ \hat{\Pi}^{\delta}_{x} \left( \mathrm{Id} - \hat{\Gamma}^{\delta}_{xy}\right) \<s30>\right](\bar{z}) 
&=\ 
\left( \hat{\Pi}^{\delta}_{x}\<s30> \right)(\bar{z}) 
-
\left( \hat{\Pi}^{\delta}_{y}\<s30> \right)(\bar{z})\\ 
&=\ 
\int du\ K(y - u) \<3>_{\delta}(u) 
-
\int du\ K(x - u) \<3>_{\delta}(u).
\end{split}
\end{equation*}
Therefore, we set
\[
\hat{\Gamma}^{\delta}_{xy} \<s30>\ 
=\ 
\<s30> 
+
\left(
\int du\ K(y - u) \<3>_{\delta}(u) 
-
\int du\ K(x - u) \<3>_{\delta}(u)
\right)
\symbol{\mathbf{1}}.
\]
The group action on all the other symbols is determined by similar considerations for integration and the compatibility condition  for products, given in \eqref{e:ReGular} below.
\begin{remark}
In general,  terms involving $\symbol{\mathbf{X}}$ appear in a model when an abstract integration leads to a symbol of positive order. It is worth mentioning that these extra terms do not occur in Gubinelli's 
approach \cite{Massimiliano1} to singular SPDEs using ``paracontrolled distributions". 
\end{remark}

\section{Lecture 4}

\subsection{Construction of Canonical Models}
In the last lecture we discussed regularity structures and models associated with controlled rough paths, $\Phi^{4}_{2}$, and $\Phi^{4}_{3}$. In this section we will show that for any fixed regularity structure $\mathcal{T}$ which is created by a set of formal rules like \eqref{rules}, there exists a canonical way to map each fixed realization of a smoothed noise $\xi_{\delta}$ to a corresponding model $(\Pi^{\delta}, \Gamma^{\delta})$. This model is called a \emph{canonical model} and it will be defined recursively with respect to $\symbol{\mathcal{I}}[\cdot]$ and the abstract product on $T$. After that we will discuss more systematically how to perform the renormalisation procedure which leads to the \emph{renormalized} models, examples of which we have already encountered in the previous lecture.

\medskip
In order to motivate concepts we ignore Definition \ref{defofmodel} for a moment. There is a naive approach to assigning a concrete space-time function (built out of $\xi_{\delta}$) to each of the abstract symbols appearing in our regularity structure, one can recursively define a \emph{single} linear map 
$\mathbf{\Pi^{\delta}}:T \rightarrow \mathcal{S}'(\R^{d+1})$ by setting:

\begin{equation}\label{boldpi}
\begin{split}
%\left( \mathbf{\Pi^{\delta}}\Xk \right)(\bar{z}) 
%&:= \bar{z}^{k}\\
\left( \mathbf{\Pi^{\delta}}\symbol{\Xi} \right)(\bar{z})
&:= 
\xi_{\delta}(\bar{z})\\ 
\left( \mathbf{\Pi^{\delta}}\symbol{\mathcal{I}}[\tau]\right)(\bar{z})
&:= \int_{\R^{d+1}}dy\ K(\bar{z} - y ) \left( \mathbf{\Pi^{\delta}}\tau \right)(y)\\
\left(\mathbf{\Pi^{\delta}} \tau_{1} \tau_{2} \right)(\bar{z})
&:=\left( \mathbf{\Pi^{\delta}}\tau_{1} \right)(\bar{z}) \times \left( \mathbf{\Pi^{\delta}}\tau_{1} \right)(\bar{z}).
\end{split}
\end{equation}
The map $\mathbf{\Pi^{\delta}}$ is a much simpler object than a model but it encodes less structure. In particular, it does not directly encode anything about the order of objects. The additional structure that models encode is  what makes the machinery of Sections \ref{sec:reconstruction} and  \ref{sec: integration}  continuous with respect to models.

\medskip
We have already seen above that when a regularity structure includes symbols of positive order, the maps $\{\Pi_{z}\}$ in a model must be allowed to be $z$-dependent, if we want the first bound of \eqref{analyticboundmodel} to hold.  
Keeping this in mind, we now describe how the maps $\Pi_{z}^{\delta}$ of the canonical model $(\Pi^{\delta},\Gamma^{\delta})$ are defined. For any $z \in \R^{d+1}$ one sets:
\begin{equation}\label{canonicalpidef}
\begin{split}
\left( \Pi^{\delta}_{z} \Xk \tau \right)(\bar{z})
&:= 
(\bar{z}-z)^{k} \times \left( \Pi^{\delta}_{z} \tau \right)(\bar{z})\\
\left( \Pi^{\delta}_{z}\symbol{\Xi} \right)(\bar{z})
&:= 
\xi_{\delta}(\bar{z})\\ 
\left( \Pi^{\delta}_{z}\symbol{\mathcal{I}}[\tau]\right)(\bar{z})
&:= \int_{\R^{d+1}}dy\ K(\bar{z} - y ) \left( \Pi^{\delta}_{z}\tau \right)(y)\\
&\hspace{1cm}
- 
\sum_{
\substack{k \\
|k|_{\mathfrak{s}} < |\tau| + 2
}
}
\frac{
(\bar{z} - z)^{k}}
{k!}
\int_{\R^{d+1}}dy\ (D^{k}K)(z - y ) \left( \Pi^{\delta}_{z}\tau \right)(y)\\
\left(\Pi^{\delta}_{z} \tau_{1} \tau_{2} \right)(\bar{z})
&:=\left( \Pi^{\delta}_{z} \tau_{1} \right)(\bar{z})
\times
\left(\Pi^{\delta}_{z}\tau_{1} \right)(\bar{z}).
\end{split}
\end{equation}
The key point here is that the application of $\symbol{\mathcal{I}[\cdot]}$ can produce a new object of positive order from an old one of negative order. This is why in the third line of \eqref{canonicalpidef} the subtraction we have implemented is just the subtraction of a partial Taylor expansion when $|\symbol{\mathcal{I}}[\tau]| \ge 0$. Of course, multiplication can also produce new objects of positive order but this is dealt with automatically when we enforce the product property, see Section \ref{subsec: multiplication} below, in particular \eqref{e:ReGular}. 

\medskip
In Section \ref{regphi43} we described models\footnote{We remark again that these models were \emph{not} canonical models, see the next section.} $(\hat{\Pi}^{\delta},\hat{\Gamma}^{\delta})$ where for $\tau \in \{ \symbol{\Xi}, \<s1>, \<s2>, \<s3>\}$ the function $\hat{\Pi}^{\delta}_{z}\tau$ had no $z$-dependence.
This was possible because these abstract symbols are of negative order \emph{and} the latter three objects are built (using $\symbol{\mathcal{I}}$ and the abstract product) out of objects which are all of negative order. 
However the symbol $\<s30> = \symbol{\mathcal{I}[ \mathcal{I}[\symbol{\Xi}]^{3}]}$ is of positive order so $\Pi^{\delta}_{z}\<s30>$ had to be $z$-dependent in order for the first bound of \eqref{analyticboundmodel} to hold. We also remark that $\hat{\Pi}^{\delta}_{z}\<s32>$ will also be $z$-dependent even though $\<s32>$ is of negative order - this is because $\<s32>$ is a product of $\<s2>$ and $\<s30>$ and the latter is of positive order.

\medskip
The recursive definition \eqref{canonicalpidef} is convenient to state and useful in many proofs. One can also recursively define the corresponding operators $\Gamma^{\delta}_{xy}$ as we have already sketched above. For models like the canonical model for $\Phi^{4}_{3}$  one could in principle  check the algebraic properties by hand, but this can easily become unwieldy with more complicated examples. \rf{In \cite{RegStr}} \refcomment{REF18} the connection between a recursive definition of the $\Gamma^{\delta}_{xy}$'s and their algebraic properties is made clear in an elegant way by using the language of Hopf algebras (see \cite[Sec 5.3]{hairer2015introduction}). This formulation also becomes very useful when one wants to go beyond the canonical model and construct the renormalized models of the next section. 

\subsection{Convergence of Models and Renormalization}

Most of the machinery we discuss in this lecture is completely deterministic, applied separately for each fixed realization of the noise $\xi$. However there is one major obstacle that this deterministic analysis cannot overcome: in general the canonical models $(\Pi^{\delta},\Gamma^{\delta})$ built from a $\delta$-smoothing of a fixed realization of $\xi$ will \emph{not} converge in the limit $\delta \downarrow 0$. The canonical models associated to $\Phi^{4}_{2}$ and $\Phi^{4}_{3}$ are examples of this: we have already seen in Lecture~\ref{s:l2} that the random space-time distributions $\<1>_\delta^3$ (which play the role of $\Pi^{\delta}_{z}\symbol{\mathcal{I}[\Xi]^{3}} $ in the canonical model based on $\xi_\delta$) do not converge as $\delta \to 0$. The Da Prato - Debussche argument for $\Phi^{4}_{2}$ overcame this by implementing renormalization subtractions and the approach in regularity structures is similar.

\medskip
We first discuss a criterion for the stochastic convergence of models. 
Let $\mathcal{T}$ be a regularity structure defined as in Section \ref{sec: regstructforspde} and  let $\mathfrak{T}_{-}$ be the set of abstract symbols \emph{of negative order} appearing in $\mathcal{T}$.
% and $\mathfrak{T}_{-} \subset \mathfrak{T}$ be the subset of symbols $\tau$ with $|\tau| < 0$.
%
%Now let $(\Pi^{\delta},\Gamma^{\delta})$ be a sequence of random models on the regularity structure $\mathcal{T}$. 
%Suppose that for any $z \in \R^{d+1}$, $\tau \in \mathfrak{T}_{-}$, and smooth test function $\eta$, the random variables $\left(\tilde{\Pi}^{\delta}_{z}\tau\right)(\eta)$ converge in probability, as $\delta \downarrow 0$, to a limiting random variable which we denote $\left(\tilde{\Pi}_{z}\tau\right)(\eta)$. \hwcomment{why do we need to suppose that?}
%
We seek conditions that ensure  that a sequence  $(\Pi^{\delta}, \Gamma^{\delta})$ of random models converge in probability to a random limiting model $(\Pi,\Gamma)$. The key stochastic estimates to show this are the following: for every $\tau \in \mathfrak{T}_{-}$, test function $\eta$, there should exist $\nu, \nu' > 0$ such that the bounds 
\begin{equation}\label{uniformboundonmodels} 
\mathbb{E}
\left[
\left|
(\Pi^{\delta}_{z} \tau)(\CS^{\lambda}_{z} \eta )
\right|^{p}
\right] \lesssim \lambda^{p(|\tau| + \nu)}
\end{equation}
and %\accomment{REF19 \hw{I agree - I like the $\delta$ here}}
\refcomment{REF19 - The authors feel that the $\delta$ should appear in \eqref{uniformboundonmodels} since this is precisely the stochastic estimate (uniform for small $\delta$) one in practice works for}
\begin{equation}\label{differenceboundonmodels}
\mathbb{E}
\left[
\left|
(\Pi^{\delta}_{z} \tau - \Pi_{z} \tau)(\CS^{\lambda}_{z} \eta )
\right|^{p}
\right] \lesssim \delta^{\nu'} \lambda^{p(|\tau| + \nu)}
\end{equation}
hold for every $p \in \mathbf{N}$, uniformly over $\delta, \lambda \in (0,1]$, uniformly over a suitable class of test-functions $\eta$ and locally uniformly in $z$\footnote{For a precise statement of all required conditions in the Gaussian case see \cite[Theorem 10.7]{RegStr}.}. 

\medskip
Note that the conditions above only involve a finite number of symbols $\tau$. This is similar to what we saw in the Da Prato - Debussche argument and as before is a consequence of subcriticality. We also remark that under some  natural assumptions on the sequence of models $(\Pi^{\delta},\Gamma^{\delta})$ one does not have to perform any stochastic estimates on the $\hat{\Gamma}^{\delta}$'s. Finally, as we have already seen in Lecture~\ref{s:l2},   one can win a major simplification when the driving noise $\xi$ is Gaussian. In the Gaussian case one can apply Nelson's Estimate \eqref{e:Nelson} and it suffices to establish the stochastic estimates just for $p=2$.
The reader has already seen the derivation of the bound \eqref{uniformboundonmodels} for $\tau = \<s3>$ in $\Phi^4_2$ and $\Phi^4_3$. The simple graphical approach based on convolutions 
presented there is sufficient for $\Phi^4_3$.

\medskip
It is important to observe that the bound \eqref{uniformboundonmodels} consists of \emph{two statements}: One is showing that the given quantity satisfies the right type of upper bound in $\lambda$ which  determines the \emph{order} of the limiting object. The other statement is that the quantity  remains finite as $\delta \downarrow 0$ and relates to the so called ultraviolet divergences and renormalization. These two issues are essentially orthogonal to each other. 
In particular, if the canonical models have divergent behaviour as $\delta \downarrow 0$ then one will already see this at the level of the simpler map $\mathbf{\Pi^{\delta}}$ of \eqref{boldpi}; it is conceptually simpler to first try to renormalize this map. One does this by defining a new map $\mathbf{\hat{\Pi}^{\delta}} := \mathbf{\Pi^{\delta}} M^{\delta}$ where $M^{\delta}:T \mapsto T$ is a linear map which performs renormalization subtractions at the level of the regularity structure $\mathcal{T}$. For example, for $\Phi^{4}_{3}$ one would have 
\begin{equation*}
\begin{split}
M^{\delta} \<s2> &= \<s2> - C_{\delta}, \\
M^{\delta} \symbol{\mathbf{X_{j}}}\<s2> &= \<s2>\symbol{\mathbf{X_{j}}} - C_{\delta}\symbol{\mathbf{X_{j}}} \textnormal{ for }j=1,2,3 \\
M^{\delta} \<s22> &= \<s22> - C_{\delta}\<s20> - \tilde{C}_{\delta} \symbol{\mathbf{1}},\\
M^{\delta} \<s32> &= \<s32> - 3C_{\delta} \<s12> - C_{\delta} \<s30> + 3C_{\delta}^{2} \<s10> - 3\tilde{C}_{\delta} \<s1>,
\end{split}
\quad
\begin{split}
M^{\delta} \<s3> &= \<s3> - 3C_{\delta}\<s1>,\\
M^{\delta} \<s12> &= \<s12> - C_{\delta}\<s10>, \\
M^{\delta} \<s31> &= \<s31> - 3C_{\delta}\<s11>,\\
\end{split}
\end{equation*}
where $C_{\delta} \sim \frac{1}{\delta}$ and $\tilde{C}_{\delta} \sim - \log(\delta)$ \footnote{The renormalization procedure given for $\<s22>$ is a bit inconsistent, one should also include terms $- C_{\delta}c_{K}\<s2> + C_{\delta}^{2}c_{K}\symbol{\mathbf{1}}$ on the RHS, here $c_{K}$ is a $\delta$-independent finite constant which is formally given by $c_{K} = \int dz\ K(z)$. We will later make the convention that $\symbol{\mathcal{I}}$ should encode integration with just a piece of the heat kernel, defined so that it annihilates constants which means $c_{K}$ will vanish. See Section \ref{sec: integration}.}. $M^{\delta}$ is given by the identity on all remaining abstract symbols for this regularity structure.\footnote{Note that in order to allow for renormalization of the canonical model this regularity structure has more symbols than those listed in Table \ref{Phi43symbols}, in particular it suffices to take $\gamma = 3/2$ when defining $\mathfrak{T}$. Taking $\gamma$ larger does not create any new technical difficulties since $\mathfrak{T}_{-}$ remains the same.}

\medskip
The map $M^{\delta}$ has been defined so that if one views the objects $\mathbf{\hat{\Pi}^{\delta}}\tau$ as \emph{random} space-time distributions then they converge in probability in the $\delta \downarrow 0$ limit.
The canonical model $(\Pi^{\delta},\Gamma^{\delta})$ is an enhancement of the simpler map $\mathbf{\Pi^{\delta}}$. It is possible to postulate suitably flexible conditions on $M^{\delta}$ (for details see \cite[Sec. 8.3]{RegStr}) which would guarantee that one can build another model $(\hat{\Pi}^{\delta},\hat{\Gamma}^{\delta})$, called a renormalized model, which is an analogous enhancement of the map $\mathbf{\hat{\Pi}^{\delta}}$. In the case of the $\Phi^{4}_{3}$ and many other examples the canonical and renormalized models will satisfy the ``diagonal identity''
\begin{equation}\label{renormident}
(\hat{\Pi}^{\delta}_{z}\tau)(z)\ 
=\ 
(\Pi^{\delta}_{z}M^{\delta}\tau)(z).
\end{equation}
Defining this process for general regularity structures is non-trivial (see \cite[Sec 7.1]{hairer2015introduction} for a summary). We remark that the relationship between the renormalized model and canonical model is fairly complex, in general one \emph{does not} have the equality $\hat{\Pi}^{\delta}_{z}\tau = \Pi^{\delta}_{z}M^{\delta}\tau$. 

\medskip
The renormalization procedure can be seen as a deformation of the multiplicative structure of the canonical model. The first three relations of \eqref{canonicalpidef} are essential properties that in practice we always require from models. The maps $\hat{\Pi}^{\delta}_{z}$ will satisfy the first three relations of \eqref{canonicalpidef} but they will not satisfy the last relation.  In the limit $\delta = 0$  the $\hat{\Pi}_z \tau$ will be distributions and it is a priori not even clear how  this condition could  be interpreted. One should then really  view the left hand side of this expression as a \emph{definition} of the right hand side.

%It can be checked that the particular choice of $M^{\delta}$ made above for $\Phi^{4}_{3}$ does satisfy these conditions, and the corresponding random models $(\hat{\Pi}^{\delta},\hat{\Gamma}^{\delta})$ do converge in probability to a limiting random model $(\hat{\Pi},\hat{\Gamma})$. 

\subsection{The deterministic analysis}

It is an amazing fact about the theory of regularity structures,  that once the stochastic calculations that lead to the renormalised model have been performed, the 
deterministic argument that shows the short time existence of solutions, follows ``automatically''. We will now proceed to show how. 

\medskip
Instead of solving the desired equation in a space of $\R$ valued function, we will now solve the equation in a space of \emph{modelled distributions}, i.e. functions taking values in $T$. We stress again that this space depends on the specific realisation of the model (which is in turn constructed as random variables on a suitable probability space). In the case of $\Phi^4_2$ this ``lifted'' solution will take values in the linear space spanned by $\symbol{\mathbf{1}}$ and $\<s1>$. Furthermore, the coefficient for $\<s1>$ will be one, so that we will have
\begin{align*}
\Phi(z) =  \<s1> + \Phi_{\symbol{\mathbf{1}}} (z)\symbol{\mathbf{1}} \;.
\end{align*}
Here, the function $\Phi_{\symbol{\mathbf{1}}}$ which describes the solution $\Phi$ at order $0$ corresponds exactly to the remainder $v$ we already saw in the Da Prato - Debussche argument in Section~\ref{dpdsec}.
In the case of $\Phi^4_3$ we need more terms that describe the solution $\Phi$ and again several of these will be constant. In the end, we  have
\begin{align}\label{e:ansatzPhi}
\Phi(z) =  \<s1> + \Phi_{\symbol{\mathbf{1}}} (z)\symbol{\mathbf{1}}   -\<s30> +  \Phi_{\<s20>}(z) \<s20> + \langle \Phi_{\symbol{\mathbf{X}}}(z), \symbol{\mathbf{X} }\rangle \,,
\end{align}
where we have used the notation $\langle \Phi_{\symbol{\mathbf{X}}}, \symbol{\mathbf{X} }\rangle  = \sum_{j=1}^3 \Phi_{\symbol{\mathbf{X}_j}} \symbol{\mathbf{X}_j}$. Furthermore, we will see that the structure of the equation dictates that $ \Phi_{\<s20>} = -3  \Phi_{\symbol{\mathbf{1}}} $, so that solving for $\Phi$ really involves solving for a system of two functions (one real valued and one vector-valued) $\Phi_{\symbol{\mathbf{1}}}$ and $\Phi_{\symbol{\mathbf{X}}}$. We will justify this particular form of the expansion in Section~\ref{sec:FPA} below.
It is important to note that in both cases the description of the solution $\Phi$ requires much fewer symbols than contained in the regularity structure. The remaining symbols will be used to define the non-linear operations in the fixed point map. The three operations  we need are reconstruction, multiplication and  integration. We will now proceed to explain each of these operations.

\subsubsection{The Reconstruction Theorem}\label{sec:reconstruction}
The fundamental link between the abstract definitions/machinery we have introduced and the concrete results described by Metatheorem \ref{metatheorem} is  the Reconstruction Theorem which states that there is a reconstruction operator $\mathcal{R}$ which establishes  a correspondence between modelled distributions of strictly positive regularity and actual space-time distributions. 
\def\reconstructionRef{\cite[Thm. 3.10]{RegStr}}
\begin{theorem}[\reconstructionRef]\label{ReconstructionTheorem}
Let $\TT = (A,T,G)$ be a regularity structure, let $\alpha := \textnormal{ min } A$, and $r := \lceil -\alpha \rceil$. Let $(\Pi,\Gamma)$ be a model for $\TT$. For any $\gamma > 0$ we set $\mathcal{D}^{\gamma}$ to be the corresponding space of modelled distributions based on the model $(\Pi,\Gamma)$.

Then there exists a continuous linear map  $\mathcal{R}: \mathcal{D}^{\gamma} \rightarrow \Cs^{\alpha}$ with the property that for any compact set $\mathfrak{K} \subset \R^{d+1}$ and any $F \in \mathcal{D}^{\gamma}$ one has that $\mathcal{R}F$ is the unique space-time distribution satisfying
\begin{equation}\label{reconstructionbound}
\left| \left(\mathcal{R}F \right)(\CS_z^{\lambda} \eta) 
- 
\left(\Pi_{z}F(z) \right)(\CS_z^{\lambda} \eta) 
\right|
\lesssim 
\lambda^{\gamma}
\end{equation}
uniformly over test functions $\eta \in B_{r}$, $\lambda \in (0,1]$ and uniformly over compacts in $z$. 

Furthermore, if the given model $(\Pi,\Gamma)$ takes values in continuous functions (this means that $\left(\Pi_{z} \tau \right)$ is always a continuous function) then $\mathcal{R}F$ will also be a continuous function and one has the identity
\begin{equation}\label{reconstructionidentity}
\left( 
\mathcal{R}F
\right)
(z)
=
\left( \Pi_{z}F(z)\right)(z)
\end{equation}
\end{theorem}

For a given modelled distribution $F$ the space-time distribution $\mathcal{R}F$ is constructed as the limit $\mathcal{R}F := \lim_{n \rightarrow \infty} \mathcal{R}^{n}F$ where $\mathcal{R}^{n}F$ is built by ``stitching together" the distributions $\{ \left(\Pi_{z}F(z)\right)\}_{z \in \Lambda_{n}}$ where $\Lambda_{n} \subset \R^{d+1}$ is a discrete set of grid points with resolution $2^{-n}$. More precisely, $\mathcal{R}^{n}F$  is defined as
\begin{equation}\label{ConstrRecOp}
\left( \mathcal{R}^{n}F \right)(z) =
\sum_{x \in \Lambda_{n}}
\left( \Pi_{x} F(x) \right)(\psi^{n}_{x})
\psi^{n}_{x}(z)
\end{equation}
where the functions $\{ \psi_x^{n} \}_{x \in \Lambda_{n}}$ are appropriately scaled and translated copies of a fairly regular function\footnote{In particular, the proof of the Reconstruction Theorem in \cite{RegStr},\cite{hairer2015introduction} uses wavelet analysis.}
\footnote{The reader is encouraged to compare \eqref{ConstrRecOp} with the proof of Theorem~\ref{Kolmogorov}.} $\psi$.
Establishing the convergence $\lim_{n \rightarrow \infty} \mathcal{R}^{n}F$ uses the algebraic and analytic conditions imposed by Definitions \ref{defofmodel} and \ref{modeldistdef} in a very direct manner. In fact, the Reconstruction Theorem served as the initial motivation for the abstract setting of regularity structures. A nice exposition of the proof can be found in \cite{hairer2015introduction}.

\medskip
One can also let the models in Theorem \ref{ReconstructionTheorem} vary, then the theorem gives the existence of a reconstruction map $\tilde{\mathcal{R}}$ acting on triples $(\Pi,\Gamma, F) \in \mathcal{M} \ltimes \mathcal{D}^{\gamma}$. The following theorem gives another essential property of the reconstruction operation.
\begin{theorem}\label{continuityofrecon}
Let $\mathcal{T}$ and $\alpha$ be as in Theorem \ref{ReconstructionTheorem}. Then for any $\gamma > 0$ the reconstruction map $\tilde{\mathcal{R}}:\mathcal{M} \ltimes \mathcal{D}^{\gamma} \rightarrow \Cs^{\alpha}$ is continuous.
\end{theorem}

\subsubsection{Multiplication}\label{subsec: multiplication}
We aim to \emph{lift} the non-linear fixed point problem \eqref{e:Picard1} to the level of modelled distributions. 
This will take the form 
\begin{align}\label{FixedPointAbst1}
\Phi = - \KKg \Phi^3 + \KKg \symbol{\Xi} \;,
\end{align}
where $\KKg$  is a linear operator acting on a space of modelled distribution corresponding to the convolution with the heat kernel $K$. We will discuss the definition of $\KKg$ in the next section. 
For now we start with the definition of the operation $U \mapsto U^3$, i.e. we have to define multiplication of certain modelled distributions.

\medskip
The  product of modelled distributions is defined pointwise on the level of the regularity structure.  
\def\Citeprod{\cite[Def. 41]{RegStr}}
\begin{definition}[\Citeprod]\label{def:prod}
A mapping $T \times T \ni (a,b) \mapsto a  b \in T$ is called a \emph{product} if it is bilinear and 
\begin{itemize}
\item It respects the orders in the sense that for $a \in T_\alpha$ and $b \in T_\beta$ we have $a  b \in T_{\alpha +\beta}$,
\item we have $\symbol{\mathbf{1}} \, a = a \, \symbol{\mathbf{1}} = a $ for all $a \in T$.
\end{itemize}
\end{definition}
We have already seen the actual construction of this  product as a part of the construction of the regularity structure. In the case of the 
regularity structure for $\Phi^4_3$ we have, for example 
\begin{align*}
\symbol{\mathbf{X}^k \;\mathbf{X}^{\ell}} = \symbol{\mathbf{X}^{k+\ell}} ,\quad \<s1> \; \<s1> = \<s2>, \quad \<s30> \; \<s1> = \<s31>  , \ldots  
\end{align*}
and the product is extended  in a bilinear way.
It is important to observe, that many products that could be built from the entries in table \ref{Phi43symbols} do not have a natural definition. 
For example, we have not included a symbol for  $\<s3>  \; \<s2>$ or for any product involving $\symbol{\Xi}$ in the regularity structure. This is because the regularity structure is tailor-built to include only those 
symbols that we actually need in the construction of the fixed point map.  In the same way, we will set $\symbol{\mathbf{X}^k \mathbf{X}^\ell} = 0$ as soon as the order $|k + \ell |_{\mathfrak{s}} \geq 2 $.
In order to satisfy the assumption of Definition~\ref{def:prod} we can always define such products to 
be zero.

\medskip

Our aim for this section is to prove a ``multiplicative" inequality in the spirit of Theorem~\ref{multiplication} for modelled distributions. To this end we need to make sure that 
the product  is compatible with spatial translations, represented by the group $G$.  More precisely, we need to assume that the product is $\gamma$-\emph{regular} which means that the identity 
\begin{align}\label{e:ReGular}
\Gamma(a \; b) = \Gamma a \;\Gamma b \;.
\end{align}

holds for all  $\Gamma \in G$  and all ``relevant'' $a, b \in T$ of order $\leq \gamma$, where $\gamma$ is the order up to which we aim to describe the product   \footnote{See \cite[Def. 4.6]{RegStr} for precise definition.}. 
In the case of  $\bar{\mathcal{T}} $, the regularity structure of polynomials,   this condition reduces to the trivial identity $(x - h)^k \, (x-h)^\ell = (x-h)^{k + \ell}$. In the recursive definition of the canonical model \eqref{e:ReGular} in conjunction with the action of the integration map $\symbol{\mathcal{I}}$ completely determines the action of the $\Gamma_{x y}$. However, it is non-trivial to construct the renormalised models $(\hat{\Pi}^\delta, \hat{\Gamma}^\delta)$ in such a 
way that \eqref{e:ReGular} remains true.

\medskip
In order to state the main result of this section, we need to introduce one more notion -- let $F \in \CD^\gamma$ be a modelled distribution as defined in \eqref{modeldistdef}. We will say that $F \in \CD^\gamma_\alpha$ if $F$ takes values in a subspace of $T$ which is spanned by symbols of order $\geq \alpha$. Note that a non-trivial modelled distribution must have a component of order $\leq 0$ so that necessarily $\alpha \leq 0$.
\def\MultRef{\cite[Thm. 4.7]{RegStr}}
\begin{theorem}[\MultRef] \label{thm:abstractSchaud}
Let $F,G$ be modelled distributions over a regularity structure endowed with a product as explained above. If $F \in \CD^{\gamma_1}_{\alpha_1}$ and $G \in \CD^{\gamma_2}_{\alpha_2}$ we have 
 $F  G \in \mathcal{D}^\gamma_{\alpha_1 + \alpha_2}$ for  $\gamma = (\gamma_1 + \alpha_2) \wedge (\gamma_2 + \alpha_1)$. Furthermore, we have for every compact set $\mathfrak{K}$
\begin{align}
\| F G\|_{\gamma, \mathfrak{K}} \ls \| F\|_{\gamma_1,\mathfrak{K}} \, \| G\|_{\gamma_1, \mathfrak{K}}\;.
\end{align}
\end{theorem}

\begin{remark}
Unlike Theorem~\ref{multiplication}, this theorem does not require any condition on the exponents. Indeed, the product is always well-defined pointwise,  independently of the choice of $\gamma_1$ and $\gamma_2$. However, we encounter a condition on $\gamma$ when applying the \emph{reconstruction operator}. Although the product is always defined as an abstract expansion,  it is only for $\gamma>0$ that this expansion uniquely describes  a \emph{real} distribution.
\end{remark}
\begin{example}
We have seen above in \eqref{liftofholder} that any $\Cs^\gamma$ function $f$ can be lifted naturally to a modelled distribution in $\CD^\gamma$ (which should be denoted by $\CD^\gamma_0$ here)  by setting 
\begin{equation}
F(z) = 
\sum_{ |k|_{\mathfrak{s}} \le \gamma }
\frac{1}{k!}
D^{k}f(z)
\Xk  .
\end{equation}
If we have another $\Cs^\gamma$ function $g$ which is lifted to $G$ in the same way we get 
\begin{align*}
F\; G(z) = \sum_{|k|_{\mathfrak{s}} \le \gamma} \frac{1}{k!} \sum_{j} \binom{k}{j}\big( D^j f(z) D^{k-j} g(z) \big)  \Xk \;,
\end{align*} 
which is nothing but Leibniz rule. Note, that here we have truncated the expansion to involve only those terms of order $\leq \gamma$. Indeed, the function $F G$ is only of class $\Cs^\gamma$ and polynomials of order higher than $\gamma$ give no information about the local behaviour of this function.
\end{example}
\begin{example}
Now we can finally  explain up to which order we need to expand $\Phi$ in order to solve the abstract fixed point problem for $\Phi^4_3$.  As in \eqref{e:ansatzPhi} we make the ansatz
\begin{equation}\label{Phi43ansatz}
\Phi(z) =  \<s1> + \Phi_{\symbol{\mathbf{1}}}(z)\symbol{\mathbf{1}}   - \<s30>  -3  \Phi_{\symbol{\mathbf{1}}}(z)\<s20>  + \langle \Phi_{\symbol{\mathbf{X}}}(z), \symbol{\mathbf{X} }\rangle.
\end{equation}
The term of lowest order in this description is the symbol $\<s1>$ which is of order $-\frac12 -\kappa$. Then we get
\begin{equation}\label{Phicubed}
\Phi^3 :=   \<s3> + 3 \Phi_{\symbol{\mathbf{1}}} \<s2> - 3 \<s32>  + 3 \Phi_{\symbol{\mathbf{1}}}^2 \<s1> - 6 \Phi_{\symbol{\mathbf{1}}} \<s31> - 9 \Phi_{\symbol{\mathbf{1}}}\<s22> +3 \langle \Phi_{\symbol{\mathbf{X}}} , \<s2> \symbol{\mathbf{X}} \rangle+   \Phi_{\symbol{\mathbf{1}}}^3 \symbol{\mathbf{1}}    \;,
\end{equation}
where we included only terms of non-positive order.
Using Theorem~\ref{multiplication} we can conclude that for $\Phi \in \CD^\gamma_{\alpha}$ we have $\Phi^3 \in \CD^{\gamma - 2 \alpha}_{3\alpha}$. This statement is always true for any $\gamma$, but in order to have a meaningful \emph{reconstruction} of $\Phi^3$ the exponent $\gamma + 2 \alpha$ needs to be strictly positive. As $\alpha = - \frac12 -\kappa$  we need to describe $\Phi$ to order at least $\gamma > 1 + 2\kappa$. 
\end{example}

\subsection{Integration}\label{sec: integration}
At this stage, the only operation missing to define the fixed point operator \eqref{FixedPointAbst}  is the \emph{integration map}. Recall that above in Theorem~\ref{Schaudestimate} we had stated that convolution with the heat kernel $K$ can be defined for quite general distributions $f \in \Cs^\alpha$ and that (for $\alpha \notin \Z$) this operation improves the parabolic H\"older regularity by $2$. This result is closely related to the fact that $K$ is a singular kernel of order $-d_{\mathfrak{s}}+2$, i.e. that $K$ is a smooth function on $\R \times \R^d \setminus\{ 0\}$  with a well-controlled singularity at the origin. Our aim for this section is to define an analogue map $\KKg$ that maps modelled distribution $F \in \CD^\gamma_\alpha$ to \rf{$\CD^{\gamma+2}_{(\alpha+2) \wedge 0}$} \refcomment{REF20}. To make some expressions easier,  we will from now on use a convention (slightly inconsistently with the previous sections but consistently with the notation used in \cite{RegStr}) to give a new interpretation to the kernel $K$. We will replace the parabolic heat kernel by a kernel $K$ which satisfies
\begin{itemize}
\item[1.)] $|D^{k } K(z)| \lesssim \| z \|_{\mathfrak{s}}^{-d_{\mathfrak{s}}+2 -|k|_{\mathfrak{s}} }$ for all multi-indices $k$ (recall the definition of the parabolic dimension $d_{\mathfrak{s}} = 2+d$). %and $|k|_{\mathfrak{s}} = 2k_0 + k_1 + \ldots k_d$ for any multi-index $k = (k_0, \ldots ,k_d)$).
\item[2.)] $K(z) =0$ for all  $z = (t,x)$ with $t<0$.
\item[3.)] $K$ has compact support in $\{z \colon \| z\|_{\mathfrak{s}}< 1\}$.
\item[4.)] $\int K(z)  z^k \, dz = 0$ for all multi-indices $k$ with $|k|_{\mathfrak{s}} < \gamma$.
\end{itemize}
Of course, the Gaussian heat kernel satisfies  assumptions 1.) and 2.) but not 3.) and 4.). However, for any $\gamma>0$ it is possible to add a smooth function $R_\gamma$ to the Gaussian kernel such that one obtains a kernel that 
also satisfies the assumptions 3.) and 4.). The convolution  with a smooth function $R_\gamma$ is an infinitely smoothing operation which can easily be dealt with separately. Therefore, for the rest of these notes we will assume that $K$ satisfies 
all of these four assumptions, neglecting the extra terms that come from the convolution with $R_\gamma$.

\medskip

\begin{example}\label{ex:integrationPolys}
Let us briefly discuss how to formulate a version of the classical Schauder estimate for $\mathcal{C}_{\mathfrak{s}}^{\gamma}$ functions (Theorem~\ref{Schaudestimate}) using the regularity structure of abstract polynomials defined in Sec \ref{abstractpoly}. Let $f \colon \R \times \R^d \to \R$ be 
of regularity $\mathcal{C}^\gamma_{\mathfrak{s}}$ for some $\gamma \in (0,\infty) \setminus \Z$ and let $F$ be its canonical lift to $\mathcal{D}^{\gamma}$.  We want to define a map $\KKg$ acting on $\mathcal{D}^{\gamma}$ which represents convolution with the  kernel $K$. A natural definition would be
\begin{align}\label{e:Int1}
\KKg F(z) =  \sum_{|k|_{\s} < \gamma +2} 
\frac{\Xk}{k!}
 \int_{\R \times \R^d} D^{k} K(z-y) \underbrace{\mathcal{R} F(y)}_{=f(y)} \, dy
\end{align}
but it is not obvious why this integral converges. 
Indeed, in general $|D^k K(z)| \sim |z|^{2 - d_{\mathfrak{s}} -|k|_{\mathfrak{s}}}$ which fails to be integrable for $|k|_{\mathfrak{s}} \geq 2$.
However, if we use the regularity of $F$ and replace $\mathcal{R} F(y)$ by $\mathcal{R} F(y) - \Pi_z F(y)$ (i.e. if we subtract the Taylor polynomial around $z$) we obtain a convergent expression
\begin{align}\label{e-def-Ng}
\Ng F(z) := 
 \sum_{|k|_{\s} < \gamma +2} 
\frac{\Xk}{k!}  \int_{\R \times \R^d}  \underbrace{D^{k} K(z-y)}_{\ls \|z-y\|_{\mathfrak{s}}^{2 - d_{\mathfrak{s}} - |k|_{\mathfrak{s}}  }}  \underbrace{\big( \mathcal{R} F(y)  -  \Pi_z F(y) \big)}_{\|z-y\|_{\mathfrak{s}}^\gamma } \, dy  \;.
\end{align}

\medskip

It remains to discuss, how to interpret the integrals of $ D^{k} K(z - \cdot)  $ against $\Pi_z F  $. Note that this expression depends only \emph{locally} on $F$ at the point $z$ and it is completely determined by   
$\int D^{k} K(z) \, z^{\ell} dz $ for finitely many $k$ and $\ell$. But of course, (some of) these integrals still fail to converge absolutely. However, using the formal integration by part
\begin{equation*}
\int_{\R \times \R^d} D^{(k)} K(z) z^{\ell} \, dz = (-1)^{|k|} \int_{\R \times \R^d}  K(z) D^{(k)} z^{\ell } \, dz  ,
\end{equation*}
it is not hard to show that these integrals converge if they are interpreted as principal values
\begin{equation}\label{integral-against-polys}
 \lim_{\eps \to 0}  \int_{|z| > \eps} z^\ell D^{k} K(z) \, dz  \;,
\end{equation}
and furthermore, using our convenient assumption that $K$ integrates to zero against polynomials, one can see that 
these limits are zero.
We can therefore define 
%the linear operator $R_\gamma \colon T \to T$ by its action on the monomials 
%\begin{equation}\label{e-def-R}
%R_\gamma \colon \symbol{\mathbf{X}^{\ell}} \mapsto \sum_{k < \gamma+2} \frac{a_{k,\ell }}{k!} \Xk   \;,
%\end{equation} 
 $\KKg = \Ng$ and the operator $\KKg$ defined this way does indeed map $\mathcal{D}^\gamma_0$ to $\mathcal{D}^{\gamma+2}_0$.

\end{example}

\medskip

Of course our main focus is on models and regularity structures that are larger than $\bar{\mathcal{T}}$ which allow us to work with singular distributions. In this case the integration map $\KKg$ will in general consist of three different components 
\begin{equation*}
\KKg = \symbol{\mathcal{I}} + \mathcal{J}_\gamma  + \Ng  \;.
\end{equation*}
We have already encountered the operator $\Ng$ in Example~\ref{ex:integrationPolys} above. The defining equation \eqref{e-def-Ng}  remains meaningful if it  is interpreted in a distributional sense. The operators $\symbol{\mathcal{I}}$ and $\mathcal{J}_\gamma$ correspond to the additional information provided by non-polynomial symbols. Both operators vanish when applied to abstract polynomials.

\medskip
The operator $\symbol{\mathcal{I}}$ takes values in the span of the abstract symbols $\mathfrak{T}$ which are \emph{not} polynomials.
Like the multiplication, it is defined as a linear operator on $T$ and its definition was essentially part of the recursive construction of the regularity structure $\mathcal{T}$. We have, for example 
\begin{align*}
\symbol{\mathcal{I} \Xi} = \<s1> , 
\end{align*}
etc. Note in particular, that (unlike a convolution operator) $\symbol{\mathcal{I}}$ acts \emph{locally} on modelled distributions.  One important property is that by definition $\symbol{\mathcal{I}} $ increases the order of every $a \in T_\alpha$ by $2$.
As for the product, we do not have to give a non-trivial interpretation for $\symbol{\mathcal{I}} \tau$ for all $\tau$. Indeed, in order to describe our solution to a certain regularity it is sufficient to keep those basis elements 
of order $< \gamma$.

\medskip
Of course,  at this stage the definition of the \emph{abstract integration map} $\symbol{\mathcal{I}}$ has nothing to do with the kernel $K$. The connection with $K$ is established in the choice of the model. We had already discussed this issue in the context of the 
\emph{canonical model} in \eqref{canonicalpidef}. We will now turn the relevant property of the canonical model into a definition.
\begin{definition}\label{def: admissible}
A model is \emph{admissible}  if  we have $(\Pi_z\mathbf{X}^k)(y) = (y-z)^k$ and
\begin{align*}
\left( \Pi_{z}\symbol{\mathcal{I}}[\tau]\right)(\bar{z})
&:= \int_{\R^{d+1}}dy\ K(\bar{z} - y ) \left( \Pi_{z}\tau \right)(y)\\
&\hspace{1cm}
- 
\sum_{
\substack{k \\
|k|_{\mathfrak{s}} < |\tau| + 2
}
}
\frac{
(\bar{z} - z)^{k}}
{k!}
\int_{\R^{d+1}}dy\ (D^{k}K)(z - y ) \left( \Pi_{z}\tau \right)(y) \;.
\end{align*}
We will denote by $\mathcal{M}_{0} \subset \mathcal{M}$ the space of admissible models. 
\end{definition}
The construction of \emph{canonical models} we explained above, automatically produces admissible
models\footnote{Of course we discussed the construction of canonical models in the context of the Gaussian heat kernel, but the construction goes through unchanged if it is replaced by our modified kernel.}.
But  it is non-trivial to perform the renormalization such that the models  remain admissible. 

\medskip 
The operator $\mathcal{J}_\gamma$ takes values in the abstract polynomials. It is the analog of the integrals \eqref{integral-against-polys} which in the case of a general regularity structure cannot be 
removed by a convenient choice of kernel. The operator is defined as 
\begin{equation}\label{e:DefJ}
\mathcal{J}_\gamma F (z) := \sum_{|k|_{\s} < \alpha +2} 
\frac{\Xk}{k!}
\int D^{(k)} K(z-y)  \;
	\Pi_z F  (y)
dy \;,
\end{equation}
%  Note, that formally 
% we have the natural identity
% \begin{align*}
% ( \mathcal{J}_\gamma  + \Ng) F(z) =  
% \sum_{|k|_{\s} < \gamma +2} 
% \frac{\Xk}{k!}
% \int D^{(k)} K(z-y)  \;
% 	\mathcal{R} F  (y)\;,
% \end{align*}
% %
% resembling \eqref{e:Int1}.

\medskip
With these definitions in place we have the following result.
\begin{theorem}
Let $\mathcal{T}$ be a regularity structure endowed with an admissible model and assume that $\gamma \notin A$. Then the operator $\KKg$ is compatible with integration against the kernel $K$ in the sense that
\begin{equation}\label{reconstructionintegration}
\mathcal{R} \KKg F = K \star \mathcal{R}F \;.
\end{equation}
Furthermore $\KKg$ maps $\mathcal{D}^\gamma_\alpha$ into \rf{$\mathcal{D}^{\gamma+2}_{(\alpha+2) \wedge 0}$} and we have for every compact $\mathfrak{K}$ \refcomment{REF21}
\begin{equation*}
\| \KKg F \|_{\gamma +2, \mathfrak{K}} \ls \| F\|_{\gamma, \bar{\mathfrak{K}}} \;.
\end{equation*}
where $\bar{\mathfrak{K}}  = \{ z \colon \;\inf_{\bar{z} \in \mathfrak{K}} \| z - \bar{z} \|    \leq 1 \}$.

\end{theorem}

\begin{example}\label{ex:integration}
For $\Phi^4_2$ we had made the ansatz $\Phi(z) = \<s1> + \Phi_{\symbol{\mathbf{1}}} \symbol{\mathbf{1}}$ and we had already seen above that this implies
\begin{align*}
\Phi^3 =  \<s3> + 3  \Phi_{\symbol{\mathbf{1}}} \<s2> +   3 \Phi_{\symbol{\mathbf{1}}}^2 \<s1> + \Phi_{\symbol{\mathbf{1}}}^3 \symbol{\mathbf{1}}\;.
\end{align*}
Let us now give an explicit description of,  $\KKg \symbol{\Xi} -  \KKg \big[  \Phi^3  \big]$ because it is instructive. First of all, we have $\KKg \symbol{\Xi} = \symbol{\mathcal{I}[ \Xi  ]}= \<s1> $.  Indeed, $\symbol{\mathcal {I } [\Xi ]}$ has order $-\kappa < 0$ and therefore, the sum \eqref{e:DefJ} which defines the operator $\symbol{\mathcal{J}}$ is empty. Furthermore,  the fact that  $\Pi_z \Xi $ does not depend on $z$ implies that  $(\mathcal{R}\Xi)(z) = \Pi_z \Xi$ which in turn implies by \eqref{e-def-Ng} that $\Ng \Xi$ vanishes as well. 

\medskip
On the other hand, we  can take $ -  \KKg \Phi^3 = -( \mathcal{J}_\gamma  + \Ng)  \Phi^3$. Indeed, the symbols $\<s1>, \<s2>, \<s3>$ appearing in $\Phi^3$ have order $-\kappa, -2 \kappa, -3 \kappa$ so that  the abstract integration map $\symbol{\mathcal{I}}$ acting on these symbols would produce terms of order $>1$. We do not require a description to such order, so these terms and the corresponding $\symbol{\mathcal{J}}$ can be dropped. We get
\begin{align*}
( \mathcal{J}_\gamma  + \Ng) \big( \Phi^3  \big) =     K \ast  \mathcal{R}  \Big( \<s3> + 3  \Phi_{\symbol{\mathbf{1}}} \<s2> +   3 \Phi_{\symbol{\mathbf{1}}}^2 \<s1> + \Phi_{\symbol{\mathbf{1}}}^3 \symbol{\mathbf{1}} \Big) \; \symbol{\mathbf{1}} \;.
\end{align*}
The reconstruction operator gives $\mathcal{R} \<s3>  = \<3>$, $\mathcal{R} \Phi_{\symbol{\mathbf{1}}} \<s2>  = \Phi_{\symbol{\mathbf{1}}} \<2> $, $\mathcal{R} \Phi_{\symbol{\mathbf{1}}}^2 \<s1>  = \Phi_{\symbol{\mathbf{1}}}^2 \<1>$ and  $\mathcal{R} \Phi_{\symbol{\mathbf{1}}}^3 \symbol{\mathbf{1}} = \Phi_{\symbol{\mathbf{1}}}^3$
\footnote{Actually, the processes $\<1> $, $\<2>$ and $\<3>$ should be constructed with the modified  kernel.}.
Hence, the equation for $\Phi_{\symbol{\mathbf{1}}}$ reduces to the equation for the remainder $v_\delta$ in the Da Prato - Debussche method. In this context, the continuity of the multiplication of modelled distribution, together with the existence and continuity of the reconstruction operator take the role of the multiplicative inequality, Theorem~\ref{multiplication}.  This is actually a general fact - in \cite[Thm 4.14 ]{RegStr} it is shown how Theorem~\ref{multiplication} can be derived as a consequence of these two statements.
\end{example}

\subsection{The fixed point argument}\label{sec:FPA}
We now state a theorem guaranteeing the existence of a modelled distribution which solves the abstract fixed point problem \eqref{FixedPointAbst}. Our discussion will be informal; a precise version of such a theorem, stated in a quite general context, can be found in \cite[Sec 7.3, Theorem 7.8]{RegStr}.  
\medskip

%Recall that one way
We aim  to prove existence of solutions to the dynamic $\Phi^4_3$ equation
%to a concrete PDE is to treat the problem as an ODE in infinite dimensions. Finding the solution to this ODE is posed as a 
by solving a fixed point problem 
%on some Banach space $\mathfrak{B}$ 
in a space $\mathcal{D}^\gamma$ of modelled distributions on $[0,T] \times \R^d$. The tools we developed in the previous sections
will show that the non-linearity is locally Lipschitz continuous and for $T$ small enough we can apply the contraction mapping theorem 
on some ball in $\mathcal{D}^\gamma$.
%
%of functions from $[0,T] \rightarrow \mathcal{B}$ where $[0,T]$ is the time interval the solution will live on and $\mathcal{B}$ is some Banach space of functions on $\R^{d}$ (or some subset of $\R^{d}$). 
%When working with non-linear PDE the assumption that the non-linearity is locally Lipschitz allows one to show that for $T$ small enough the contraction mapping theorem can be applied on some ball in $\mathfrak{B}$. This same approach is used for the abstract fixed point theorem. 

\medskip
At this point it is important to remember  that we have derived uniform bounds on models only locally in space-time. Indeed, going through the proof 
of the Kolmogorov Lemma, Theorem~\ref{Kolmogorov}, the reader can easily convince himself that the constants explode over infinite space-time domains and the 
same phenomenon presents itself in the construction of various models discussed at the beginning of this section. This problem could be circumvented, by introducing 
weights in the norms that measure these models (as has been implemented in \cite{Cyril,JCH2}) but this makes the deterministic analysis more difficult. Here we choose the 
simpler situation and compactify space by assuming that the noise is periodic.

\medskip
Accordingly, we now assume that our space-time white noise $\xi$ is defined on $\T^{d} \times \R$. When convenient we interpret $\xi$ as a distribution on $\R^d \times \R$ which is periodic in space.
We will again lift realizations of the noise to admissible models as before
(see \cite[Section 3.6]{RegStr} for the precise notion of periodicity for models).
For any  ``periodic'' model $(\Pi,\Gamma)$ we define $\mathcal{D}^{\gamma}(\Pi,\Gamma,\Lambda_{T})$ to be the family of modelled distributions $F: \T^{d} \times [0,T] \rightarrow T$ 
which satisfy condition \eqref{modeldistdef}.
We define $\mathcal{M}_{0} \ltimes \mathcal{D}^{\gamma}(\Lambda_{T})$ to be the set of all triples $(\Pi,\Gamma,F)$ with $(\Pi,\Gamma) \in \mathcal{M}_{0}$ ``periodic'' and $F \in \mathcal{D}^{\gamma}(\Pi,\Gamma,\Lambda_{T})$. As before it will be important that this space can be equipped with a metric which behaves well with the machinery of the theory of regularity structures. With all of this notation in hand we can now state the following theorem.

\begin{theorem}\label{thm:FPT}
Let $\mathcal{T}$ be a regularity structure for $\Phi^4_3$, where the corresponding list of symbols $\mathfrak{T}$ was defined with a choice of $\gamma > 1 + 3 \kappa$.  Then for any admissible model $(\Pi,\Gamma)$ there exists a strictly positive $T >0$ and a unique modelled distribution $\Phi \in \mathcal{D}^\gamma(\Pi,\Gamma,\Lambda_{T})$ that solves the fixed point equation
\begin{align}\label{FixedPointAbst}
\Phi = - \KKg \Phi^3 + \KKg \symbol{\Xi} \;.
\end{align} 

Additionally, $T$ is lower semicontinuous in $(\Pi,\Gamma)$. Furthermore, if for some $(\bar{\Pi},\bar{\Gamma})$ we have  \rf{$T \left[ (\bar{\Pi},\bar{\Gamma}) \right]>t $} then \refcomment{REF22} $
\Phi : (\Pi, \Gamma) \mapsto (\Pi,\Gamma,\Phi) \in  \mathcal{M}_{0} \ltimes \mathcal{D}^\gamma(\Lambda_{t})$ is continuous in a neighbourhood of $(\bar{\Pi},\bar{\Gamma})$. 
\end{theorem}

We make a few remarks about the contraction mapping argument used for the above theorem. The continuity of the mapping $\Phi \mapsto \Phi^3$ and the integration operator $\KKg$ immediately imply that the mapping $\Phi \mapsto - \KKg \Phi^3 + \KKg \symbol{\Xi} $ is Lipschitz continuous on every ball of $\mathcal{D}^\gamma(\Lambda_t)$. The Lipschitz constant can be made arbitrarily small by choosing a slightly smaller $\gamma$ (which for $t < 1$ produces a small power of $t$ in front of the bounds) and then choosing $T$ small enough. 

Similar arguments yield the essential continuity statement promised by the last sentence of Theorem \ref{thm:FPT}.

\begin{remark}
Up to now we have always assumed that the initial data for our fixed point problem is zero. This is quite unsatisfactory for the solution theory, because it prevents us from \emph{restarting} the solution at $T$ to obtain a maximal solution. The theory of \cite{RegStr} does allow for the restarting of solutions but one must work with larger classes of modelled distributions. Since initial condition will typically not have a nice local description the theory introduces spaces of ``singular'' modelled distributions where the given local expansions are allowed to blow up near the time zero hyperplane (see \cite[Section 6]{RegStr}). 
\end{remark}
We now continue our discussion but will suppress the fact that we are actually working in a spatially periodic setting and finite time horizon.
The particular form of the modelled distribution $\Phi \in \mathcal{D}^{\gamma}$ which solves the abstract fixed point problem can now be deduced by running a few steps of a Picard iteration. We run through this computation now which will end up justifying the ansatz \eqref{e:ansatzPhi}. 
We start the iteration by setting $\Phi_{0}(z) = 0$. Applying the map $\Phi \mapsto - \KKg \Phi^3 + \KKg \symbol{\Xi} $  to $\Phi_0$ gives
\begin{equation*}
\Phi_{1}(z)  = \KKg \Xi = \symbol{\mathcal{I} \Xi} = \<s1>\;.
\end{equation*}
\footnote{Note that the $\Phi_1$ is different from $\Phi_{\symbol{\bf{1}}}$ with a blue bold subscript.   }
Here we used that $\symbol{\mathcal{J} \Xi} = \Ng \symbol{\Xi} =0$ as was explained in Example~\ref{ex:integration}. Applying this map again then gives
\begin{equation*}
\Phi_{2}(z)   = \<s1> - \left[\KKg (\Phi_{1})^3\right](z) =  \<s1> - ( \symbol{\mathcal{I}} + \symbol{\mathcal{J}} + \Ng) \<s3> = \<s1>  - \<s30> -  \<30> (z) \symbol{\mathbf{1}} \;. 
\end{equation*}
Observe that the appearance of the positive order symbol $\<s30>$ automatically produces the first ``polynomial" (the symbol $\symbol{\mathbf{1}}$). Here the notation $\<30> (\bullet)$ (which is colored black) refers to the corresponding concrete space-time distribution which was introduced in Section \ref{dpdsec}.
In going to our expression $\Phi_{2}$ we used that $\Ng \<s3> = 0$, this is because $\Pi_z \<s3>$ does not depend on $z$ and hence $\mathcal{R} \<s3> = \Pi_z \<s3>$.

\medskip
Going one step further in the Picard iteration gives
\begin{align*}
\Phi_{3}(z)   &= \<s1> - \left[ \KKg (\Phi_{2})^3 \right](z)  \\
&= \<1> - \<s30> -  \<30> (z) \symbol{\mathbf{1}}  +  \KKg \Big( 3\<s32>     +3\<30> (z)\<s2>   -  6\<30> (z) \<s31> - 3 \<30>  (z)^2 \<s1>    \Big) \;.
\end{align*}
Here we have dropped  all terms of order $>0$ under the operator $\KKg$ because we do not need them. Indeed, the two requirements that determine the degree to which we have to expand each quantity are: 
\begin{itemize}
\item The solution $\Phi$ should solve a fixed point problem in $\mathcal{D}^{\gamma}_\alpha$ for $\gamma \approx 1 + 3\kappa$ and $\alpha = -\frac12 - \kappa$. Therefore, we need to keep all symbols with order less than or equal to $\gamma$ in the expansion $\Phi$.
\item Below we will apply the reconstruction operator to the right hand side of the fixed point problem in order to identify the equation the reconstruction of $\Phi$ solves. In order to be able to do that we need to ensure that quantities under the integral operator are described to strictly positive order.
\end{itemize}

\medskip
It is now clear the fixed point $\Phi \in \mathcal{D}^{\gamma}$ for the map $\Phi \mapsto - \KKg \Phi^3 + \KKg \symbol{\Xi} $ will have the property that the symbols $\<s1>$ and $\<s30>$ enter with $z$-independent coefficients. Indeed, both symbols only ever arise after integrating the terms $\symbol{\Xi}$ and ${\<s3>}$ from the previous step, both of which cannot have a non-constant prefactor. Furthermore, it is clear why the pre-factor of $\<s20>$ has to be $\Phi_{\symbol{\mathbf{1}}}$. Indeed, this symbol only arises after applying $\symbol{\mathcal{I}}$ to $\<s2>$ which in turn only appears from the multiplication of $\<s2> $ with $\Phi_{\symbol{\mathbf{1}}}\symbol{\mathbf{1}}$.

\medskip

As mentioned above, we will now apply the reconstruction operators to $\Phi$ to get concrete space-time distributions and then show that these objects satisfy certain concrete PDE.
For $\delta > 0$ let $(\Pi^{\delta},\Gamma^{\delta})$ be the canonical model built from the smoothed noise $\xi_{\delta}$. Let $\mathcal{R}^{\delta}$ be the associated reconstruction operator on the $\mathcal{D}^{\gamma}$ space built from the canonical model with $\gamma$ slightly larger than $1$ as in Theorem~\ref{thm:FPT}. We denote by $\Phi_{\delta}$ the modelled distribution which is the solution to the corresponding abstract fixed point problem. It follows that  
\begin{equation*}
\begin{split}
\left(\mathcal{R}^{\delta}\Phi_{\delta}\right)(z)
=&
\mathcal{R}^{\delta}
\left[ \mathcal{K} \left( \symbol{\Xi} - \Phi_{\delta}^{3} \right) \right](z)\\
=&
\int  dy\ 
K(z - y)
\left[
\mathcal{R}^{\delta}
\left( \symbol{\Xi} - \Phi_{\delta}^{3} \right)
\right](y)\\
=&
\int  dy\ 
K(z - y)
\left[ \left(\Pi^{\delta}_{y}\symbol{\Xi}\right)(y) 
- 
\left(\Pi^{\delta}_{y}\Phi_{\delta}^{3}(y)
\right)(y) \right]\\
=&
\int  dy\ 
K(z - y)
\left[ \left(\Pi^{\delta}_{y}\symbol{\Xi}\right)(y) - \left( \Pi^{\delta}_{y}\Phi_{\delta}(y) \right)^{3}(y) \right]\\
=&
\int dy\ 
K(z - y)
\left[ \xi_{\delta}(y) -  \left(\mathcal{R}^{\delta}\Phi_{\delta}\right)^{3}(y) \right].\\
\end{split}
\end{equation*}

The first equality above is just the fixed point relation. The second equality is \eqref{reconstructionintegration}, the third is \eqref{reconstructionidentity}, the fourth is a consequence of the product property of the canonical model, and in the final equality we again use \eqref{reconstructionidentity}.  It follows that $\mathcal{R}^{\delta}\Phi_{\delta}$ is the mild solution to the equation
\[
\partial_{t} \phi_{\delta} = \Delta \phi_{\delta} - \phi_{\delta}^{3} + \xi_{\delta}.
\] 

Now let $(\hat{\Pi}^{\delta},\hat{\Gamma}^{\delta})$ be the renormalized model we introduced earlier, and let $\hat{\mathcal{R}}^{\delta}$ be the reconstruction operator on the associated $\hat{\mathcal{D}}^{\gamma}$ space (with $\gamma$ as before). Let $\hat{\Phi}_{\delta}$ be the solution to the abstract fixed point problem with model $(\hat{\Pi}^{\delta},\hat{\Gamma}^{\delta})$. We then have

\begin{equation}\label{renormalizedeqncomp}
\begin{split}
\left(\hat{\mathcal{R}}^{\delta}\hat{\Phi}_{\delta}\right)(z)
=&
\int  dy\ 
K(z - y)
\left[ \left(\hat{\Pi}^{\delta}_{y}\symbol{\Xi}\right)(y) 
- 
\left(\hat{\Pi}^{\delta}_{y}\hat{\Phi}_{\delta}^{3}(y)
\right)(y) \right]\\
=&
\int  dy\ 
K(z - y)
\left[ \xi_{\delta}(y) - \left(\Pi^{\delta}_{y}M\hat{\Phi}_{\delta}^{3}(y)
\right)(y)  \right].
\end{split}
\end{equation}
In going to the last line we have used \eqref{renormident}.

\medskip
We know that $\hat{\Phi}_{\delta}$ has an expansion of the form \eqref{Phi43ansatz} where the spatially varying coefficients appearing in the expansion are unknown. Since $\left(\Pi^{\delta}_{y} \tau \right)(y) = 0$ for any homogenous $\tau$ with $|\tau| > 0$ we can replace $M^{\delta}\hat{\Phi}_{\delta}^{3}(y)$ with something equivalent modulo symbols of strictly positive order. Denoting this approximate equivalence by $\approx$, the reader is encouraged to apply $M^{\delta}$ to the formula \eqref{Phicubed} to see that

\begin{equation*}
M^{\delta}
\hat{\Phi}_{\delta}^{3}(y) 
\approx
(M^{\delta} \hat{\Phi}_{\delta}(y))^{3}
-
(3C_{\delta} + 9 \tilde{C}_{\delta})
M^{\delta}\hat{\Phi}_{\delta}(y).
\end{equation*}

Applying the canonical model $\Pi^{\delta}_{y}$ to both sides then gives

\begin{equation*}
\begin{split}
\left(\Pi^{\delta}_{y}M\hat{\Phi}_{\delta}^{3}(y)
\right)(y)
&=\ 
\left(
\Pi^{\delta}_{y}
(M^{\delta} \hat{\Phi}_{\delta}(y))^{3}
\right)(y)
-
(3C_{\delta} + 9 \tilde{C}_{\delta})
\left(
\Pi^{\delta}_{y}
M^{\delta}\hat{\Phi}_{\delta}(y)
\right)
(y)\\
&=\ 
\left(\Pi^{\delta}_{y} M^{\delta} \hat{\Phi}_{\delta}(y)
\right)^{3}
(y)
-
(3C_{\delta} + 9 \tilde{C}_{\delta})
\left(
\Pi^{\delta}_{y}
M^{\delta}\hat{\Phi}_{\delta}(y)
\right)
(y)\\
&=\ 
\left(\hat{\Pi}^{\delta}_{y} \hat{\Phi}_{\delta}(y)
\right)^{3}
(y)
-
(3C_{\delta} + 9 \tilde{C}_{\delta})
\left(
\hat{\Pi}^{\delta}_{y}
\hat{\Phi}_{\delta}(y)
\right)
(y)\\
&=\ 
\left( \hat{\mathcal{R}}^{\delta} 
\hat{\Phi}_{\delta}
\right)^{3}
(y)
-
(3C_{\delta} + 9 \tilde{C}_{\delta})
\left(\hat{\mathcal{R}}^{\delta} 
\hat{\Phi}_{\delta}
\right)
(y) \;.
\end{split}
\end{equation*}

Inserting this into \eqref{renormalizedeqncomp} immediately yields that $\left(\hat{\mathcal{R}}^{\delta}\hat{\Phi}_{\delta}\right)(z)$ is the mild solution to the PDE
\begin{equation}\label{renormalizedequation0}
\partial_{t} \phi_{\delta} = \Delta \phi_{\delta} - \phi_{\delta}^{3} + (3C_{\delta} + 9\tilde{C}_{\delta})\phi_{\delta} + \xi_{\delta}.
\end{equation}
We now take advantage of the fact that all of the abstract machinery introduced in this lecture has good continuity properties with respect to the convergence of models. If the models $(\hat{\Pi}^{\delta},\hat{\Gamma}^{\delta})$ converge in probability to a limiting model $(\hat{\Pi},\hat{\Gamma})$ as $\delta \downarrow 0$ then from Theorem~\ref{thm:FPT}  it follows that the triples $( \hat{\Pi}^{\delta},\hat{\Gamma}^{\delta}, \hat{\Phi}_{\delta})$, viewed as random elements of the space $\mathcal{M} \ltimes \mathcal{D}^{\gamma}$, converge in probability to a limiting triple $(\hat{\Pi},\hat{\Gamma}, \hat{\Phi})$ as $\delta \downarrow 0$. 

\medskip
Theorem \ref{continuityofrecon} then implies that the solutions of \eqref{renormalizedequation0}, given by
$\hat{\mathcal{R}}^{\delta}\hat{\Phi}^{\delta}$, converge in probability to a limiting space-time distribution we will call $\phi$. Here convergence in probability takes place on the space $\Cs^{\alpha}$. While one may not be able to write down an explicit SPDE that the $\phi$ satisfies, we can say $\phi$ solves the abstract formulation of the given SPDE since $\phi = \tilde{\mathcal{R}}\left[(\hat{\Pi},\hat{\Gamma},\hat{\Phi})\right]$ and the triple $(\hat{\Pi},\hat{\Gamma},\hat{\Phi})$ is a solution to our abstract fixed point problem.

\appendix

\section{A primer on iterated stochastic integrals}
\label{s:AppA}
In this appendix we collect  some facts about iterated stochastic integrals used in Lecture~\ref{s:l2}. Our discussion is brief and somewhat formal - a detailed exposition can be found in \cite[Chapter 1]{Nualart}. Throughout the appendix we adopt a slightly more general framework than in Lecture~\ref{s:l2} and replace $\R \times \R^d$ or $\R \times \T^d$  by an arbitrary measure space $(E,\mathcal{E})$ endowed with a  sigma-finite non-atomic measure $\mu$.  Extending the definition presented in Section \ref{ss:whiteNoise} a  \emph{white noise} is then defined as a centred Gaussian family of  random variables $(\xi,\varphi)$ indexed by $\varphi \in L^2(E,\mu)$ which satisfy
\begin{equation}\label{e:WNCovariance}
\E (\xi,\varphi_1) (\xi,\varphi_2) = \int_E \varphi_1 (z) \, \varphi_2(z) \, \mu(dz) \;.
\end{equation}
It is particularly interesting to  evaluate $\xi$ at indicator functions $\mathbf{1}_A$ of measurable sets in $A \in \mathcal{E}$ with $\mu(A)<\infty$ and we write $\xi(A) $ as a shorthand for $\xi(\mathbf{1}_A)$. The following properties follow
\begin{itemize}
\item $\E \xi(A) = 0$ and $\E \xi(A)^2 = \mu(A)$.
\item If $A_1$ and $A_2$ are disjoint, then $\xi(A_1)$ and $\xi(A_2)$ are independent.
\item If $(A_j)_{j \in \N} \ldots$ are pairwise disjoint and of finite measure, then $\xi(\cup_j A_j) = \sum_j \xi(A_j)$, where the convergence holds in $L^2(\Omega, \P)$.
\end{itemize}
Although the last identity suggests to interpret $A \mapsto \xi(A)$ as a random signed measure, it is important to note that in general the $\xi$ does not have a modification as a random signed measure (cf. the regularity discussion in Besov spaces above).

\medskip
We now discuss, how to construct an iterated stochastic integrals of the type ``$ \int_{E^n } f(z_1, z_2, \ldots , z_n) \xi(dz_1)  \ldots \xi(dz_n)$'' for $f \in L^2( E^n  , \mu^{ \otimes n})$. For simplicity we will only treat the case $n=2$, the general case of $n$-fold iterated integrals following in a similar way. In this case, we will call \emph{elementary} any $f \colon E \times E \to \R$ of the form
\begin{equation*}
f = \sum_{\substack{j,k=1 \\ j \neq k}}^N \alpha_{j,k} \mathbf{1}_{A_j \times A_k} \;,
\end{equation*}
for pairwise disjoint sets $A_1, \ldots, A_N$ with finite measure and real coefficients $\alpha_{j,k}$. Note that such a function $f$ is necessarily zero on the diagonal $f(z,z) = 0$ for $z \in E$. We define for such an $f$
\begin{align*}
\int_{E \times E} f(z_1, z_2) \xi(dz_1) \, \xi(dz_2) = \sum_{\substack{j,k=1 \\ j \neq k}}^N \alpha_{j,k} \xi(A_j) \xi(A_k)\;.
\end{align*}
We then get the following identity which resembles the It\^o isometry:
\begin{align} \label{e:ItoIso}
\E \Big(& \int_{E \times E} f(z_1, z_2) \xi(dz_1) \, \xi(dz_2) \Big)^2 \notag\\
& = \sum_{j_1< k_1} \sum_{j_2 <k_2} (\alpha_{j_1,k_1} + \alpha_{k_1, j_1}) (\alpha_{j_2,k_2} + \alpha_{k_2, j_2})  \;\E(\xi(A_{j_1}) \xi(A_{k_1} ) \xi(A_{j_2}) \xi(A_{k_2} )  ) \notag\\
&= \sum_{j<k}  (\alpha_{j,k} + \alpha_{k, j})^2 \mu(A_{k}) \mu(A_j) \notag \\
& =  2 \int_{E \times E}  \big( \frac{1}{2} \big( f(z_1,z_2) + f(z_2,z_1) \big) \big)^2 \mu(dz_1) \mu(dz_2) \notag \\
&\leq 2 \int_{E \times E}  \big( f(z_1,z_2) \big)^2 \mu(dz_1) \mu(dz_2)  \;.
\end{align}
Note that we have crucially used the fact that no ``diagonal terms'' appear when passing from the second to the third line. It is relatively easy to show that the elementary functions are dense in $L^2(E\times E , \mu \otimes \mu)$ (due to the off-diagonal assumption this is only true for non-atomic measures) and hence we can extend the definition of $ \int_{E \times E} f(z_1, z_2) \xi(dz_1) \, \xi(dz_2)$ to all of $L^2(E\times E , \mu \otimes \mu)$. 

\medskip

However, the fact that we have defined the iterated integral as a limit of approximations that ``cut out'' the diagonal has an effect, when treating non-linear functions of iterated stochastic integrals.  Formally, for $f \in L^2(E, \mu)$ one might expect the identity 
\begin{align}\label{e:Wrong1}
\Big( \int_E f(z) \xi(dz) \Big)^2 =  \int_{E \times E}  f(z_1)f(z_2) \xi(dz_1) \xi(dz_2) \;,  
\end{align}
which ``follows'' by formally expanding the integral. But at this point it becomes relevant that as mentioned above $\xi$ is typically \emph{not} a random measure, so that this operation is not admissible.  In order to get the right answer, we have to approximate $f$ by simple functions $f \approx \sum_{j=1}^N \alpha_j A_j$. Mimicking the construction of the iterated integrals above we write
\begin{align*}
\Big( \int f(z) \xi(dz) \Big)^2 \approx \sum_{j \neq k}   \alpha_j \alpha_k  \xi(A_j) \xi(A_k)  + \sum_{j} \alpha_j^2 \xi(A_j)^2 \;.
\end{align*}
As expected, the first sum involving only off-diagonal entries will converge to $\int f(z_1) f(z_2) \xi(dz_1) \xi(dz_2)$ as the partition $(A_j)$ gets finer. However, differing from the case where $\xi$ is a measure, the ``diagonal'' term does not vanish in the limit. Indeed, its expectation is given by 
\begin{align*}
\sum_j \alpha_j^2 \; \E \xi(A_j)^2 = \sum_j \alpha_j^2 \; \mu(A_j) \approx \int_E f(z) \mu (dz)  \;,
\end{align*}
but the variance of this term will go to zero as the partition gets finer. 
This suggests that instead of \eqref{e:Wrong1} we should get
\begin{equation}\label{e:True}
\Big( \int_E f(z) \xi(dz) \Big)^2 =  \int_{E \times E}  f(z_1)f(z_2) \xi(dz_1) \xi(dz_2) + \int_E f(z)^2 \, dz \;,
\end{equation}
and this formula is indeed true. 
\begin{remark}
In the one-dimensional case, i.e. when $E=\R$ and $\mu$ is the Lebesgue measure our construction yields iterated \emph{It\^o} integrals. In the case where $f = \mathbf{1}_{[0,t]}$ the formula \eqref{e:True} reduces to 
\begin{align*}
 \Big(\int \mathbf{1}_{[0,t]}(s) \xi(ds) \Big)^2  =  \int_{\R \times \R}  \mathbf{1}_{[0,t]}(s_1) \mathbf{1}_{[0,t]}(s_2) \; \xi(ds_2) \; \xi(ds_1) + t
\end{align*}
which in the more common notation of stochastic calculus reduces to the It\^o formula
\begin{align*}
W_t^2 = 2 \int_0^t \big(\int_0^{s_1} dW_{s_2}\big) \;dW_{s_1} + t \;.
\end{align*}
\end{remark}
\begin{remark}
The generalisation to iterated integrals of arbitrary order follows a similar pattern. We leave it as an exercise to the reader to convince himself that  for $n=3$ formula \eqref{e:True} becomes
\begin{align}\label{e:True11}
\Big( \int_E f(z) \xi(dz) \Big)^3 =&  \int_{E \times E}  f(z_1)f(z_2)f(z_3)  \xi(dz_1) \xi(dz_2) \xi(dz_3) \notag\\
& \quad  + 3 \int_E f(z)^2 \, dz  \int_E f(z) \xi(dz) \;.
\end{align}
For larger $n$ such identities are expressed most conveniently in terms of \emph{Hermite} polynomials $H_n$. For example, one gets
\begin{align*}
H_n\big(\int f(z) \xi(dz), \| f\|_{L^2(\mu)} \big) = \int \ldots \int \prod_{j=1}^n f(z_j) \xi(dz_1) \ldots \xi(dz_n) \;,
\end{align*}
where $H_0(Z, \sigma)=1, H_1(Z,\sigma) =Z$, $H_2(Z,\sigma) = Z^2- \sigma^2$, $H_3(Z,\sigma) = Z^3 - 3\sigma^2 Z$ etc.
(see \cite[Proposition 1.1.4]{Nualart}).
\end{remark}
\medskip 
At this point almost all the tools we need in the analysis of non-linear functionals of  Gaussian processes are in place. For example, we use expression \eqref{e:True11} to decompose $\<1>_\delta^3$ into two parts, the variances of each can be evaluated explicitly by the $L^2$ isometry \eqref{e:ItoIso}. However,  if we wanted to feed these bounds directly into  the Kolmogorov Theorem \ref{Kolmogorov} we would lose too much regularity. As in the case of white noise, above we have to replace the $L^2$-type bounds by $L^p$ bounds for $p$ large enough. In the Gaussian case we used the fact that for centred Gaussian random variables all moments are controlled by the variance. Fortunately iterated stochastic integrals satisfy a similar property. This is the content of the famous \emph{Nelson estimate} which states in our context that 
\begin{align}
\E\Big( & \int f(z_1, \ldots, z_n) \xi(dz_1) \ldots \xi(dz_n) \Big)^p \notag\\
&\qquad \leq C_{n,p} \Big( \int f(z_1, \ldots, z_n)^2 \mu(dz_1) \ldots \mu(dz_n) \Big)^{\frac{p}{2}}\;. \label{e:Nelson}
\end{align}
This estimate is an immediate consequence of the hypercontractivity of the Ornstein-Uhlenbeck semigroup, see \cite[Thm 1.4.1]{Nualart}.

%

%=======================
% bibliography
%=======================
\bibliography{refs}
\bibliographystyle{abbrv}
\end{document}